\documentclass[12pt,oneside]{amsart}
\usepackage[margin=1in]{geometry}
\usepackage{enumitem,verbatim}
\usepackage{microtype}
\allowdisplaybreaks

\usepackage{amsmath,amssymb,amsthm,mathtools}
\usepackage{graphicx}
\usepackage{tikz, tikz-cd}
\usetikzlibrary{decorations.markings, arrows.meta}
\tikzset{
    partial ellipse/.style args={#1:#2:#3}{
        insert path={+ (#1:#3) arc (#1:#2:#3)}
    }
}
\usetikzlibrary{decorations.pathmorphing}
\tikzset{snake it/.style={decorate, decoration=snake}}

\usepackage[hyphens]{url}
\usepackage[pdfusetitle]{hyperref}
\usepackage[nameinlink]{cleveref}
\usepackage{color}
\hypersetup{colorlinks,linkcolor=blue,citecolor=cyan}
\usepackage{physics}
\usepackage{bbm}
\usepackage{float}

\usepackage{alphabeta}
\newcommand{\mathsfgreek}[1]{\text{\sffamily#1}}

\newtheorem{thm}{Theorem}[section]
\newtheorem{mainthm}{Theorem}[section]

\newtheorem{cor}[thm]{Corollary}
\newtheorem{lem}[thm]{Lemma}
\newtheorem{prop}[thm]{Proposition}
\newtheorem{conj}[thm]{Conjecture}

\theoremstyle{definition}
\newtheorem{defn}[thm]{Definition}
\newtheorem{eg}[thm]{Example}

\newtheorem{qn}[thm]{Question}
\newtheorem{rmk}[thm]{Remark}

\DeclareMathOperator{\Sk}{\mathrm{Sk}}

\DeclareMathOperator{\SkCat}{\mathrm{SkCat}}

\newcommand{\qbin}[2]{\begin{bmatrix}{#1}\\ {#2}\end{bmatrix}}

\DeclareMathOperator{\KH}{\mathcal{H}}

\definecolor{darkgreen}{rgb}{0,0.7,0}

\title{Flow loops and quantum groups}
\author{Sunghyuk Park}
\begin{document}

\maketitle

\begin{abstract}
This paper connects two seemingly different ways of studying knots: quantum group invariants and the dynamics of Morse flows. 
For fibered knots, we define a two-variable series invariant by counting Morse flow loops in the complement. 
This dynamical series is conjectured to agree with the BPS $q$-series of the knot complement, which arises from Verma modules for quantum groups and encodes all colored Jones polynomials. 
We prove this correspondence for all braid-homogeneous knots. 
\end{abstract}

\tableofcontents

\newpage
\section{Introduction}
\addtocontents{toc}{\protect\setcounter{tocdepth}{1}}
Dynamics is a subject with a long history, going back at least to Newtonian mechanics. 
It studies the long-term behavior of systems that evolve over time, such as the famous Lorenz system \cite{Lorenz} (Figure \ref{fig:Lorenz_attractor}), which is an example of a three-dimensional flow.  
\begin{figure}
    \centering
    \includegraphics[width=0.3\linewidth]{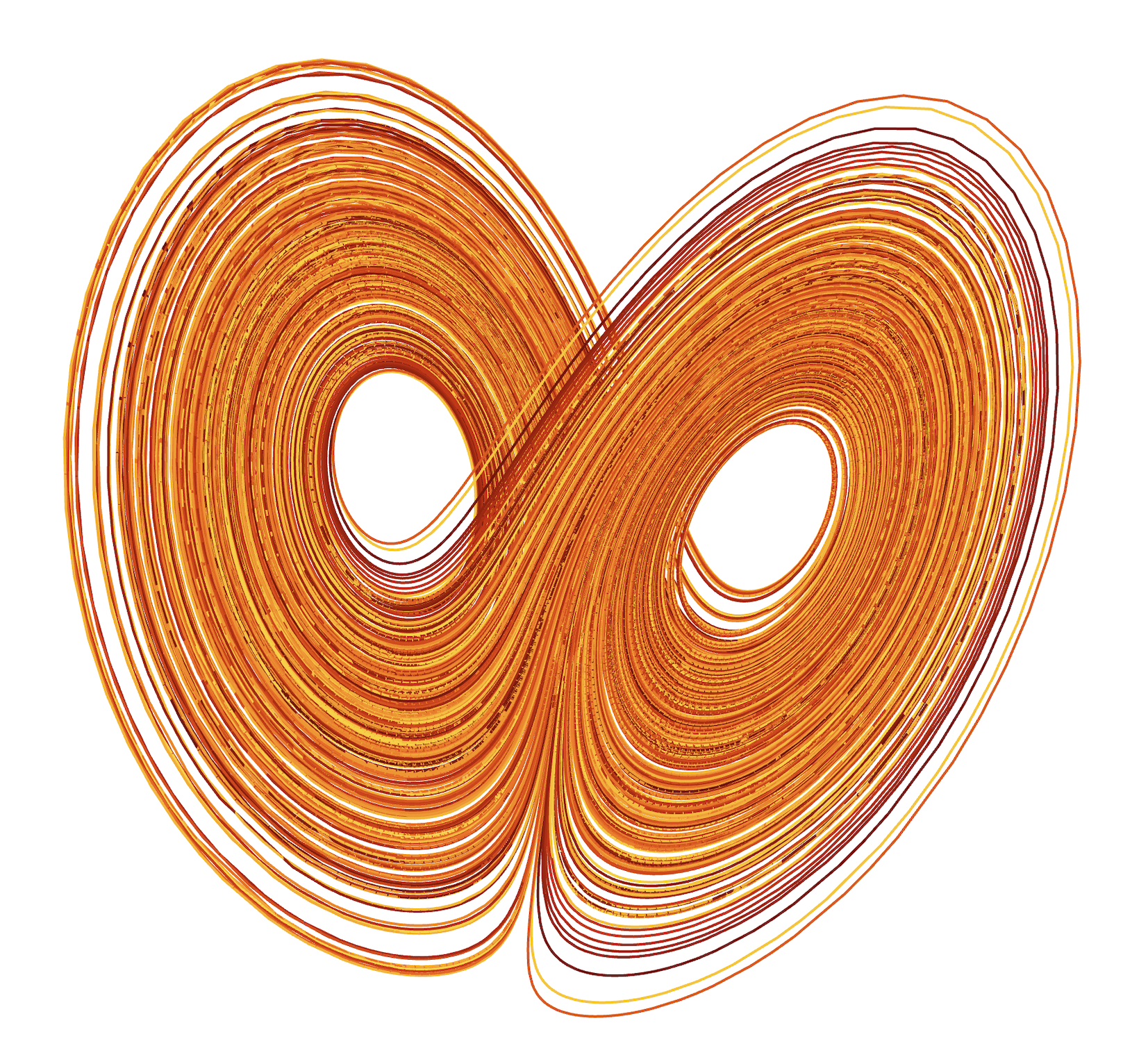}
    \caption{Lorenz system is a system of ODEs determining a three-dimensional flow. 
    Every solution to the system eventually gets attracted to the set shown, called the Lorenz attractor.}
    \label{fig:Lorenz_attractor}
\end{figure} 
The dynamics of three-dimensional flows are particularly interesting, since their periodic orbits can be knotted in interesting ways. 
Birman and Williams \cite{BirmanWilliams} showed that 
such knotted periodic orbits can be completely analyzed using certain branched surfaces known as \emph{knot holders} (a.k.a.~templates \cite{GhristHolmesSullivan}), which ``hold'' the knotted periodic orbits; the dynamics of three-dimensional flows can be encoded into the symbolic dynamics on knot holders. 

Quantum topology, in contrast, studies knots and low-dimensional manifolds through invariants that are usually constructed using representation-theoretic tools. 
It emerged in the late 1980s, largely driven by the discovery of the Jones polynomial \cite{Jones} and its interpretation and generalization in terms of representation theory of quantum groups \cite{ReshetikhinTuraev}. 
These representation-theoretic definitions of link polynomials are \emph{diagrammatic} in that one associates a polynomial to each link diagram and then checks the invariance under Reidemeister moves. 

In his lecture at the Hermann Weyl Symposium in 1987, Atiyah famously posed the following question:
\begin{qn}[\cite{Atiyah}]\label{qn:Atiyah}
Is there an intrinsically three-dimensional definition of the Jones polynomial?
\end{qn}
Witten later gave two remarkable answers to this question. 
First, in \cite{Witten_Jones}, he interpreted the Jones polynomial in terms of the 3-dimensional topological quantum field theory now known as Chern--Simons theory, where it arises as the partition function of $S^3$ in the presence of a Wilson line defect supported on the link. 
Second, in \cite{Witten_fivebranes}, he proposed a gauge-theoretic interpretation in terms of counting solutions to the Kapustin--Witten equations on $S^3\times[0,\infty)$ with boundary conditions determined by the link. 

These insights are among the most beautiful and influential in quantum topology and mathematical physics. 
At the same time, from the perspective of obtaining a direct and computable mathematical definition of the Jones polynomial, substantial difficulties remain. 
In Chern--Simons theory, the partition function is formulated in terms of path integrals, for which a direct mathematical definition is still unavailable.\footnote{There is a recent alternative approach using factorization homology; see \cite{CostelloFrancisGwilliam}.} 
On the gauge-theoretic side, no rigorous mathematical definition is currently available for the count of solutions to the Kapustin--Witten equations, even for the unknot, and substantial analytical difficulties, including compactness, remain. 
This motivates the search for another answer to Question \ref{qn:Atiyah}, one that is both intrinsically three-dimensional and computable. 

The purpose of this article is to give such a \emph{geometric} definition of quantum invariants for fibered links, using ideas from dynamics. 
This new bridge is motivated by ideas from topological string theory. 
Witten \cite{Witten_CS-string} showed that Chern--Simons theory can be embedded in topological string theory, and Ooguri and Vafa \cite{OoguriVafa} developed this picture further, predicting a relationship between quantum link invariants and open Gromov--Witten invariants. 
A rigorous mathematical formulation of this picture emerged only much later, when Ekholm and Shende \cite{EkholmShende} defined deformation-invariant counts of open holomorphic curves via skein-valued curve counting, culminating in a proof of the Ooguri--Vafa conjecture. 
This development is a crucial ingredient for the present work, but it does not by itself provide an independent way to compute quantum link invariants, since in practice the evaluation of these curve counts still relies on prior knowledge of the corresponding colored HOMFLYPT polynomials. 

In \cite{EGGKPS}, we proposed a related picture for knot complements, predicting that the corresponding open Gromov--Witten invariants recover the \emph{BPS \(q\)-series} (a.k.a.~the $\widehat{Z}$-invariant) for the knot complement \cite{GukovManolescu, Park_large-color, Park_inverted}, which in turn encodes all the colored Jones polynomials of the knot. 
For fibered links, this picture becomes especially concrete. 
In that case, the knot complement Lagrangian $M_K \subset T^*S^3$, diffeomorphic to \(S^3 \setminus K\), can be shifted off the zero section $S^3 \subset T^*S^3$, so that the skein-valued count of holomorphic curves in $(T^*S^3, S^3 \sqcup M_K)$ gives a precise definition of these open Gromov--Witten invariants. 
In the Morse flow graph limit \cite{Ekholm_Morse-flow-tree} (see also \cite[Appendix A]{EkholmLonghiParkShende}), these holomorphic curves degenerate to Morse flow graphs, which in the present setting are Morse flow loops, i.e., periodic orbits of the Morse flow \cite{ChauhanEkholmLonghi}. 
What makes this picture computable is that the relevant \(\mathrm{GL}_2\)-quantum invariants are recovered from the \(\mathrm{GL}_1\)-skein-valued count of these flow loops, which depends only on their linking numbers. 
This is where ideas from dynamics enter, through tools such as knot holders.

\subsection*{Summary of main results}

We first define an invariant of fibered knots coming from the $\mathrm{GL}_1$-skein-valued flow loop counts:
\begin{mainthm}\label{mainthm:flow-loop-count}
For any fibered knot $K$, the $\mathrm{GL}_1$-skein-valued flow loop count
\[
\Phi_{S^3 \setminus K} (x,q) \in \mathbb{Z}[q^{\pm 1}][[x]]
\]
is a well-defined knot invariant. 
\end{mainthm}
\begin{rmk}\label{rmk:fibered_link}
While we mainly work with fibered knots for simplicity, this invariant can be straightforwardly extended to fibered links, and even fibered $3$-manifolds. 
If $K$ is a fibered link with $s$ components, then the $\mathrm{GL}_1$-skein-valued flow loop count takes a value in the completion of $\mathbb{Z}[q^{\pm 1}][x_1^{\pm 1}, \cdots, x_s^{\pm 1}]$ with respect to the total $x$-degree, where $x_1, \cdots, x_s$ represent the homology classes in $H_1(S^3 \setminus K)$ corresponding to the positive meridians. 
See Appendix~\ref{sec:fibered_cone} for further discussions. 
\end{rmk}

While this invariant is essentially the $\mathrm{GL}_1$-reduction of the skein-valued flow loop count considered in \cite{ChauhanEkholmLonghi}, we provide an elementary proof of deformation invariance (independent of the skein-valued curve count of \cite{EkholmShende}) via bifurcation analysis. 

This $\mathrm{GL}_1$-skein-valued flow loop count is a mathematically precise version of the open Gromov-Witten invariants considered in \cite{EGGKPS}. 
In particular, the conjecture of \cite{EGGKPS} can be reformulated precisely as follows:
\begin{conj}\label{conj:flow-equals-quantum}
For any fibered link $K$, 
\[
\Phi_{S^3 \setminus K} \doteq \widehat{Z}_{S^3 \setminus K},
\]
where $\doteq$ denotes equality up to an overall monomial factor of the form $\pm q^{\cdots} x^{\cdots}$.\footnote{Equivalently, since the left-hand side always starts with $1 + \cdots$, we may normalize the right-hand side in the same way by factoring out the leading monomial; see the end of Section \ref{sec:BPS-q-series-from-flow-loops} for the precise monomial. } 
\end{conj}
Here, $\widehat{Z}_{S^3 \setminus K}$ on the right-hand side denotes the BPS $q$-series of $S^3 \setminus K$. 
The mathematical definition of the BPS $q$-series is currently available for all braid-homogeneous links \cite{Park_inverted}, or more generally, for all links admitting ``nice'' diagrams \cite{Park_thesis}.\footnote{While it is not known whether all fibered links are nice, all fibered knots up to 12 crossings are known to be nice \cite{OSSS}.}
In the case of fibered links for which we do not currently have an independent definition of the right-hand side, the conjecture should be understood as saying that the flow loop count provides a resummation of the Melvin-Morton-Rozansky (MMR) expansion \cite{MelvinMorton, BarNatanGaroufalidis, Rozansky, Rozansky_RC} of the colored Jones polynomials of $K$, in the sense of \cite[Conj. 1.5]{GukovManolescu}. 
In particular, Conjecture \ref{conj:flow-equals-quantum} would imply the following finiteness conjecture for MMR expansions of fibered knots: 
\begin{conj}[\cite{Park_inverted}]
For any fibered knot $K$, the coefficients of its MMR expansion
\[
\frac{1}{[n]} J_{K, n}(q) \bigg\vert_{q=e^{\hbar}} 
\quad
\underset{x:= e^u}{
\overset{\substack{\hbar \rightarrow 0,\;n \rightarrow \infty \\ \text{while } u := n\hbar \text{ fixed}}}{\sim}
}
\quad
\frac{1}{\Delta_K(x)} + \sum_{d\geq 1} \frac{P_d(x)}{\Delta_K(x)^{2d+1}} \hbar^d
,
\quad
P_d(x) \in \mathbb{Q}[x^{\pm 1}],
\]
when expanded into a power series in $x$, which are a priori just in $\mathbb{Q}[[\hbar]]$, can be resummed into uniquely determined Laurent polynomials in $\mathbb{Z}[q^{\pm 1}]$. 
\end{conj}

Our main theorem settles Conjecture \ref{conj:flow-equals-quantum} for all braid-homogeneous knots:
\begin{mainthm}\label{mainthm:flow-equals-quantum}
For any braid-homogeneous knot, Conjecture \ref{conj:flow-equals-quantum} is true. 
\end{mainthm}

Since $\widehat{Z}_{S^3 \setminus K}$ encodes all the colored Jones polynomials of $K$ via MMR expansion, it immediately implies that the flow loop count $\Phi_{S^3 \setminus K}$ encodes all the colored Jones polynomials of $K$, hence giving an intrinsically three-dimensional geometric definition of these quantum link invariants. 

While $\Phi_{S^3 \setminus K}$ counts flow loops in the complement $S^3 \setminus K$, the BPS $q$-series $\widehat{Z}_{S^3 \setminus K}$ is usually defined in terms of a state sum on a link diagram of $K$, using Verma modules of the quantum group $U_q(\mathfrak{sl}_2)$ \cite{Park_large-color, Park_inverted}. 
As an intermediary step toward proving Theorem \ref{mainthm:flow-equals-quantum}, we show that braid group representations typically constructed using quantum groups can be recovered from these flow loop counts: 
\begin{mainthm}\label{mainthm:Lawrence}
The braid group representation on tensor powers of Verma modules can be obtained by counting flow loops. 
\end{mainthm}

\subsection*{Organization of the paper}
In Section \ref{sec:skein-valued-flow-loop-count}, we define the skein-valued count of flow loops in fibered 3-manifolds and prove Theorem \ref{mainthm:flow-loop-count}. 
In Section \ref{sec:knot-holder-state-sum}, we explain how to actually count such flow loops in practice, using knot holders; 
this leads to a geometric state sum model on knot holders. 
In Section \ref{sec:braid-group-repn}, we describe the braid group representation induced from these flow loop counts. 
We then prove Theorem \ref{mainthm:Lawrence}, showing that this braid group representation is equivalent to the one coming from tensor products of Verma modules. 
Finally, in Section \ref{sec:BPS-q-series-from-flow-loops}, we prove our main theorem (Theorem \ref{mainthm:flow-equals-quantum}) showing that the flow loop count is equal to the BPS $q$-series, for all braid-homogeneous knots. 
Appendix~\ref{sec:fibered_cone} discusses how the skein-valued flow loop count gives an invariant of fibered cones, which may be of independent interest.

\subsection*{Future directions}
We conclude the introduction by outlining some directions of research we plan to pursue in future work. 
\begin{itemize}
\item The BPS $q$-series $\widehat{Z}$, as originally predicted from physics \cite{GukovPutrovVafa, GukovPeiPutrovVafa, GukovManolescu}, is expected to admit a categorification. 
The geometric definition of the BPS $q$-series in terms of flow loop count provides a promising approach to this problem: it suggests that there should be a Floer-type theory whose generators are these flow loops. 
It would be very interesting to construct such a homology theory. 

\item We can also add extra Wilson line defects along some link $L\subset S^3 \setminus K$, as in Theorem \ref{thm:abelianization}. 
The resulting two-variable series $\widehat{Z}_{S^3 \setminus K,\, L} \in \mathbb{Z}[q^{\pm 1}]((x))$ then counts not only flow loops but also flow lines starting and ending on $L$. 
This is the analog of the Jones polynomial of $L$, in $S^3 \setminus K$ instead of $S^3$. 
A categorification of such a two-variable series would be the analog of Khovanov homology \cite{Khovanov} for links in $S^3 \setminus K$. 
When $K$ is the unknot, this should specialize to the annular Khovanov homology (a.k.a.~annular APS homology) \cite{AsaedaPrzytyckiSikora}. 

\item Extending down, there should be a dg- (or $A_\infty$-) category associated to marked fiber surfaces\footnote{Categorifying the object $Z(F)$ of the skein category discussed in Section \ref{subsec:skein-category}.}, so that the categorical trace (i.e., Hochschild homology) of the monodromy autofunctor should recover the conjectural homology theory in the first bullet point above. 
It would be interesting to see if there is any connection to the one proposed in \cite{ColinHondaTian}. 

\item It would also be interesting to extend the results of this paper to non-fibered knots. 
In case of non-fibered knots, the circle-valued Morse function will necessarily have some critical points, so we need to count not just flow loops but also flow lines between those critical points. 
In terms of Lagrangians, this would mean shifting the knot complement Lagrangian $M_K$ so that it intersects the zero section $S^3 \subset T^*S^3$ transversely at finitely many points; see \cite[Sec. 8.4]{DiogoEkholm}. 
The count of holomorphic curves in $(T^*S^3, S^3 \cup M_K)$ should then take a value in the ``skein module of the intersecting Lagrangians $S^3 \cup M_K$.'' 
Understanding such a skein module and identifying it with an appropriate two-variable power series ring would be the key steps in this direction. 

\end{itemize}

\subsection*{Acknowledgements}
I am indebted to Tobias Ekholm and Sergei Gukov for numerous inspiring conversations and discussions over the years. 

This work was supported in part by Simons Foundation through Simons Collaboration on Global Categorical Symmetries.

\addtocontents{toc}{\protect\setcounter{tocdepth}{2}}
\section{Skein-valued count of flow loops}\label{sec:skein-valued-flow-loop-count}

In this section, we define the main object of this paper, the $\mathrm{GL}_1$-skein-valued count of flow loops. 
This section is highly inspired by a recent work \cite{ChauhanEkholmLonghi}, as our count is nothing but the $\mathrm{GL}_1$-reduction of the count considered in that paper. 

\subsection{General setup}\label{subsec:setup}
Let $K$ be a fibered knot, so that its complement admits a fibration $\pi : S^3 \setminus K \rightarrow S^1$, which is unique up to isotopy. 
Such a fibration can be viewed as a circle-valued Morse function on the knot complement. 
Accordingly, there is a nowhere-vanishing closed $1$-form $\pi^*(d\theta)$ on $S^3 \setminus K$. 
Once a metric on $S^3 \setminus K$ is chosen, the closed $1$-form can be dualized to a Morse flow vector field $\xi_{\mathrm{Mor}}$, generating a Morse flow. 
We assume that the knot $K$ is a sink of the Morse flow, by modifying the circle-valued Morse function near the boundary of $S^3 \setminus K$ if necessary. 
We also assume that all the periodic orbits of the Morse flow are nondegenerate (i.e., $1$ is not an eigenvalue of the linearized return maps), by making a generic perturbation of the metric if necessary. 

We are interested in counting periodic orbits of the Morse flow.\footnote{Equivalently, we are counting the periodic orbits of the suspension flow of the Poincar\'e return map of the fiber. 
Conversely, any suspension flow of a surface diffeomorphism arises as a Morse flow, if we choose $\pi$ as the circle-valued Morse function and choose a metric on the mapping torus which is fiberwise a product metric.}  
A compactness argument shows that, for each period (i.e., a positive class in $H_1(S^3 \setminus K) \cong \mathbb{Z}$), there are only finitely many periodic orbits of that period. 
Let $\gamma$ be a nondegenerate periodic orbit, and let $\phi_\gamma$ be the corresponding linearized return map, and $\lambda_1, \lambda_2$ its eigenvalues. 
Then, there are three possibilities:\footnote{Since $\det(I - \phi_\gamma) + \det(I + \phi_\gamma) = 2(1+\det \phi_\gamma) > 0$, $\det(I-\phi_\gamma)$ and $\det(I+\phi_\gamma)$ can't be both negative.} 
\begin{itemize}
\item (positive hyperbolic) $\det(I-\phi_\gamma) < 0$, or equivalently, $0 < \lambda_1 < 1 < \lambda_2$. 
In this case, the normal bundle of $\gamma$ splits into stable and unstable directions, both of which are trivial real line bundles over $\gamma$. 
\item (negative hyperbolic) $\det(I + \phi_\gamma) < 0$, or equivalently, $\lambda_1 < -1 < \lambda_2 < 0$. 
In this case, the normal bundle of $\gamma$ splits into stable and unstable directions, both of which are non-trivial real line bundles over $\gamma$, i.e., M\"obius bands. 
\item (elliptic) $\det(I-\phi_\gamma) > 0$ and $\det(I + \phi_\gamma) > 0$. 
This includes complex conjugate eigenvalues, sinks, and sources. 
\end{itemize}
We say a periodic orbit is \emph{primitive} if it is not a multiple cover of another orbit of smaller period. 
Let $\mathcal{O}_{\mathrm{prim}}$ denote the set of primitive flow loops. 

Consider the following count of flow loops:
\begin{equation}\label{eqn:flow-zeta}
\zeta_{S^3 \setminus K} := \prod_{\gamma \in \mathcal{O}_{\textrm{prim}}} w_\gamma \in \mathbb{Z}[[x]],
\end{equation}
where $x$ is the positive generator of $H_1(S^3 \setminus K)$, 
\[
w_\gamma := 
\begin{cases}
1-[\gamma] &\text{if }\gamma \text{ is elliptic}, \\
\frac{1}{1-[\gamma]} = 1 + [\gamma] + [\gamma]^2 + \cdots &\text{if }\gamma \text{ is positive hyperbolic}, \\
\frac{1}{1+[\gamma]} = 1 - [\gamma] + [\gamma]^2 - \cdots &\text{if }\gamma \text{ is negative hyperbolic},
\end{cases}
\]
and $[\gamma] \in H_1(S^3 \setminus K) \cong \mathbb{Z}$ is the corresponding homology class, so that $[\gamma] = x^{\deg \gamma}$ if $\deg \gamma$ denotes the period of $\gamma$. 
It follows from a well-known result of Hutchings--Lee \cite{HutchingsLee} that this count agrees with the inverse\footnote{Inverse because our weight $w_\gamma$ is the inverse of the one used in \cite{HutchingsLee}.} Reidemeister torsion:
\[
\zeta_{S^3 \setminus K} = \frac{1-x}{\Delta_K(x)},
\]
where $\Delta_K(x)$ is the Alexander polynomial\footnote{Normalized here in such a way that the power series expansion of the right-hand side starts with $1 + \cdots$.} of $K$, giving a dynamical definition of the Alexander polynomial. 
By simply expanding the product \eqref{eqn:flow-zeta}, we also get the following expression:
\begin{equation}\label{eqn:flow-zeta-expanded}
\zeta_{S^3 \setminus K} = \sum_{\gamma \in \mathcal{O}} (-1)^{\mathrm{e}(\gamma) + \mathrm{nh}(\gamma)} [\gamma],
\end{equation}
where $\mathcal{O}$ denotes the set of all multi-flow loops (i.e., with possibly multiple components) where the hyperbolic components are allowed to be multiply covered, while the elliptic components cannot be multiply covered, 
and $\mathrm{e}(\gamma)$ (resp., $\mathrm{nh}(\gamma)$) denotes the total number of elliptic (resp., negative hyperbolic) flow loops in $\gamma$.

The $\mathrm{GL}_1$-skein-valued count of flow loops $\Phi_{S^3 \setminus K}$ that we define below can be thought of as a $q$-deformation of this count $\zeta_{S^3 \setminus K}$. 
The basic idea is to count these flow loops in the $\mathrm{GL}_1$-skein module, by not only remembering their homology class but also how they link with each other.

\subsection{Framing}\label{subsec:framing}

In order to introduce framing on the flow loops, we choose a framing vector field $\xi_{\mathrm{fr}}$ (a.k.a.~4-chain vector field) similarly to \cite[Sec. 9.3]{EkholmLonghiParkShende}, 
\[
\xi_{\mathrm{fr}} = \xi_{\mathrm{Mor}} + \epsilon \xi_{\mathrm{pert}},
\]
by perturbing the Morse flow vector field with a perturbation vector field $\xi_{\mathrm{pert}}$, which is tangent to each fiber surface and has simple zeros, for which the boundary of the fiber surface is a sink. 
The locus of the zeros of $\xi_{\mathrm{pert}}$ is a closed $1$-chain $\ell_{\mathrm{fr}} \subset S^3 \setminus K$, called the \emph{framing line defect}, oriented in the direction of the Morse flow, with each component of $\ell_{\mathrm{fr}}$ carrying a sign $\in \{\pm 1\}$ corresponding to the index of the zero of $\xi_{\mathrm{pert}}$ in each fiber surface. 
We call framing lines with sign $+1$ (resp., $-1$) elliptic (resp., hyperbolic). 

As long as a flow loop $\gamma \in \mathcal{O}$ does not intersect $\ell_{\mathrm{fr}}$, the vector field $\xi_{\mathrm{fr}}$ is everywhere transverse to the tangent vectors of $\gamma$, so we can use it to give a framing on the oriented link $\gamma$. 
Note, since the tangent vectors of $\gamma$ are parallel to $\xi_{\mathrm{Mor}}$ in this setup, this is equivalent to the framing on $\gamma$ induced by $\xi_{\mathrm{pert}}$. 
In case a flow loop $\gamma \in \mathcal{O}$ intersects $\ell_{\mathrm{fr}}$, we may arbitrarily homotope $\gamma$ off the framing line defect to induce a framing on $\gamma$;
the relation \eqref{eq:framing-line} below ensures that the result is independent in the $\mathrm{GL}_1$-skein module with framing lines.

\subsection{Skein modules}\label{subsec:skein-modules}

Our count of Morse flow loops will take a value in a completion of the $\mathrm{GL}_1$-skein module of $S^3 \setminus K$, twisted by framing lines along $\ell_{\mathrm{fr}}$, whose definition is given as follows. 
\begin{defn}\label{defn:GL1-skein}
Let $Y$ be an oriented $3$-manifold, and let $\ell \subset Y$ be an oriented link each of whose component carries a sign. 
The \emph{$\mathrm{GL}_1$-skein module $\Sk^{\mathrm{GL}_1}_q(Y, \ell)$ of $Y$ with framing lines along $\ell$} is the $\mathbb{Z}[q^{\pm 1}]$-module generated by framed, oriented links in $Y \setminus \ell$, modulo the following local relations (all shown in blackboard framing):
% \[
% \Sk^{\mathrm{GL}_1}_q(Y, \gamma) := \frac{R\langle \text{isotopy classes of framed oriented links in }Y\setminus \gamma\rangle}{\langle \mathrm{GL}_1\text{-skein relations}\rangle},
% \]
% where the $\mathrm{GL}_1$-skein relations are given by
\begin{gather}
q^{-1}\;\vcenter{\hbox{
\begin{tikzpicture}[scale=0.7]
\draw[dotted] (0,0) circle (1);
\draw[ultra thick, ->] ({sqrt(2)/2},{-sqrt(2)/2}) -- ({-sqrt(2)/2},{sqrt(2)/2});
\draw[white, line width=2.5mm] ({-sqrt(2)/2},{-sqrt(2)/2}) -- ({sqrt(2)/2},{sqrt(2)/2});
\draw[ultra thick, ->] ({-sqrt(2)/2},{-sqrt(2)/2}) -- ({sqrt(2)/2},{sqrt(2)/2});
\end{tikzpicture}
}}
\;\;=\;\;
\vcenter{\hbox{
\begin{tikzpicture}[scale=0.7]
\draw[dotted] (0,0) circle (1);
\draw[ultra thick, <-] ({sqrt(2)/2},{sqrt(2)/2}) arc (135:225:1);
\draw[ultra thick, ->] ({-sqrt(2)/2},{-sqrt(2)/2}) arc (-45:45:1);
\end{tikzpicture}
}}
\;\;=\;\;
q\;
\vcenter{\hbox{
\begin{tikzpicture}[scale=0.7]
\draw[dotted] (0,0) circle (1);
\draw[ultra thick, ->] ({-sqrt(2)/2},{-sqrt(2)/2}) -- ({sqrt(2)/2},{sqrt(2)/2});
\draw[white, line width=2.5mm] ({sqrt(2)/2},{-sqrt(2)/2}) -- ({-sqrt(2)/2},{sqrt(2)/2});
\draw[ultra thick, ->] ({sqrt(2)/2},{-sqrt(2)/2}) -- ({-sqrt(2)/2},{sqrt(2)/2});
\end{tikzpicture}
}}
\;, \label{eq:gl1skeinrel1}
\\
\vcenter{\hbox{
\begin{tikzpicture}[scale=0.7]
\draw[dotted] (0,0) circle (1);
\draw[ultra thick, ->] (0.5,0) arc (0:370:0.5);
\end{tikzpicture}
}}
\;\;=\;\;
\vcenter{\hbox{
\begin{tikzpicture}[scale=0.7]
\draw[dotted] (0,0) circle (1);
\end{tikzpicture}
}}
\;, \label{eq:gl1skeinrel2}
\\
\vcenter{\hbox{
\begin{tikzpicture}[scale=0.7]
\draw[dotted] (0,0) circle (1);
\draw[ultra thick, green, <-] (-1, 0) -- (1, 0);
\draw[white, line width=2.5mm] (0, -1) -- (0, 1);
\draw[ultra thick, ->] (0, -1) -- (0, 1);
\node[green, right] at (1, 0){$\ell$};
\end{tikzpicture}
}}
\;\;=\;\;
q^{\mp}\;
\vcenter{\hbox{
\begin{tikzpicture}[scale=0.7]
\draw[dotted] (0,0) circle (1);
\draw[ultra thick, ->] (0, -1) -- (0, 1);
\draw[white, line width=2.5mm] (-1, 0) -- (1, 0);
\draw[ultra thick, green, <-] (-1, 0) -- (1, 0);
\node[green, right] at (1, 0){$\ell$};
\end{tikzpicture}
}}
\;, \label{eq:framing-line}
\end{gather}
where the sign $\mp$ in the last relation is the opposite of the sign $\pm$ of the component of $\ell$ involved. 
\end{defn}

\begin{rmk}\label{rmk:framing-line-canonical-identification}
The dependence of $\Sk^{\mathrm{GL}_1}_q(Y, \ell)$ on the framing line $\ell$ is only through its homology class $[\ell] \in H_1(Y)$, in that any 2-chain $\eta$ with boundary $\partial \eta = \ell - \ell'$ induces a natural isomorphism $\Sk^{\mathrm{GL}_1}_q(Y, \ell) \cong \Sk^{\mathrm{GL}_1}_q(Y, \ell')$; 
it sends each framed, oriented link $[L] \in \Sk^{\mathrm{GL}_1}_q(Y, \ell)$ to $q^{L\cdot \eta}[L] \in \Sk^{\mathrm{GL}_1}_q(Y, \ell')$. 
\end{rmk}

\begin{rmk}\label{rmk:half-twist}
It is often convenient to allow half-twists (i.e., half-integer framings) on the links generating the skein module. 
This can be done simply by extending the base ring to $\mathbb{Z}[q^{\pm \frac12}]$ and introducing the following extra relation:
\begin{align}
\vcenter{\hbox{
\begin{tikzpicture}[scale=0.7]
\fill[top color=black, bottom color=white] (-0.085, .85) to[out=-90, in=100] (0, 0) to[out=80, in=-90] (0.085, .85)--cycle;
\fill[bottom color=black, top color=white] (-0.085, -1) to[out=90, in=-100] (0, 0) to[out=-80, in=90] (0.085, -1)--cycle;
\draw[line width=1] (-0.075, -1) to[out=90, in=-90] (0.075, .85);
\draw[line width=3,-{Classical TikZ Rightarrow[width=4mm,length=2mm]}](0,.85)--(0,1);
\draw[dotted] (0,0) circle (1);
\end{tikzpicture}
}}
\;\;=\;\;
q^{\frac{1}{2}}\;
\vcenter{\hbox{
\begin{tikzpicture}[scale=0.7]
\draw[dotted] (0,0) circle (1);
\draw[line width=3,-{Classical TikZ Rightarrow[width=4mm,length=2mm]}] (0, -1) -- (0, 1);
\end{tikzpicture}
}}
\;. \label{eq:half-twist}
\end{align}
\end{rmk}

\begin{defn}
If, in addition, $Y$ is equipped with a closed $1$-form $\zeta$, there is a decreasing \emph{action filtration} $\{\Sk^{\mathrm{GL}_1}_q(Y,\ell)_{> E}\}_{E \in \mathbb{R}}$ on the skein module, where $\Sk^{\mathrm{GL}_1}_q(Y,\ell)_{> E}$ is the submodule of $\Sk^{\mathrm{GL}_1}_q(Y,\ell)$ spanned by links $\kappa$ for which $\int_{[\kappa]} \zeta > E$. 
Define the \emph{action completion} $\widehat{\Sk}^{\mathrm{GL}_1}_q(Y,\ell)$ to be the Novikov-type limit with respect to this filtration:
\[
\widehat{\Sk}^{\mathrm{GL}_1}_q(Y,\ell) = \underset{E \rightarrow \infty}{\varprojlim} \Sk^{\mathrm{GL}_1}_q(Y,\ell)/\Sk^{\mathrm{GL}_1}_q(Y,\ell)_{> E}. 
\]
\end{defn}

\subsection{Skein-valued flow loop count}\label{subsec:skein-valued_flow_loop_count}

\begin{defn}\label{defn:skein-valued-flow-loop-count}
In the setting as in Sections \ref{subsec:setup}-\ref{subsec:framing} above, define the \emph{$\mathrm{GL}_1$-skein-valued flow loop count} to be
\[
\Phi_{S^3 \setminus K} 
:= \sum_{\gamma \in \mathcal{O}} (-1)^{\mathrm{e}(\gamma) + \mathrm{nh}(\gamma)} [\gamma]
\;\in \widehat{\Sk}^{\mathrm{GL}_1}_q(S^3 \setminus K, \ell_{\mathrm{fr}})
,
\]
where $[\gamma]$ denotes the class of the corresponding framed, oriented link in the skein module $\Sk^{\mathrm{GL}_1}_q(S^3 \setminus K, \ell_{\mathrm{fr}})$. 
\end{defn}

While much of the discussion in this section generalizes straightforwardly to any fibered $3$-manifold (see Appendix~\ref{sec:fibered_cone}), what is particularly nice about fibered knot complements is that we can canonically identify the skein module with the ring of two-variable series:
\begin{prop}\label{prop:GL_1-skein_series_identification}
There is a canonical isomorphism 
\[
\widehat{\Sk}^{\mathrm{GL}_1}_q(S^3 \setminus K, \ell_{\mathrm{fr}}) \cong \mathbb{Z}[q^{\pm 1}]((x)), 
\]
which sends the positive meridian of $K$ to $x$. 
\end{prop}
\begin{proof}
Using the $\mathrm{GL}_1$-skein relations given in Definition \ref{defn:GL1-skein}, it is easy to see that any framed oriented link $L$ can be turned into some parallel copies of the $0$-framed meridian times a power of $q$. 
Moreover, the power of $q$ is uniquely determined; it is given by the linking number
\[
\mathrm{lk}(L,L) - \mathrm{lk}(L, \ell_{\mathrm{fr}}),
\]
where $\mathrm{lk}(L, \ell_{\mathrm{fr}}) = \mathrm{lk}(L, \ell_{\mathrm{fr}}^+) - \mathrm{lk}(L, \ell_{\mathrm{fr}}^-)$ if $\ell_{\mathrm{fr}}^{\pm}$ denotes the $\pm$-signed components of $\ell_{\mathrm{fr}}$. 
This proves the isomorphism. 
\end{proof}

With this isomorphism, we can turn the $\mathrm{GL}_1$-skein-valued flow loop count defined in Definition \ref{defn:skein-valued-flow-loop-count} into an explicit two-variable series
\[
\Phi_{S^3 \setminus K} (x,q) \in \mathbb{Z}[q^{\pm 1}][[x]].
\]
This is the form used in the statement of Theorem \ref{mainthm:flow-loop-count}. 

\begin{proof}[Proof of Theorem~\ref{mainthm:flow-loop-count}]
Any two choices of auxiliary data (i.e., a circle-valued Morse function without any critical points, a metric, and a framing vector field) can be joined by a generic $1$-parameter family. 

A deformation of the framing vector field does not change the set of periodic orbits and only deforms the framing line $\ell_{\mathrm{fr}}$. 
The $2$-chain swept out by the framing line induces the canonical identification of the skein modules as in Remark~\ref{rmk:framing-line-canonical-identification}, and it follows from the framing line skein relation \eqref{eq:framing-line} that the flow loop count is preserved under this identification. 
Hence, it remains to consider deformations of the Morse data. 
% In a generic $1$-parameter family of a circle-valued Morse function and a metric, there are five types of bifurcations that may occur, as listed in \cite[Sec. 1.7]{Hutchings}. 
% However, since we are only interested in circle-valued Morse functions without any critical points (and $1$-parameter families through such Morse functions), the only bifurcation that may occur is a degenerate closed orbit, which includes two possibilities, 

As analyzed by Brunovsky \cite{Brunovsky}, the periodic orbits with period bounded above by some constant (i.e., some action bound) vary by isotopy away from some isolated parameters where bifurcation occurs. 
At isolated instances, there are two possibilities for the bifurcation:\footnote{A new periodic orbit of period $kp$ can branch from a periodic orbit $\gamma$ of period $p$ only if the linearized return map of $\gamma$ has a $k$-th root of unity as an eigenvalue, but non-real roots of unity are codimension-2 conditions, so cannot be seen in a generic 1-parameter family. 
The first case corresponds to the eigenvalue $+1$, and the second case corresponds to $-1$. See \cite[Thms.~1-4]{Brunovsky}. }
\begin{enumerate}
\item\label{itm:saddle-node} simple creation/annihilation of closed orbits,
\item\label{itm:period-doubling} period-doubling bifurcation.
\end{enumerate}
Thus, in order to show the invariance of the $\mathrm{GL}_1$-skein-valued flow loop count, it suffices to prove its invariance under these two types of bifurcation. 

Near an isolated bifurcation parameter, we may choose a sufficiently small neighborhood of the bifurcating orbit disjoint from all other relevant periodic orbits and from the framing line. 
Then, the linking numbers with all the relevant flow loops outside this neighborhood, as well as with the framing line, remain the same before and after the bifurcation. 
Consequently, it suffices to prove the bifurcation invariance of the skein-valued flow loop count within this neighborhood, and this follows from simple $q$-series identities, as we show below. 

\begin{enumerate}
\item \textbf{Simple creation/annihilation (a.k.a.~saddle-node bifurcation):} 
Under this bifurcation, a positive hyperbolic flow loop and an elliptic flow loop may be created or annihilated in pairs. 
Before turning on $q$, this corresponds to the identity
\[
1 = \frac{1-x}{1-x}.
\]
After turning on $q$, the invariance follows from the following $q$-series identity:\footnote{To see why this holds, simply split the right-hand side into the parts with $n_e = 0$ and $n_e = 1$ and see how the terms telescope.}
\[
1 = \sum_{\substack{n_{h_+} \geq 0 \\ 0\leq n_e\leq 1}} q^{f(n_{h_+} + n_e)^2} x^{n_{h_+}} (-x)^{n_e},
\]
where $f$ is the framing of the flow loops and $x$ is the homology class. 

\item \textbf{Period-doubling bifurcation (a.k.a. flip bifurcation):} 
Under this bifurcation, a negative hyperbolic flow loop may turn into an elliptic flow loop, together with a positive hyperbolic flow loop of twice the period.\footnote{There is also a version for $1-x = \frac{1-x^2}{1+x}$ where an elliptic flow loop turns into a negative hyperbolic flow loop, together with an elliptic flow loop of twice the period, but it is analogous.} 
In terms of knot holders, this can be visualized as splitting of a M\"obius band, while creating an elliptic flow loop:
\[
\vcenter{\hbox{
\begin{tikzpicture}
\draw (0.35, 0) to[out=90, in=-90] (-0.35, 2);
\draw[white, line width=5] (-0.35, 0) to[out=90, in=-90] (0.35, 2);
\draw (-0.35, 0) to[out=90, in=-90] (0.35, 2);
\draw (0.35, 2) to[out=90, in=180] (0.7, 2.35) to[out=0, in=90] (1.05, 2) -- (1.05, 0) to[out=-90, in=0] (0.7, -0.35) to[out=180, in=-90] (0.35, 0);
\draw (-0.35, 2) to[out=90, in=180] (0.7, 3.05) to[out=0, in=90] (1.75, 2) -- (1.75, 0) to[out=-90, in=0] (0.7, -1.05) to[out=180, in=-90] (-0.35, 0);

\draw[->, thick] (0, 0) -- (0, 0.4);
\draw[->, thick] (0, 1.6) -- (0, 2);
\end{tikzpicture}
}}
\qquad
\leftrightarrow
\qquad
\vcenter{\hbox{
\begin{tikzpicture}

\draw (0.35, 0) to[out=90, in=-90] (-0.35, 2);
\draw[white, line width=3] (0.1, 0) to[out=90, in=-90] (-0.1, 2);
\draw (0.1, 0) to[out=90, in=-90] (-0.1, 2);

\draw[red] (0, 0) -- (0, 2) to[out=90, in=180] (0.7, 2.7) to[out=0, in=90] (1.4, 2) -- (1.4, 0) to[out=-90, in=0] (0.7, -0.7) to[out=180, in=-90] cycle;

\draw[white, line width=4] (-0.1, 0) to[out=90, in=-90] (0.1, 2);
\draw (-0.1, 0) to[out=90, in=-90] (0.1, 2);
\draw[white, line width=5] (-0.35, 0) to[out=90, in=-90] (0.35, 2);
\draw (-0.35, 0) to[out=90, in=-90] (0.35, 2);

\draw (0.35, 2) to[out=90, in=180] (0.7, 2.35) to[out=0, in=90] (1.05, 2) -- (1.05, 0) to[out=-90, in=0] (0.7, -0.35) to[out=180, in=-90] (0.35, 0);
\draw (-0.35, 2) to[out=90, in=180] (0.7, 3.05) to[out=0, in=90] (1.75, 2) -- (1.75, 0) to[out=-90, in=0] (0.7, -1.05) to[out=180, in=-90] (-0.35, 0);

\draw (0.1, 2) to[out=90, in=180] (0.7, 2.6) to[out=0, in=90] (1.3, 2) -- (1.3, 0) to[out=-90, in=0] (0.7, -0.6) to[out=180, in=-90] (0.1, 0);
\draw (-0.1, 2) to[out=90, in=180] (0.7, 2.8) to[out=0, in=90] (1.5, 2) -- (1.5, 0) to[out=-90, in=0] (0.7, -0.8) to[out=180, in=-90] (-0.1, 0);

\draw[->, thick] (-0.225, 0) -- (-0.225, 0.3);
\draw[->, thick] (0.225, 0) -- (0.225, 0.3);

\draw[->, thick] (-0.225, 1.7) -- (-0.225, 2);
\draw[->, thick] (0.225, 1.7) -- (0.225, 2);
\end{tikzpicture}
}}
\]
Before turning on $q$, this corresponds to the identity
\[
\frac{1}{1+x} = \frac{1-x}{1-x^2}.
\]
After turning on $q$, the invariance follows from the following $q$-series identity:\footnote{To see why this holds, simply make the change of variable $n_{h_-} = 2n_{h_+} + n_e$.} 
\[
\sum_{n_{h_-} \geq 0} q^{f n_{h_-}^2} (-x)^{n_{h_-}}
=
\sum_{\substack{n_{h_+}\geq 0 \\ 0\leq n_e \leq 1}} q^{f(2n_{h_+} + n_e)^2} (x^2)^{n_{h_+}} (-x)^{n_e}
,
\]
where $f$ is the (half-integer) framing and $x$ is the homology class. 
\end{enumerate}
\end{proof}

\begin{rmk}\label{rmk:connection-to-curve-count}
As described briefly in the introduction, the count in Definition \ref{defn:skein-valued-flow-loop-count} is (the $\mathrm{GL}_1$-reduction of) the skein-valued curve count in $(T^*S^3, S^3 \sqcup M_K)$, where $M_K$ denotes the knot complement Lagrangian, obtained from a Lagrangian surgery on the clean intersection between the conormal Lagrangian $L_K$ and the zero section $S^3$, shifted off the zero section along a non-vanishing closed $1$-form; see \cite{ChauhanEkholmLonghi}. 
Alternatively, we can view $S^3 \setminus K$ as a $3$-manifold with cylindrical end, and count holomorphic curves in $(T^*(S^3 \setminus K), S^3 \setminus K \sqcup S^3 \setminus K)$, where the Lagrangians $S^3 \setminus K$ are shifted off from each other using a non-vanishing closed $1$-form.\footnote{To be more precise, the skein-valued curve count would give $\Phi_{S^3 \setminus K}(x,q^2)$, since all the linking numbers are counted twice, from the two copies of the $\mathrm{GL}_1$-skein module.}  

The framing vector field $\xi_{\mathrm{fr}}$ determines a $4$-chain that can be used in this count. 
When the shift is small, the correspondence between holomorphic curves and Morse flow graphs \cite{Ekholm_Morse-flow-tree} (see also \cite[Appendix A]{EkholmLonghiParkShende}) shows that all the holomorphic curves are holomorphic annuli (and their multiple covers), which correspond to Morse flow loops, and this is what we are counting. 
Thus, the invariance also follows from the corresponding deformation invariance of skein-valued curve count \cite{EkholmShende}. 
\end{rmk}

\subsection{Adding extra Wilson line defects}

Before we conclude this section, we remark that Theorem \ref{mainthm:flow-loop-count} is a special case of a more general theorem, where we allow insertion of extra Wilson line defects in $S^3 \setminus K$:
\begin{thm}\label{thm:abelianization}
For any fibered knot $K$, we have the ``abelianization'' homomorphism\footnote{As before, the $\mathrm{GL}_1$-skein module on the right-hand side is twisted by framing lines and action-completed; we just omitted it for simplicity of notation.}
\[
\mathrm{Ab} : \Sk^{\mathrm{GL}_2}_q(S^3 \setminus K) \rightarrow \Sk^{\mathrm{GL}_1}_q(S^3 \setminus K)^{\otimes 2}. 
\]
\end{thm}
\begin{proof}
This is a special case of the skein trace map \cite[Theorem 1.3]{EkholmLonghiParkShende}, applied to the trivial double cover Lagrangian $S^3 \setminus K \sqcup S^3 \setminus K \subset T^*(S^3 \setminus K)$. 
\end{proof}
The skein-valued flow loop count in Definition \ref{defn:skein-valued-flow-loop-count} is nothing but the image of the empty skein under this abelianization map:
\[
\Phi_{S^3 \setminus K}(x,q^2) = \mathrm{Ab}([\emptyset]).
\]

The following conjecture generalizes Conjecture \ref{conj:flow-equals-quantum}:
\begin{conj}
The abelianization map in Theorem \ref{thm:abelianization} is equivalent to the one constructed using quantum groups \cite[Theorem B]{Park_skeins}.\footnote{
We used $\mathrm{SL}_2$-skein modules in \cite{Park_skeins}, but they can be turned into $\mathrm{GL}_2$-skein modules by multiplying each framed, oriented link $L$ by $q^{\frac{\sum_i \mathrm{lk}(L_i, L_i)}{4}}$, where the sum is over all connected components $L_i$ of $L$. 
} 
\end{conj}

\section{Geometric state sum models on knot holders}\label{sec:knot-holder-state-sum}

In this section, we explain how one can compute the flow loop count defined in Section \ref{sec:skein-valued-flow-loop-count} in practice. 
The key idea is to reduce the dynamics of Morse flow into the symbolic dynamics on the corresponding knot holder, which in turn gives rise to a state sum model on knot holders.

\subsection{Knot holders}

We first review the theory of knot holders. 
See \cite{BirmanWilliams} for the original paper and \cite{GhristHolmesSullivan} for a comprehensive review. 
\begin{defn}[{\cite[Definition 2.2.1]{GhristHolmesSullivan}}]
A \emph{template} is a branched surface with boundary, supporting an expansive semiflow. 
\end{defn}
Here, ``semiflow'' means that the flow is only defined in forward time direction, and ``expansive'' means that the volume is expanding everywhere. 
In other words, a template is a union of strips, whose ends are glued along some branch lines where they can first join and then split; see Figure \ref{fig:template}. 
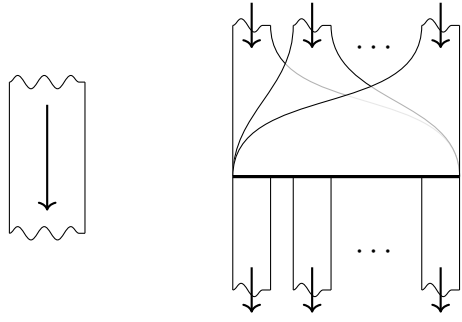
\begin{figure}
\centering
\[
\vcenter{\hbox{
\begin{tikzpicture}
\draw (0, 0) -- (0, 2);
\draw (1, 0) -- (1, 2);
\draw[snake it] (0, 0) -- (1, 0);
\draw[snake it] (0, 2) -- (1, 2);
\draw[->, thick] (0.5, 1.7) -- (0.5, 0.3);
\end{tikzpicture}
}}
\qquad\qquad
\vcenter{\hbox{
\begin{tikzpicture}
\draw (0, 0) to[out=90, in=-90] (0, 2);
\draw (3, 0) to[out=90, in=-90] (0.5, 2);
\draw[snake it] (0, 2) -- (0.5, 2); 
\draw[->, thick] (0.25, 2.3) -- (0.25, 1.7);

\filldraw[white, opacity=0.7] (0, 0) to[out=90, in=-90] (0.8, 2) -- (1.3, 2) to[out=-90, in=90] (3, 0) -- cycle;
\draw (0, 0) to[out=90, in=-90] (0.8, 2);
\draw (3, 0) to[out=90, in=-90] (1.3, 2);
\draw[snake it] (0.8, 2) -- (1.3, 2); 
\draw[->, thick] (1.05, 2.3) -- (1.05, 1.7);

\node at (1.9, 1.7){$\cdots$};

\filldraw[white, opacity=0.7] (0, 0) to[out=90, in=-90] (2.5, 2) -- (3, 2) to[out=-90, in=90] (3, 0) -- cycle;
\draw (0, 0) to[out=90, in=-90] (2.5, 2);
\draw (3, 0) to[out=90, in=-90] (3, 2);
\draw[snake it] (2.5, 2) -- (3, 2); 
\draw[->, thick] (2.75, 2.3) -- (2.75, 1.7);

\draw[very thick] (0, 0) -- (3, 0);

\draw (0, 0) to[out=-90, in=90] (0, -1.5);
\draw (0.5, 0) to[out=-90, in=90] (0.5, -1.5); 
\draw[snake it] (0.0, -1.5) -- (0.5, -1.5); 
\draw[->, thick] (0.25, -1.2) -- (0.25, -1.8);

\draw (0.8, 0) to[out=-90, in=90] (0.8, -1.5);
\draw (1.3, 0) to[out=-90, in=90] (1.3, -1.5);
\draw[snake it] (0.8, -1.5) -- (1.3, -1.5); 
\draw[->, thick] (1.05, -1.2) -- (1.05, -1.8);

\node at (1.9, -1){$\cdots$};

\draw (2.5, 0) to[out=-90, in=90] (2.5, -1.5);
\draw (3.0, 0) to[out=-90, in=90] (3.0, -1.5);
\draw[snake it] (2.5, -1.5) -- (3.0, -1.5); 
\draw[->, thick] (2.75, -1.2) -- (2.75, -1.8);
\end{tikzpicture}
}}
\]
\caption{Left: a strip supporting a flow along the direction of the arrow. Right: ends of strips glued along a branch line.}
\label{fig:template}
\end{figure}

A template embedded in a $3$-manifold is called an \emph{embedded template}, or a \emph{knot holder} \cite{BirmanWilliams}, as it can ``hold knots''; 
any closed $1$-manifold drawn on a knot holder along the direction of the semiflow determines a framed\footnote{The framing is induced by the normal direction to the knot holder; since the knot holder may be unorientable, here we are allowing half-twists in the framing.}, oriented link in the ambient $3$-manifold. 

A template $\mathcal{T}$ determines a symbolic dynamical system known as a \emph{shift of finite type} (SFT\footnote{Not to be confused with Symplectic Field Theory \cite{EliashbergGiventalHofer_SFT}, which is also very much relevant to skein-valued curve counts; see \cite[Sec. 8.3]{EkholmLonghiParkShende}.}) $\Sigma_{\mathcal{T}}$ as follows. 
Assign a symbol to each strip. 
Given any symbol, look at the branch line corresponding to the outgoing end of the strip corresponding to that symbol. 
The strips whose incoming ends are glued to that branch line determine the symbols that can follow the given symbol. 
Then, $\Sigma_{\mathcal{T}}$ is defined to be the set of all forward orbits, i.e., sequences
\[
\mathbf{a} =  a_0 a_1 a_2 \cdots
\]
of symbols obeying the admissibility condition (i.e., outgoing end of the strip $a_n$ is glued to the incoming end of the strip $a_{n+1}$). 
The shift map
\begin{align*}
\sigma : \Sigma_{\mathcal{T}} &\rightarrow \Sigma_{\mathcal{T}} \\
a_0 a_1 a_2 \cdots &\mapsto a_1 a_2 a_3 \cdots
\end{align*}
describes the discrete time evolution in this symbolic dynamical system. 
In this language, a knot contained in a knot holder corresponds to a periodic sequence of the symbols. 

Knot holders are central to the study of dynamics of three-dimensional flows, thanks to the following powerful theorem of Birman and Williams:
\begin{thm}[{\cite[Theorem 2.1]{BirmanWilliams}}]\label{thm:BirmanWilliams}
Given a flow $\phi_t$ on a $3$-manifold $M$ having a hyperbolic chain-recurrent set, there is a knot holder $\KH$ in $M$ such that the link of periodic orbits $L_\phi$ of the flow is in one-to-one correspondence (via isotopy) with the link of periodic orbits $L_{\KH}$ on the knot holder.\footnote{To be precise, with at most two extra orbits in $L_{\KH}$, coming from the DA (``derived from Anosov'') procedure, when the dimension of the chain-recurrent set is $>1$.} 
\end{thm}
See, e.g., \cite[Sec. 1.2]{GhristHolmesSullivan}, for a concise introduction to basic notions in dynamical systems, including the definition of the chain-recurrent set and its hyperbolicity. 
Here, we just remark that it is a mild assumption that can always be achieved in the cases of our interest, by modifying the circle-valued Morse function appropriately. 
Hence, the dynamics of a three-dimensional flow can be encoded in the symbolic dynamics on a knot holder, at least when it comes to questions on periodic orbits and their link types. 

\begin{eg}
The Lorenz knot holder \cite{BirmanWilliams_Lorenz}, which is a knot holder for the Lorenz system (Figure \ref{fig:Lorenz_attractor}), is shown in Figure \ref{fig:Lorenz-knot-holder}. 
\begin{figure}
\centering
\[
\vcenter{\hbox{
\begin{tikzpicture}[scale=0.8]

\draw (0.0, 0) [partial ellipse = 0 : 180 : 2.0 and 1.8];
\draw (-0.5, 0) [partial ellipse = 0 : 360 : 0.5];
\draw (-0.5, 0) [partial ellipse = -180 : 0 : 1.5];
\draw[->, thick] (-0.5, 0) [partial ellipse = -10 : -40 : 1.0];

\filldraw[white, opacity=0.9] (2.0, 0) [partial ellipse = 0 : 180 : 2.0 and 1.8];
\draw (2.0, 0) [partial ellipse = 0 : 180 : 2.0 and 1.8];
\draw (2.5, 0) [partial ellipse = 0 : 360 : 0.5];
\draw (2.5, 0) [partial ellipse = -180 : 0 : 1.5];
\draw[->, thick] (2.5, 0) [partial ellipse = -170 : -140 : 1.0];

\draw[very thick] (0, 0) -- (2, 0);
\end{tikzpicture}
}}
\]
\caption{Lorenz knot holder}
\label{fig:Lorenz-knot-holder}
\end{figure}
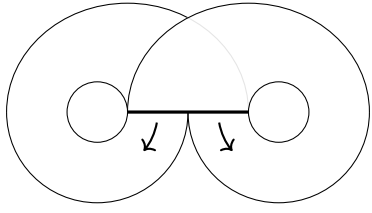
\end{eg}

\subsection{Knot holders for fibered link complements}\label{subsec:knot-holder-construction}
Here, we describe in detail how to construct a knot holder for the suspension flow in any mapping torus of a surface with boundary, such as in any fibered link complement. 

\begin{defn}
Given a compact, connected, oriented surface $F$, a \emph{marking} on $F$ consists of the following:
\begin{enumerate}
\item A collection of (green) non-intersecting arcs $\alpha_1, \cdots, \alpha_{n} \subset F$ with boundary points on $\partial F$, such that $F \setminus (\alpha_1 \cup \cdots \cup \alpha_n)$ is a disjoint union of disks. 
\item A collection of (red) base points $p_1, \cdots, p_m \in \mathrm{int}\qty(F \setminus (\alpha_1 \cup \cdots \cup \alpha_n))$, one for each disk component. 
\end{enumerate}
\end{defn}
If $F$ is a surface of genus $g$ and $b$ boundary components, then $n \geq 2g + (b-1)$, and
\[
m = \chi(F \setminus (\alpha_1 \cup \cdots \cup \alpha_n)) = 2-2g-b + n.
\]
See Figure \ref{fig:marked-surface} for an example of a marked surface, where the arcs are drawn in green, and the base points are drawn in red. 
\begin{figure}
\centering
\[
\vcenter{\hbox{
\begin{tikzpicture}[scale=0.8]
\draw (0, 0) -- (0, -2) -- (8, -2) -- (8, 0);
\draw (3.5, 0) -- (4.5, 0);
\node[above] at (4, 0){$\cdots$};

\filldraw[red] (4, -1.7) circle (0.05);

\draw (1.25, 0) [partial ellipse = 0 : 180 : 1.25];
\draw (1.25, 0) [partial ellipse = 0 : 180 : 0.75];
\draw[darkgreen, thick] (1.25, 0.75) -- (1.25, 1.25);

\filldraw[white, opacity=0.9] (1, 0) arc (180:0:1.25) -- (3,0) arc (0:180:0.75) -- cycle; 

\draw (2.25, 0) [partial ellipse = 0 : 180 : 1.25];
\draw (2.25, 0) [partial ellipse = 0 : 180 : 0.75];
\draw[darkgreen, thick] (2.25, 0.75) -- (2.25, 1.25);

\draw (0.5, 0) -- (1, 0);
\draw (1.5, 0) -- (2, 0);
\draw (2.5, 0) -- (3, 0);

\begin{scope}[shift={(4.5, 0)}]
    \draw (1.25, 0) [partial ellipse = 0 : 180 : 1.25];
    \draw (1.25, 0) [partial ellipse = 0 : 180 : 0.75];
    \draw[darkgreen, thick] (1.25, 0.75) -- (1.25, 1.25);
    
    \filldraw[white, opacity=0.9] (1, 0) arc (180:0:1.25) -- (3,0) arc (0:180:0.75) -- cycle; 
    
    \draw (2.25, 0) [partial ellipse = 0 : 180 : 1.25];
    \draw (2.25, 0) [partial ellipse = 0 : 180 : 0.75];
    \draw[darkgreen, thick] (2.25, 0.75) -- (2.25, 1.25);
    
    \draw (0.5, 0) -- (1, 0);
    \draw (1.5, 0) -- (2, 0);
    \draw (2.5, 0) -- (3, 0);
\end{scope}
\end{tikzpicture}
}}
\]
\caption{A marked surface}
\label{fig:marked-surface}
\end{figure}
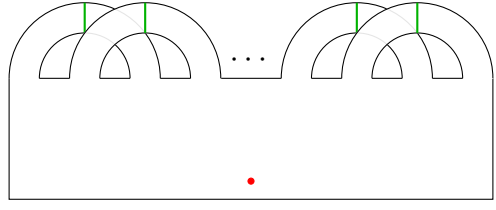

\begin{rmk}\label{rmk:marking-Morse-function}
Note that, given such a marking, there is a Morse function on $F$ for which
\begin{enumerate}
\item the boundary $\partial F$ is a level set which is the global minimum, 
\item the (red) base points are the local maxima, and 
\item the (green) arcs are the descending manifolds of the saddle points. 
\end{enumerate}
\end{rmk}

Given a marking on $F$, any self-diffeomorphism $\varphi : F \rightarrow F$ fixing the boundary $\partial F$ pointwise, induces another marking on $F$; we will simply denote the resulting marked surface by $\varphi(F)$. 
By choosing some isotopy of $\varphi$, we may assume that $\varphi$ preserves the set of base points, \(\{p_1, \cdots, p_m\} = \{\varphi(p_1), \cdots, \varphi(p_m)\}\).\footnote{$\varphi$ may permute the base points.} 
Since each disk component $D$ of $F \setminus (\alpha_1 \cup \cdots \cup \alpha_n)$ contains a base point in the interior, we can isotope all parts of the arcs of $\varphi(F)$ drawn on $D$ to a small neighborhood of the boundary $\partial D$, without crossing the base point. 
We can then split the arcs by colliding them with the boundary $\partial D \cap \partial F$ (i.e., the part of $\partial D$ which is not in $\alpha_1 \cup \cdots \cup \alpha_n$) using the move shown on the left of Figure \ref{fig:splitting-joining}. 
After such splitting, we will have a collection of arcs each of which is parallel to some arc $\alpha_i$. 
We then join the parallel arcs using the move shown on the right of Figure \ref{fig:splitting-joining}. 
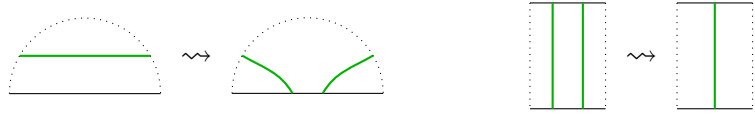
\begin{figure}
\centering
\[
\vcenter{\hbox{
\begin{tikzpicture}
\draw (-1, 0) -- (1, 0);
\draw[darkgreen, thick] ({-sqrt(3)/2}, 0.5) -- ({sqrt(3)/2}, 0.5);
\draw[dotted] (-1, 0) arc (180:0:1); 
\end{tikzpicture}
}}
\rightsquigarrow
\vcenter{\hbox{
\begin{tikzpicture}
\coordinate (a) at ({-sqrt(3)/2}, 0.5);
\coordinate (b) at ({sqrt(3)/2}, 0.5);
\draw (-1, 0) -- (1, 0);
\draw[darkgreen, thick] (a) to[out=-30, in=120] (-0.2, 0);
\draw[darkgreen, thick] (0.2, 0) to[out=60, in=-150] (b);
\draw[dotted] (-1, 0) arc (180:0:1); 
\end{tikzpicture}
}}
\qquad\qquad
\vcenter{\hbox{
\begin{tikzpicture}
\draw (-0.5, 0.7) -- (0.5, 0.7);
\draw (-0.5, -0.7) -- (0.5, -0.7);
\draw[dotted] (-0.5, 0.7) -- (-0.5, -0.7);
\draw[dotted] (0.5, 0.7) -- (0.5, -0.7);
\draw[darkgreen, thick] (-0.2, 0.7) -- (-0.2, -0.7);
\draw[darkgreen, thick] (0.2, 0.7) -- (0.2, -0.7);
\end{tikzpicture}
}}
\rightsquigarrow
\vcenter{\hbox{
\begin{tikzpicture}
\draw (-0.5, 0.7) -- (0.5, 0.7);
\draw (-0.5, -0.7) -- (0.5, -0.7);
\draw[dotted] (-0.5, 0.7) -- (-0.5, -0.7);
\draw[dotted] (0.5, 0.7) -- (0.5, -0.7);
\draw[darkgreen, thick] (0, 0.7) -- (0, -0.7);
\end{tikzpicture}
}}
\]
\caption{Splitting (left) and joining (right)}
\label{fig:splitting-joining}
\end{figure}

The movie of splitting and joining of arcs we just described determines a branched surface $S \subset F \times [0, 1]$ such that
\begin{enumerate}
\item $S \cap (F \times \{0\}) = \varphi(\alpha_1 \cup \cdots \cup \alpha_n)$, 
\item $S \cap (F \times \{1\}) = \alpha_1 \cup \cdots \cup \alpha_n$, and 
\item $S$ is disjoint from $\{p_1, \cdots, p_m\} \times [0, 1] \subset F \times [0, 1]$. 
\end{enumerate}
Gluing the two boundary components of $F \times [0,1]$ by the monodromy $\varphi$ (i.e., $(q,1) \sim (\varphi(q),0)$ for any $q \in F$), we obtain the mapping torus $M_\varphi$ of $\varphi$, equipped with
\begin{enumerate}
\item a branched surface $\KH := S/\sim \;\subset M_\varphi$, and
\item an oriented link $\gamma_e := (\{p_1, \cdots, p_m\} \times [0, 1])/\sim \;\subset M_\varphi$ disjoint from $\KH$. 
\end{enumerate}

We now deform the suspension flow on $M_\varphi$ using the movie of markings as follows. 
At each stage of the movie, choose a Morse function on $F$ adapted to the corresponding marking, as in Remark~\ref{rmk:marking-Morse-function}, and let the flow spend a sufficiently long time following the associated fiberwise Morse flow. 
In this way, every point of $F$ away from a small neighborhood of the base points, after sufficiently long time, will have flowed either toward the boundary or into a small neighborhood of the descending manifolds of the saddle points, namely the arcs of the marking. 
As the marking undergoes splitting and joining, these neighborhoods correspondingly split and join. 
Thus, in the mapping torus $M_\varphi$, away from the repelling elliptic flow loops $\gamma_e$, the recurrent part of the resulting flow is carried by a neighborhood of the branched surface $\KH := S/\sim \;\subset M_\varphi$. 

The standard Markov partition construction in the proof of Theorem~\ref{thm:BirmanWilliams} applies to the recurrent set carried by this neighborhood. 
The corresponding Markov flow boxes may be chosen so that, after collapsing their stable directions, they give precisely the strips swept out by the arcs of the marking, with the standard joining and splitting charts. 
The resulting branched surface is therefore $\KH$. 
The resulting properties of the flow are summarized as follows:
\begin{cor}\label{cor:knot-holder-existence}
For any mapping torus $M_\varphi$ of a surface $F$ with boundary, one can choose some metric and a circle-valued Morse function on $M_\varphi$ in such a way that 
\begin{enumerate}
\item the boundary of $M_\varphi$ is a sink of the Morse flow, 
\item all the periodic orbits of the Morse flow are nondegenerate, 
\item the hyperbolic flow loops are in one-to-one correspondence with periodic orbits of the symbolic dynamics on the knot holder $\KH$, and
\item the (primitive) elliptic flow loops are in one-to-one correspondence with the components of the link $\gamma_e$. 
\end{enumerate}
Moreover, the one-to-one correspondence of the periodic orbits is by ambient isotopy. 
\end{cor}

We now specialize the construction described above to the case of fibered link complements. 
Let $K \subset S^3$ be a fibered link, and let $F$ be its minimal genus Seifert surface, which is unique up to isotopy. 
We can then view the complement $S^3 \setminus K$ as the mapping torus $M_\varphi$ of the monodromy map $\varphi : F \rightarrow F$. 
In this setup, there are two convenient choices for the repelling elliptic flow loops $\gamma_e$ (i.e., the movie of red base points):
\begin{itemize}
\item Meridian of $K \subset S^3$: in this case, there will be a single red base point $p \in F$ close to the boundary. 
This choice of elliptic flow loop contributes an overall factor of $1-x$ to $\Phi_{S^3 \setminus K}$, as it has $0$ linking number with any other hyperbolic flow loops in $\KH$. 
\item Braid axis of $F \subset S^3$ as a braided Seifert surface \cite{Rudolph}: in this case, the number of red base points on $F$ is equal to the number of braid strands. 
While we would no longer have a nice factorization of the contribution of elliptic flow loops (as they link non-trivially with hyperbolic flow loops), this choice is often useful if we want the knot holder to be adapted to a knot diagram, as we can then choose the arcs $\alpha_1, \cdots, \alpha_n$ to sit at the bands of the braided Seifert surface; see Figure \ref{fig:braided-Seifert-surface}.  
\end{itemize}

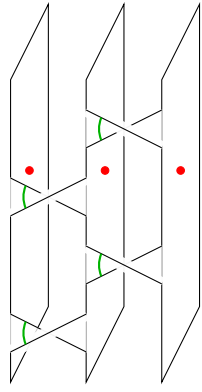
\begin{figure}
\centering
\[
\vcenter{\hbox{
\begin{tikzpicture}
\draw (0, 0) -- (0, 4) -- (0.5, 5) -- (0.5, 1) -- cycle;
\draw (1, 0) -- (1, 4) -- (1.5, 5) -- (1.5, 1) -- cycle;
\draw (2, 0) -- (2, 4) -- (2.5, 5) -- (2.5, 1) -- cycle;

\filldraw[red] (0.25, 2.8) circle (0.05);
\filldraw[red] (1.25, 2.8) circle (0.05);
\filldraw[red] (2.25, 2.8) circle (0.05);

\filldraw[white, opacity=0.8] (0, 0.4) -- (1, 0.9) -- (1, 0.4) -- (0, 0.9) -- cycle;
\draw (1, 0.4) -- (0, 0.9);
\draw[white, opacity=0.95, line width=5] (0, 0.4) -- (1, 0.9);
\draw (0, 0.4) -- (1, 0.9);
% \filldraw[darkgreen] (0.5, 0.65) circle (0.05);
\draw[darkgreen, thick] (0.2, 0.5) to[out=110, in=-110] (0.2, 0.8);

\begin{scope}[shift={(0, 1.8)}]
    \filldraw[white, opacity=0.8] (0, 0.4) -- (1, 0.9) -- (1, 0.4) -- (0, 0.9) -- cycle;
    \draw (1, 0.4) -- (0, 0.9);
    \draw[white, opacity=0.95, line width=5] (0, 0.4) -- (1, 0.9);
    \draw (0, 0.4) -- (1, 0.9);
    % \filldraw[darkgreen] (0.5, 0.65) circle (0.05);
    \draw[darkgreen, thick] (0.2, 0.5) to[out=110, in=-110] (0.2, 0.8);
\end{scope}
\begin{scope}[shift={(1, 2.2)}, yscale=-1]
    \filldraw[white, opacity=0.8] (0, 0.4) -- (1, 0.9) -- (1, 0.4) -- (0, 0.9) -- cycle;
    \draw (1, 0.4) -- (0, 0.9);
    \draw[white, opacity=0.95, line width=5] (0, 0.4) -- (1, 0.9);
    \draw (0, 0.4) -- (1, 0.9);
    % \filldraw[darkgreen] (0.5, 0.65) circle (0.05);
    \draw[darkgreen, thick] (0.2, 0.5) to[out=110, in=-110] (0.2, 0.8);
\end{scope}
\begin{scope}[shift={(1, 4.0)}, yscale=-1]
    \filldraw[white, opacity=0.8] (0, 0.4) -- (1, 0.9) -- (1, 0.4) -- (0, 0.9) -- cycle;
    \draw (1, 0.4) -- (0, 0.9);
    \draw[white, opacity=0.95, line width=5] (0, 0.4) -- (1, 0.9);
    \draw (0, 0.4) -- (1, 0.9);
    % \filldraw[darkgreen] (0.5, 0.65) circle (0.05);
    \draw[darkgreen, thick] (0.2, 0.5) to[out=110, in=-110] (0.2, 0.8);
\end{scope}
\end{tikzpicture}
}}
\]
\caption{A natural choice of marking on a braided Seifert surface; the red base points are the intersection of the braid axis with the braided Seifert surface, and the green arcs are the cocores of the bands.}
\label{fig:braided-Seifert-surface}
\end{figure}

\begin{eg}[Trefoil knot complement]\label{eg:trefoil-direct-model}
We illustrate the construction of knot holders through the simplest non-trivial fibered knot, namely the trefoil knot $K = \mathbf{3}_1^r$. 
\begin{figure}
\centering
\[
\vcenter{\hbox{
\begin{tikzpicture}

\draw (0, 0) -- (0, 4) -- (-1, 4) -- (-1, 0) -- cycle;
\draw (1, 0) -- (1, 4) -- (2, 4) -- (2, 0) -- cycle;

\filldraw[red] (1.75, 2) circle (0.05);

\begin{scope}[shift={(0, 0)}]
    \filldraw[white, opacity=1.0] (0, 0.5) -- (1, 1.0) -- (1, 0.5) -- (0, 1.0) -- cycle;
    \draw (1, 0.5) -- (0, 1.0);
    \draw[white, opacity=0.95, line width=5] (0, 0.5) -- (1, 1.0);
    \draw (0, 0.5) -- (1, 1.0);
    \draw[darkgreen, thick, ->] (0.2, 0.6) to[out=110, in=-110] (0.2, 0.9);
    \node[darkgreen, left] at (0.1, 0.75){$\alpha_1$};
\end{scope}

\begin{scope}[shift={(0, 1.25)}]
    \filldraw[white, opacity=1.0] (0, 0.5) -- (1, 1.0) -- (1, 0.5) -- (0, 1.0) -- cycle;
    \draw (1, 0.5) -- (0, 1.0);
    \draw[white, opacity=0.95, line width=5] (0, 0.5) -- (1, 1.0);
    \draw (0, 0.5) -- (1, 1.0);
    \draw[darkgreen!50, thick, ->] (0.2, 0.6) to[out=110, in=-110] (0.2, 0.9);
    \node[darkgreen!50, left] at (0.1, 0.75){$\alpha_2$};
\end{scope}

\begin{scope}[shift={(0, 2.5)}]
    \filldraw[white, opacity=1.0] (0, 0.5) -- (1, 1.0) -- (1, 0.5) -- (0, 1.0) -- cycle;
    \draw (1, 0.5) -- (0, 1.0);
    \draw[white, opacity=0.95, line width=5] (0, 0.5) -- (1, 1.0);
    \draw (0, 0.5) -- (1, 1.0);
    % \draw[darkgreen, thick] (0.2, 0.6) to[out=110, in=-110] (0.2, 0.9);
    % \node[darkgreen, left] at (0.1, 0.75){$\alpha_3$};
\end{scope}

\filldraw[lightgray, opacity=0.2] (1, 0) -- (1, 0.5) -- (0.5, 0.75) -- (1, 1.0) -- (1, 1.75) -- (0.5, 2.0) -- (1, 2.25) -- (1, 3.0) -- (0.5, 3.25) -- (1, 3.5) -- (1, 4) -- (2, 4) -- (2, 0) -- cycle;

\node[below] at (0.5, -0.2){$F$};
\end{tikzpicture}
}}
\;
\overset{\varphi}{\rightsquigarrow}
\;
\vcenter{\hbox{
\begin{tikzpicture}[ 
    mid/.style={
        postaction={
            decorate,
            decoration={
                markings,
                % Mark at position 0.5 (the middle) with an arrow tip
                mark=at position 0.5 with {\arrow{>}} 
            }
        }
    },
    midback/.style={
        postaction={
            decorate,
            decoration={
                markings,
                % Mark at position 0.5 (the middle) with an arrow tip
                mark=at position 0.5 with {\arrow{<}} 
            }
        }
    },
]

\draw (0, 0) -- (0, 4) -- (-1, 4) -- (-1, 0) -- cycle;
\draw (1, 0) -- (1, 4) -- (2, 4) -- (2, 0) -- cycle;

\filldraw[red] (1.75, 2) circle (0.05);

\begin{scope}[shift={(0, 0)}]
    \filldraw[white, opacity=1.0] (0, 0.5) -- (1, 1.0) -- (1, 0.5) -- (0, 1.0) -- cycle;
    \draw (1, 0.5) -- (0, 1.0);
    \draw[white, opacity=0.95, line width=5] (0, 0.5) -- (1, 1.0);
    \draw (0, 0.5) -- (1, 1.0);
\end{scope}

\begin{scope}[shift={(0, 1.25)}]
    \filldraw[white, opacity=1.0] (0, 0.5) -- (1, 1.0) -- (1, 0.5) -- (0, 1.0) -- cycle;
    \draw (1, 0.5) -- (0, 1.0);
    \draw[white, opacity=0.95, line width=5] (0, 0.5) -- (1, 1.0);
    \draw (0, 0.5) -- (1, 1.0);
\end{scope}

\begin{scope}[shift={(0, 2.5)}]
    \filldraw[white, opacity=1.0] (0, 0.5) -- (1, 1.0) -- (1, 0.5) -- (0, 1.0) -- cycle;
    \draw (1, 0.5) -- (0, 1.0);
    \draw[white, opacity=0.95, line width=5] (0, 0.5) -- (1, 1.0);
    \draw (0, 0.5) -- (1, 1.0);
\end{scope}

\filldraw[lightgray, opacity=0.2] (1, 0) -- (1, 0.5) -- (0.5, 0.75) -- (1, 1.0) -- (1, 1.75) -- (0.5, 2.0) -- (1, 2.25) -- (1, 3.0) -- (0.5, 3.25) -- (1, 3.5) -- (1, 4) -- (2, 4) -- (2, 0) -- cycle;

\draw[darkgreen, thick, dotted] (0.2, 0.6) to[out=100, in=190] (0.42, 0.75);
\draw[darkgreen, thick, mid] (0.58, 0.75) to[out=10, in=-90] (1.3, 1.375) to[out=90, in=-10] (0.58, 1.995);
\draw[darkgreen, thick, dotted] (0.42, 2.005) to[out=170, in=90] (-0.3, 1.375) to[out=-90, in=-100] (0.2, 0.9);
\node[darkgreen, right] at (1.3, 1.375){$\varphi(\alpha_1)$};
\begin{scope}[shift={(0, 1.25)}]
    \draw[darkgreen!50, thick, dotted] (0.2, 0.6) to[out=100, in=190] (0.42, 0.745);
    \draw[darkgreen!50, thick, mid] (0.58, 0.755) to[out=10, in=-90] (1.3, 1.375) to[out=90, in=-10] (0.58, 2.0);
    \draw[darkgreen!50, thick, dotted] (0.42, 2.0) to[out=170, in=90] (-0.3, 1.375) to[out=-90, in=-100] (0.2, 0.9);
    \node[darkgreen!50, right] at (1.3, 1.375){$\varphi(\alpha_2)$};
\end{scope}

\node[below] at (0.5, -0.2){$\varphi(F)$};
\end{tikzpicture}
}}
=
\vcenter{\hbox{
\begin{tikzpicture}[ 
    mid/.style={
        postaction={
            decorate,
            decoration={
                markings,
                % Mark at position 0.5 (the middle) with an arrow tip
                mark=at position 0.5 with {\arrow{>}} 
            }
        }
    },
    midback/.style={
        postaction={
            decorate,
            decoration={
                markings,
                % Mark at position 0.5 (the middle) with an arrow tip
                mark=at position 0.5 with {\arrow{<}} 
            }
        }
    },
]
\draw[darkgreen, thick, mid] (0, 0) -- (0, 0.5);
\draw (0, 0.5) to[out=180, in=-90] (-0.25, 0.75) to[out=90, in=180] (0, 1.0);
\draw[darkgreen!50, thick, mid] (0, 1.0) -- (0, 1.5);
\draw (0, 1.5) to[out=180, in=-90] (-0.25, 1.75) to[out=90, in=180] (0, 2.0);

\draw (0, 2.5) to[out=0, in=-90] (0.25, 2.75) to[out=90, in=0] (0, 3.0);
\draw (0, 3.5) to[out=0, in=-90] (0.25, 3.75) to[out=90, in=0] (0, 4.0);
\draw[darkgreen!50, thick, mid] (0, 3.0) -- (0, 3.5);
\draw[darkgreen, thick, mid] (0, 4.0) -- (0, 4.5);

\draw (0, 0) to[out=180, in=-90] (-1.25, 1.25) to[out=90, in=180] (0, 2.5);
\draw (0, 2.0) to[out=0, in=-90] (1.25, 3.25) to[out=90, in=0] (0, 4.5);

\node[darkgreen, right] at (0, 0.25){$\alpha_1$};
\node[darkgreen, left] at (0, 4.25){$\alpha_1$};
\node[darkgreen!50, right] at (0, 1.25){$\alpha_2$};
\node[darkgreen!50, left] at (0, 3.25){$\alpha_2$};

\draw[darkgreen, thick, dotted, mid] (0, 4.0) to[out=60, in=90] (0.5, 3.75) to[out=-90, in=0] (0, 3.25);
\draw[darkgreen, thick, dotted, mid] (0, 1.25) to[out=180, in=90] (-0.5, 0.75) to[out=-90, in=-120] (0, 0.5);
\draw[darkgreen!50, thick, dotted, mid] (0, 3.0) to[out=30, in=90] (0.5, 2.75) to[out=-90, in=90] (-0.5, 1.75) to[out=-90, in=-150] (0, 1.5);

\node[darkgreen!50, right] at (0.5, 2.75){$\varphi(\alpha_2)$};
\node[darkgreen, right] at (0.5, 3.75){$\varphi(\alpha_1)$};

\filldraw[red] (1.0, 3.25) circle (0.05);
\end{tikzpicture}
}}
\;
\rightsquigarrow
\;
\vcenter{\hbox{
\begin{tikzpicture}[scale=1.0]
\draw (0.5, 0) to[out=90, in=-90] (0, 2);
\draw[white, line width=5] (0, 0) to[out=90, in=-90] (0.5, 2);
\draw (0, 0) to[out=90, in=-90] (0.5, 2);

\filldraw[white, opacity=0.8] (-0.5, 1.5) to[out=0, in=-90] (0, 2) -- (0.5, 2) to[out=-90, in=0] (-0.5, 1.25) -- cycle;
\draw (0, 2) to[out=90, in=0] (-0.5, 2.5) to[out=180, in=90] (-1, 2) to[out=-90, in=180] (-0.5, 1.5) to[out=0, in=-90] (0, 2);
\draw (0.25, 2) to[out=90, in=0] (-0.5, 2.75) to[out=180, in=90] (-1.25, 2) to[out=-90, in=180] (-0.5, 1.25) to[out=0, in=-90] (0.5, 2);

\draw (0.25, 2) to[out=90, in=0] (-0.625, 3) to[out=180, in=90] (-1.5, 2) to[out=-90, in=180] (-0.7, -0.5) to[out=0, in=-90] (0, 0);
\draw (0.5, 2) to[out=90, in=0] (-0.625, 3.25) to[out=180, in=90] (-1.8, 2) to[out=-90, in=180] (-0.7, -1.0) to[out=0, in=-90] (0.5, 0);

\draw[darkgreen, ultra thick] (0, 0) -- (0.5, 0);
\draw[darkgreen!50, ultra thick] (0, 2) -- (0.5, 2);

\draw[->] (0.25, 0.1) -- (0.25, 0.3);
\draw[->] (0.11, 2.1) -- (0.07, 2.3);
\draw[->] (0.37, 2.1) -- (0.36, 2.3);

\node[below] at (-0.6, -1.2){$\KH$};
\end{tikzpicture}
}}
\]
\caption{Left: the minimal genus Seifert surface of the trefoil knot with a marking. Middle: monodromy action on the marking. Right: resulting knot holder.}
\label{fig:trefoil-knot_holder-direct_model}
\end{figure}
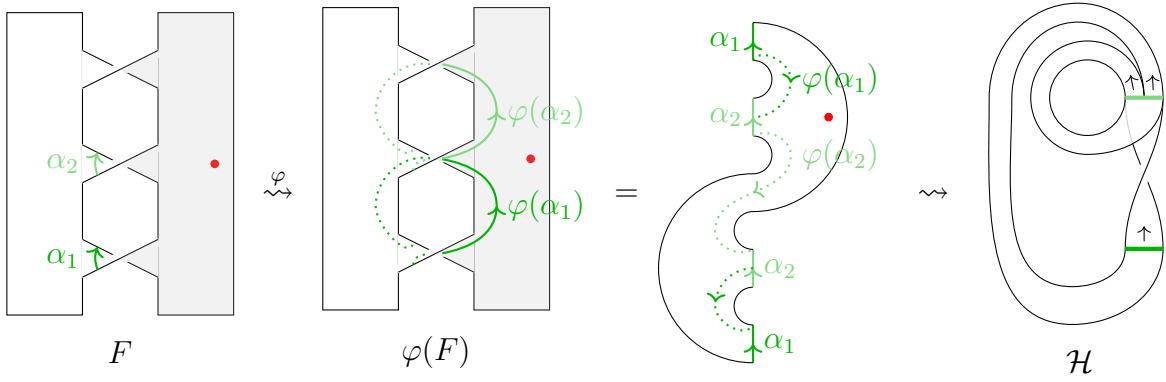
Start with the minimal genus Seifert surface $F$ of $K$, which is illustrated on the left of Figure \ref{fig:trefoil-knot_holder-direct_model}. 
Let us choose a marking with two arcs and a base point close to the boundary, as shown in the figure, so that $\gamma_e$ is a small meridian of $K$. 
The monodromy $\varphi : F\rightarrow F$ can be determined by pulling the arcs---considered as elastic cords in $S^3 \setminus F$ with fixed boundary points---from the front side of $F$ to the back side of $F$ \cite{BaaderGraf_fibered}. 
The middle of Figure \ref{fig:trefoil-knot_holder-direct_model} shows the new marking after this monodromy action, where we also show how this new marking looks in the disk $F \setminus (\alpha_1 \cup \alpha_2)$. 
Pushing the new marking to the boundary of this disk, we see that $\varphi(\alpha_1)$ becomes $\alpha_2$ with the opposite orientation, and that $\varphi(\alpha_2)$ splits into $\alpha_1$ and $\alpha_2$. 
In particular, the monodromy action on $H_1(F,\partial F)$ is given by
\[
\varphi
\begin{pmatrix}
\alpha_1 \\ 
\alpha_2
\end{pmatrix}
=
\begin{pmatrix}
0 & -1 \\ 
1 & 1
\end{pmatrix}
\begin{pmatrix}
\alpha_1 \\ 
\alpha_2
\end{pmatrix}.
\]
But the picture tells us much more than just the action on homology; 
by tracing exactly how the arcs split and join, we obtain the corresponding knot holder $\KH$ in $S^3 \setminus K$, shown on the right of Figure \ref{fig:trefoil-knot_holder-direct_model}. 
The resulting knot holder turns out to look like the Lorenz knot holder (Figure \ref{fig:Lorenz-knot-holder}) but with a half-twist, also known as the Smale horseshoe knot holder. 
The green lines shown on $\KH$ are the arcs $\KH \cap F = \alpha_1 \cup \alpha_2$; 
given any flow loop supported on $\KH$, its linking number with $K$ can be computed simply by looking at how many times it crosses those green lines. 
\end{eg}

\begin{eg}[Complement of braid-homogeneous links]\label{eg:braid-homogeneous-link-knot-holder}
Here, we explain how to read off a knot holder directly from any closure of a homogeneous braid. 
This will be a crucial ingredient in the proof of Theorem \ref{mainthm:flow-equals-quantum} in Section \ref{sec:BPS-q-series-from-flow-loops}. 

Recall that a braid word $\beta = \sigma_{i_1}^{\epsilon_1}\sigma_{i_2}^{\epsilon_2} \cdots \sigma_{i_k}^{\epsilon_k}$, $\epsilon_j = \pm 1$, in the standard Artin generators $\sigma_1, \cdots, \sigma_{N-1}$ is called \emph{homogeneous} if 
\begin{enumerate}
\item every $\sigma_i$ occurs at least once, and
\item for each $i$, the exponents of all occurrences of $\sigma_i$ are the same. 
\end{enumerate}
Stallings \cite{Stallings} showed that any braid-homogeneous link (i.e., any closure of homogeneous braid) is fibered, with the fiber surface $F$ given by $N$ copies of disks, one for each braid strand, connected by bands attached according to the braid word; Figure \ref{fig:braided-Seifert-surface} shows such a fiber surface in the case $\beta = \sigma_1 \sigma_2^{-1} \sigma_1 \sigma_2^{-1}$. 
That this is a fiber surface immediately follows from the fact that it is the Murasugi sum of fiber surfaces of $(2,n_i)$-torus links, $1\leq i\leq N-1$, where $n_i$ is the number of occurrences of $\sigma_i$ in $\beta$, counted with sign. 

On such a fiber surface $F$, we choose a marking as in Figure \ref{fig:braided-Seifert-surface}. 
That is, we declare that the braid axis is a repelling elliptic flow loop $\gamma_e$ (of period $N$), and draw arcs on $F$, one for each band. 
In this setup, to determine the monodromy map $\varphi : F \rightarrow F$, we consider elastic cords in $S^3 \setminus F$ with boundary on $\partial F$, given by the positive push-offs of the arcs to the front side of $F$; for positive crossings, the elastic cord would be above the band, while for negative crossings, the elastic cord would be below the band. 
There is a natural way to pull each of these elastic cords, without crossing the braid axis $\gamma_e$, to split it into a number of other arcs. 
Given an elastic cord at a crossing, we look at the region of the plane delimited by the homogeneous braid closure diagram that is directly above (resp., below) the crossing if it is a positive (resp., negative) crossing; see Figure \ref{fig:pulling-elastic-cords}. 
The region is a polygon whose vertices correspond to some crossings. 
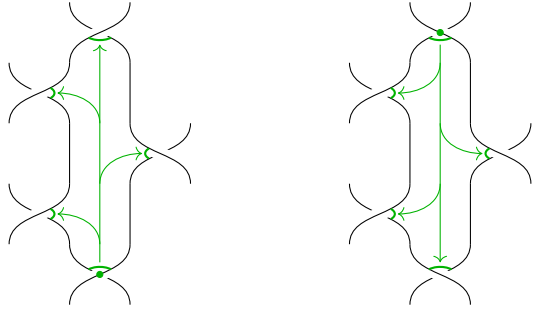
\begin{figure}
\centering
\[
\vcenter{\hbox{
\begin{tikzpicture}[scale=0.8]
\draw (1, 0) to[out=90, in=-90] (0, 1);
\draw[white, line width=5] (0, 0) to[out=90, in=-90] (1, 1);
\draw (0, 0) to[out=90, in=-90] (1, 1);
\begin{scope}[shift={(-1, 1)}]
    \draw (1, 0) to[out=90, in=-90] (0, 1);
    \draw[white, line width=5] (0, 0) to[out=90, in=-90] (1, 1);
    \draw (0, 0) to[out=90, in=-90] (1, 1);
\end{scope}
\begin{scope}[shift={(-1, 3)}]
    \draw (1, 0) to[out=90, in=-90] (0, 1);
    \draw[white, line width=5] (0, 0) to[out=90, in=-90] (1, 1);
    \draw (0, 0) to[out=90, in=-90] (1, 1);
\end{scope}
\begin{scope}[shift={(1, 2)}]
    \draw (0, 0) to[out=90, in=-90] (1, 1);
    \draw[white, line width=5] (1, 0) to[out=90, in=-90] (0, 1);
    \draw (1, 0) to[out=90, in=-90] (0, 1);
\end{scope}
\begin{scope}[shift={(0, 4)}]
    \draw (1, 0) to[out=90, in=-90] (0, 1);
    \draw[white, line width=5] (0, 0) to[out=90, in=-90] (1, 1);
    \draw (0, 0) to[out=90, in=-90] (1, 1);
\end{scope}
\draw (0, 2) -- (0, 3);
\draw (1, 1) -- (1, 2);
\draw (1, 3) -- (1, 4);

\filldraw[darkgreen] (0.5, 0.5) circle (0.05);

\draw[darkgreen, ->] (0.5, 0.7) -- (0.5, 4.3);
\draw[darkgreen, ->] (0.5, 1.0) to[out=90, in=0] (-0.2, 1.5);
\draw[darkgreen, ->] (0.5, 2.0) to[out=90, in=180] (1.2, 2.5);
\draw[darkgreen, ->] (0.5, 3.0) to[out=90, in=0] (-0.2, 3.5);

\draw[darkgreen, thick] (0.5, 0.5) [partial ellipse = 42 : 138 : 0.25 and 0.125];
\draw[darkgreen, thick] (-0.5, 1.5) [partial ellipse = -42 : 42 : 0.25 and 0.125];
\draw[darkgreen, thick] (-0.5, 3.5) [partial ellipse = -42 : 42 : 0.25 and 0.125];
\draw[darkgreen, thick] (1.5, 2.5) [partial ellipse = 138 : 222 : 0.25 and 0.125];
\draw[darkgreen, thick] (0.5, 4.5) [partial ellipse = -42 : -138 : 0.25 and 0.125];
\end{tikzpicture}
}}
\qquad\qquad
\vcenter{\hbox{
\begin{tikzpicture}[scale=0.8]
\begin{scope}[yscale=-1]
\draw (1, 0) to[out=90, in=-90] (0, 1);
\draw[white, line width=5] (0, 0) to[out=90, in=-90] (1, 1);
\draw (0, 0) to[out=90, in=-90] (1, 1);
\begin{scope}[shift={(-1, 1)}]
    \draw (0, 0) to[out=90, in=-90] (1, 1);
    \draw[white, line width=5] (1, 0) to[out=90, in=-90] (0, 1);
    \draw (1, 0) to[out=90, in=-90] (0, 1);
\end{scope}
\begin{scope}[shift={(-1, 3)}]
    \draw (0, 0) to[out=90, in=-90] (1, 1);
    \draw[white, line width=5] (1, 0) to[out=90, in=-90] (0, 1);
    \draw (1, 0) to[out=90, in=-90] (0, 1);
\end{scope}
\begin{scope}[shift={(1, 2)}]
    \draw (1, 0) to[out=90, in=-90] (0, 1);
    \draw[white, line width=5] (0, 0) to[out=90, in=-90] (1, 1);
    \draw (0, 0) to[out=90, in=-90] (1, 1);
\end{scope}
\begin{scope}[shift={(0, 4)}]
    \draw (1, 0) to[out=90, in=-90] (0, 1);
    \draw[white, line width=5] (0, 0) to[out=90, in=-90] (1, 1);
    \draw (0, 0) to[out=90, in=-90] (1, 1);
\end{scope}
\draw (0, 2) -- (0, 3);
\draw (1, 1) -- (1, 2);
\draw (1, 3) -- (1, 4);

\filldraw[darkgreen] (0.5, 0.5) circle (0.05);

\draw[darkgreen, ->] (0.5, 0.7) -- (0.5, 4.3);
\draw[darkgreen, ->] (0.5, 1.0) to[out=90, in=0] (-0.2, 1.5);
\draw[darkgreen, ->] (0.5, 2.0) to[out=90, in=180] (1.2, 2.5);
\draw[darkgreen, ->] (0.5, 3.0) to[out=90, in=0] (-0.2, 3.5);

\draw[darkgreen, thick] (0.5, 0.5) [partial ellipse = 42 : 138 : 0.25 and 0.125];
\draw[darkgreen, thick] (-0.5, 1.5) [partial ellipse = -42 : 42 : 0.25 and 0.125];
\draw[darkgreen, thick] (-0.5, 3.5) [partial ellipse = -42 : 42 : 0.25 and 0.125];
\draw[darkgreen, thick] (1.5, 2.5) [partial ellipse = 138 : 222 : 0.25 and 0.125];
\draw[darkgreen, thick] (0.5, 4.5) [partial ellipse = -42 : -138 : 0.25 and 0.125];
\end{scope}
\end{tikzpicture}
}}
\]
\caption{Pulling and splitting elastic cords}
\label{fig:pulling-elastic-cords}
\end{figure}
We can then pull the elastic cord across this polygon, splitting it into a number of elastic cords corresponding to the other vertices of the polygon, as shown in Figure \ref{fig:pulling-elastic-cords}.  
Note that the new elastic cords are at the back side of $F$, except for the ones which move to the right column, which are still at the front side of $F$. 
For the ones which move to the right column, we can repeat this process, and this process eventually terminates, as there are only $N-1$ columns in the braid. 
In this way, we can pull each of the elastic cords from the front side of $F$ to the back side of $F$, showing directly that $F$ is a fiber surface, according to the fiberedness criterion of \cite{BaaderGraf_fibered}.\footnote{For the purpose of getting a knot holder, we wouldn't need to repeat the process though, as they will lead to equivalent knot holders, under the second move in \cite[Fig. 2.4]{BirmanWilliams}, which is reproduced in Figure \ref{fig:knot-holder_equivalence}.} 
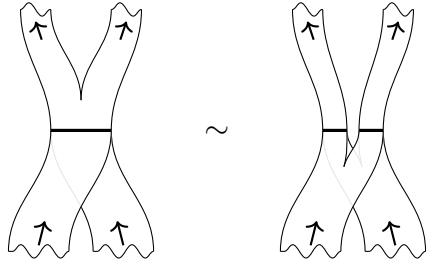
\begin{figure}
\centering
\[
\vcenter{\hbox{
\begin{tikzpicture}[scale=-0.8]
\draw (0, 0) to[out=90, in=-90] (-0.7, 2);
\draw (1, 0) to[out=90, in=-90] (0.3, 2);
\draw[snake it] (-0.7, 2) -- (0.3, 2);

\filldraw[white, opacity=0.9] (0, 0) to[out=90, in=-90] (0.7, 2) -- (1.7, 2) to[out=-90, in=90] (1, 0) -- cycle;
\draw (0, 0) to[out=90, in=-90] (0.7, 2);
\draw (1, 0) to[out=90, in=-90] (1.7, 2);
\draw[snake it] (0.7, 2) -- (1.7, 2);

\draw[very thick] (0, 0) -- (1, 0);

\draw (0, 0) to[out=-90, in=90] (-0.5, -2);
\draw (0.5, -0.5) to[out=-90, in=90] (0, -2);
\draw[snake it] (-0.5, -2) -- (0, -2);

\draw (1, 0) to[out=-90, in=90] (1.5, -2);
\draw (0.5, -0.5) to[out=-90, in=90] (1, -2);
\draw[snake it] (1.5, -2) -- (1, -2);

\draw[->, thick] (-0.2, 1.9) -- (-0.1, 1.5);
\draw[->, thick] (1.2, 1.9) -- (1.1, 1.5);

\draw[->, thick] (-0.15, -1.5) -- (-0.25, -1.8);
\draw[->, thick] (1.15, -1.5) -- (1.25, -1.8);
\end{tikzpicture}
}}
\quad\sim\quad
\vcenter{\hbox{
\begin{tikzpicture}[scale=-0.8]
\draw (0, 0) to[out=90, in=-90] (-0.7, 2);
\draw (1, 0) to[out=90, in=-90] (0.3, 2);
\draw[snake it] (-0.7, 2) -- (0.3, 2);

\draw (0.4, 0) to[out=90, in=-75] (0.35, 0.6);
\draw (0.6, 0) to[out=90, in=-75] (0.35, 0.6);

\filldraw[white, opacity=0.9] (0, 0) to[out=90, in=-90] (0.7, 2) -- (1.7, 2) to[out=-90, in=90] (1, 0) -- (0.6, 0) to[out=90, in=-105] (0.65, 0.6) to[out=-105, in=90] (0.4, 0) -- cycle;
\draw (0, 0) to[out=90, in=-90] (0.7, 2);
\draw (1, 0) to[out=90, in=-90] (1.7, 2);
\draw[snake it] (0.7, 2) -- (1.7, 2);

\draw (0.4, 0) to[out=90, in=-105] (0.65, 0.6);
\draw (0.6, 0) to[out=90, in=-105] (0.65, 0.6);

\draw[very thick] (0, 0) -- (0.4, 0);
\draw[very thick] (0.6, 0) -- (1, 0);

\draw (0, 0) to[out=-90, in=90] (-0.5, -2);
\draw (0.4, 0) to[out=-90, in=90] (0, -2);
\draw[snake it] (-0.5, -2) -- (0, -2);

\draw (1, 0) to[out=-90, in=90] (1.5, -2);
\draw (0.6, 0) to[out=-90, in=90] (1, -2);
\draw[snake it] (1.5, -2) -- (1, -2);

\draw[->, thick] (-0.2, 1.9) -- (-0.1, 1.5);
\draw[->, thick] (1.2, 1.9) -- (1.1, 1.5);

\draw[->, thick] (-0.15, -1.5) -- (-0.25, -1.8);
\draw[->, thick] (1.15, -1.5) -- (1.25, -1.8);
\end{tikzpicture}
}}
\]
\caption{Equivalent knot holders; they have the same periodic orbits.}
\label{fig:knot-holder_equivalence}
\end{figure}
What is more, since we have an explicit description of how the elastic cords are dragged around, by drawing the locus the elastic cords sweep around, we obtain a knot holder $\KH$ for the braid-homogeneous link that can be directly read off from the homogeneous braid; see Figure \ref{fig:braid-homogeneous-knot-holder}. 
The arcs $\KH \cap F$ (shown in green in the right of Figure \ref{fig:braid-homogeneous-knot-holder}) correspond to the crossings of the braid $\beta$, and the way the strips split and join directly models the way the elastic cords do.

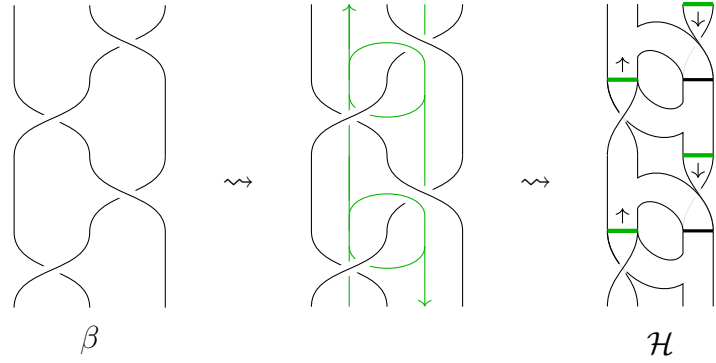
\begin{figure}
\centering
\[
\vcenter{\hbox{
\begin{tikzpicture}
\draw (1, 0) to[out=90, in=-90] (0, 1);
\draw[white, line width=5] (0, 0) to[out=90, in=-90] (1, 1);
\draw (0, 0) to[out=90, in=-90] (1, 1) to[out=90, in=-90] (2, 2);
\draw[white, line width=5] (2, 1) to[out=90, in=-90] (1, 2);
\draw (2, 0) -- (2, 1) to[out=90, in=-90] (1, 2) to[out=90, in=-90] (0, 3) -- (0, 4);
\draw[white, line width=5] (0, 2) to[out=90, in=-90] (1, 3);
\draw (0, 1) -- (0, 2) to[out=90, in=-90] (1, 3) to[out=90, in=-90] (2, 4);
\draw[white, line width=5] (2, 3) to[out=90, in=-90] (1, 4);
\draw (2, 2) -- (2, 3) to[out=90, in=-90] (1, 4);

\node[below] at (1, -0.1){$\beta$};
\end{tikzpicture}
}}
\quad\rightsquigarrow\quad
\vcenter{\hbox{
\begin{tikzpicture}
\draw[darkgreen, ->] (0.5, 0) -- (0.5, 4);
\draw[darkgreen, <-] (1.5, 0) -- (1.5, 4);
\draw[darkgreen] (1.5, 0.8) to[out=-90, in=-90] (0.5, 0.8);
\draw[darkgreen] (1.5, 2.8) to[out=-90, in=-90] (0.5, 2.8);
\draw (1, 0) to[out=90, in=-90] (0, 1);
\draw[white, line width=5] (0, 0) to[out=90, in=-90] (1, 1);
\draw (0, 0) to[out=90, in=-90] (1, 1) to[out=90, in=-90] (2, 2);
\draw[white, line width=5] (2, 1) to[out=90, in=-90] (1, 2);
\draw (2, 0) -- (2, 1) to[out=90, in=-90] (1, 2) to[out=90, in=-90] (0, 3) -- (0, 4);
\draw[white, line width=5] (0, 2) to[out=90, in=-90] (1, 3);
\draw (0, 1) -- (0, 2) to[out=90, in=-90] (1, 3) to[out=90, in=-90] (2, 4);
\draw[white, line width=5] (2, 3) to[out=90, in=-90] (1, 4);
\draw (2, 2) -- (2, 3) to[out=90, in=-90] (1, 4);
\draw[white, line width=5] (0.5, 1.2) to[out=90, in=90] (1.5, 1.2);
\draw[darkgreen] (0.5, 1.2) to[out=90, in=90] (1.5, 1.2);
\draw[white, line width=5] (0.5, 3.2) to[out=90, in=90] (1.5, 3.2);
\draw[darkgreen] (0.5, 3.2) to[out=90, in=90] (1.5, 3.2);
\draw[darkgreen] (0.5, 1) -- (0.5, 2);
\draw[darkgreen] (0.5, 3) -- (0.5, 4);

\node[below] at (1, -0.5){$\;$};
\end{tikzpicture}
}}
\quad\rightsquigarrow\quad
\vcenter{\hbox{
\begin{tikzpicture}
\draw (1.3, 0) -- (1.3, 0.3) to[out=-150, in=-90] (0.3, 1);
% \filldraw[white, opacity=0.8] (0.3, 0) to[out=90, in=-90] (0.7, 1) -- (0.3, 1) -- cycle;
\draw (1.7, 0) -- (1.7, 1);
\draw (0.7, 0) to[out=90, in=-90] (0.3, 1);
\draw[white, line width=4, opacity=0.9] (0.3, 0) to[out=90, in=-90] (0.7, 1);
\draw (0.3, 0) to[out=90, in=-90] (0.7, 1);
\draw (1.3, 1) -- (1.3, 0.7) to[out=-120, in=-90] (0.7, 1);

\draw (1.3, 1) to[out=90, in=-90] (1.7, 2);
\draw[white, line width=4, opacity=0.9] (1.7, 1) to[out=90, in=-90] (1.3, 2);
\draw (1.7, 1) to[out=90, in=-90] (1.3, 2);
\draw (0.3, 1) -- (0.3, 2);
\filldraw[white, opacity=0.9] (0.7, 1.7) to[out=30, in=90] (1.7, 1) -- (1.3, 1) to[out=90, in=60] (0.7, 1.3) -- cycle;
\draw (0.7, 1) -- (0.7, 1.3) to[out=60, in=90] (1.3, 1);
\draw (0.7, 2) -- (0.7, 1.7) to[out=30, in=90] (1.7, 1);
\begin{scope}[shift={(0, 2)}]
    \draw (1.3, 0) -- (1.3, 0.3) to[out=-150, in=-90] (0.3, 1);
    % \filldraw[white, opacity=0.8] (0.3, 0) to[out=90, in=-90] (0.7, 1) -- (0.3, 1) -- cycle;
    \draw (1.7, 0) -- (1.7, 1);
    \draw (0.7, 0) to[out=90, in=-90] (0.3, 1);
    \draw[white, line width=4, opacity=0.9] (0.3, 0) to[out=90, in=-90] (0.7, 1);
    \draw (0.3, 0) to[out=90, in=-90] (0.7, 1);
    \draw (1.3, 1) -- (1.3, 0.7) to[out=-120, in=-90] (0.7, 1);
    
    \draw (1.3, 1) to[out=90, in=-90] (1.7, 2);
    \draw[white, line width=4, opacity=0.9] (1.7, 1) to[out=90, in=-90] (1.3, 2);
    \draw (1.7, 1) to[out=90, in=-90] (1.3, 2);
    \draw (0.3, 1) -- (0.3, 2);
    \filldraw[white, opacity=0.9] (0.7, 1.7) to[out=30, in=90] (1.7, 1) -- (1.3, 1) to[out=90, in=60] (0.7, 1.3) -- cycle;
    \draw (0.7, 1) -- (0.7, 1.3) to[out=60, in=90] (1.3, 1);
    \draw (0.7, 2) -- (0.7, 1.7) to[out=30, in=90] (1.7, 1);
\end{scope}

\draw[ultra thick, darkgreen] (0.3, 1) -- (0.7, 1);
\draw[ultra thick, darkgreen] (0.3, 3) -- (0.7, 3);
\draw[very thick] (1.3, 1) -- (1.7, 1);
\draw[ultra thick, darkgreen] (1.3, 2) -- (1.7, 2);
\draw[very thick] (1.3, 3) -- (1.7, 3);
\draw[ultra thick, darkgreen] (1.3, 4) -- (1.7, 4);

\draw[->] (0.5, 1.1) -- (0.5, 1.3);
\draw[->] (0.5, 3.1) -- (0.5, 3.3);
\draw[->] (1.5, 1.9) -- (1.5, 1.7);
\draw[->] (1.5, 3.9) -- (1.5, 3.7);

\node[below] at (1, -0.2){$\KH$};
\end{tikzpicture}
}}
\]
\caption{Reading off a knot holder from a homogeneous braid. Left: a homogeneous braid. Middle: how the elastic cords move around. Right: the resulting knot holder.}
\label{fig:braid-homogeneous-knot-holder}
\end{figure}

\end{eg}

\begin{rmk}
There are other situations in which the construction above can be carried out purely combinatorially. 
For instance, a signed\footnote{Here, ``signed'' means each layered ideal tetrahedron carries a sign.} layered ideal triangulation of a mapping torus determines a movie of ideal triangulations on the fiber. 
We can take the edges of each ideal triangulation on the fiber as the arcs of a marking, with a base point at the center of each ideal triangle. 
The signed flips determine the corresponding splitting and joining moves (Figure~\ref{fig:flip-knot-holder}), and hence a canonical knot holder. 
\begin{figure}
\centering
\[
\vcenter{\hbox{
\begin{tikzpicture}
\draw[darkgreen, thick] (1, 1) -- (1, -1);
\draw[darkgreen, thick] (-1, 1) -- (-1, -1);
\draw[darkgreen, thick] (-1, 1) -- (1, 1);
\draw[darkgreen, thick] (-1, -1) -- (1, -1);
\draw[darkgreen, thick] (-1, 1) -- (1, -1);
\filldraw[white] (1, 1) circle (0.2);
\filldraw[white] (-1, 1) circle (0.2);
\filldraw[white] (-1, -1) circle (0.2);
\filldraw[white] (1, -1) circle (0.2);
\draw (1, 1) [partial ellipse = -90 : -180 : 0.2];
\draw (-1, 1) [partial ellipse = 0 : -90 : 0.2];
\draw (-1, -1) [partial ellipse = 0 : 90 : 0.2];
\draw (1, -1) [partial ellipse = 90 : 180 : 0.2];
\filldraw[red] (0.4, 0.4) circle (0.05);
\filldraw[red] (-0.4, -0.4) circle (0.05);
\end{tikzpicture}
}}
\overset{\text{isotopy}}{\rightsquigarrow}
\vcenter{\hbox{
\begin{tikzpicture}
\draw[darkgreen, thick] (1, 1) -- (1, -1);
\draw[darkgreen, thick] (-1, 1) -- (-1, -1);
\draw[darkgreen, thick] (-1, 1) -- (1, 1);
\draw[darkgreen, thick] (-1, -1) -- (1, -1);
\draw[darkgreen, thick] (-1, 1) .. controls (-0.4, -1.5) and (0.4, 1.5) .. (1, -1); 
\filldraw[white] (1, 1) circle (0.2);
\filldraw[white] (-1, 1) circle (0.2);
\filldraw[white] (-1, -1) circle (0.2);
\filldraw[white] (1, -1) circle (0.2);
\draw (1, 1) [partial ellipse = -90 : -180 : 0.2];
\draw (-1, 1) [partial ellipse = 0 : -90 : 0.2];
\draw (-1, -1) [partial ellipse = 0 : 90 : 0.2];
\draw (1, -1) [partial ellipse = 90 : 180 : 0.2];
\filldraw[red] (0.4, -0.4) circle (0.05);
\filldraw[red] (-0.4, 0.4) circle (0.05);
\end{tikzpicture}
}}
\overset{\text{split}}{\rightsquigarrow}
\vcenter{\hbox{
\begin{tikzpicture}
\draw[darkgreen, thick] (1, 1) -- (1, -1);
\draw[darkgreen, thick] (-1, 1) -- (-1, -1);
\draw[darkgreen, thick] (-1, 1) -- (1, 1);
\draw[darkgreen, thick] (-1, -1) -- (1, -1);
\draw[darkgreen, thick] (-1, -1) -- (1, 1); 
\draw[darkgreen, thick] (-1, 1) to[out=-70, in=70]  (-1, -1); 
\draw[darkgreen, thick] (1, 1) to[out=-110, in=110]  (1, -1); 
\filldraw[white] (1, 1) circle (0.2);
\filldraw[white] (-1, 1) circle (0.2);
\filldraw[white] (-1, -1) circle (0.2);
\filldraw[white] (1, -1) circle (0.2);
\draw (1, 1) [partial ellipse = -90 : -180 : 0.2];
\draw (-1, 1) [partial ellipse = 0 : -90 : 0.2];
\draw (-1, -1) [partial ellipse = 0 : 90 : 0.2];
\draw (1, -1) [partial ellipse = 90 : 180 : 0.2];
\filldraw[red] (0.4, -0.4) circle (0.05);
\filldraw[red] (-0.4, 0.4) circle (0.05);
\end{tikzpicture}
}}
\overset{\text{join}}{\rightsquigarrow}
\vcenter{\hbox{
\begin{tikzpicture}
\draw[darkgreen, thick] (1, 1) -- (1, -1);
\draw[darkgreen, thick] (-1, 1) -- (-1, -1);
\draw[darkgreen, thick] (-1, 1) -- (1, 1);
\draw[darkgreen, thick] (-1, -1) -- (1, -1);
\draw[darkgreen, thick] (-1, -1) -- (1, 1); 
\filldraw[white] (1, 1) circle (0.2);
\filldraw[white] (-1, 1) circle (0.2);
\filldraw[white] (-1, -1) circle (0.2);
\filldraw[white] (1, -1) circle (0.2);
\draw (1, 1) [partial ellipse = -90 : -180 : 0.2];
\draw (-1, 1) [partial ellipse = 0 : -90 : 0.2];
\draw (-1, -1) [partial ellipse = 0 : 90 : 0.2];
\draw (1, -1) [partial ellipse = 90 : 180 : 0.2];
\filldraw[red] (0.4, -0.4) circle (0.05);
\filldraw[red] (-0.4, 0.4) circle (0.05);
\end{tikzpicture}
}}
\]
\caption{Splitting and joining corresponding to a positive flip; for a negative flip, the two red base points braid negatively.}
\label{fig:flip-knot-holder}
\end{figure}

Another example comes from Rampichini diagrams \cite{MortonRampichini} for $P$-fibered links (i.e., braided open book decompositions of $S^3$ \cite{Bode}). 
A vertical slice of a Rampichini diagram determines a band word presentation and hence a braided fiber surface, and the movie of these band words gives a visual movie of markings. 
Applying the construction above produces a canonical knot holder associated to the Rampichini diagram. 
For braid-homogeneous links, this recovers the knot holder of Example~\ref{eg:braid-homogeneous-link-knot-holder}. 
In forthcoming work \cite{Park_P-fibered}, we use this construction to study Conjecture~\ref{conj:flow-equals-quantum} for fibered knots beyond the braid-homogeneous case. 
\end{rmk}

\subsection{Framing adapted to knot holders}\label{subsec:adapted-framing}
Now that we have seen how to obtain knot holders in practice, we would like to discuss how to compute the $\mathrm{GL}_1$-skein-valued flow loop count (Definition \ref{defn:skein-valued-flow-loop-count}) from the knot holders. 
As in Corollary \ref{cor:knot-holder-existence}, let $\KH \subset S^3 \setminus K$ be a knot holder supporting all the hyperbolic flow loops, and $\gamma_e \subset (S^3 \setminus K) \setminus \KH$ be the oriented link of elliptic flow loops. 
In order to count these flow loops in the $\mathrm{GL}_1$-skein module, we need to choose a framing, as in Section \ref{subsec:framing}. 
For the purpose of obtaining a simple state sum model, we would in particular want the framing induced on flow loops to be the same as the one from the tangent direction of the knot holder.\footnote{The knot holder may not be orientable, so here we are allowing half-twists; see Remark \ref{rmk:half-twist}. By shifting the framing lines off the knot holder, we can obtain genuine framings without half-twists, but that's less convenient, as it would be different from the one induced by the knot holder.} 

Such a framing can be constructed as follows. 
On a marked fiber surface $F$, we set the perturbation vector field $\xi_{\mathrm{pert}}$ to be the one coming from a Morse function associated to the marking on the surface, as in Remark \ref{rmk:marking-Morse-function}. 
That is, we pick an interior point $q_i$ from each arc $\alpha_i$ in the marking of $F$, and choose a Morse function on $F$ for which red base points $p_1, \cdots, p_m$ are the local maxima, the chosen interior points $q_1, \cdots, q_n$ are the saddle points, and the arc $\alpha_i$ is the descending manifold from the saddle $q_i$. 
Then, in the mapping cylinder $F\times [0, 1]$ of a monodromy $\varphi : F \rightarrow F$, $\{p_1, \cdots, p_m\} \times [0, 1]$ are the elliptic framing lines, and $\{q_1, \cdots, q_n\} \times [0, 1]$ are the hyperbolic framing lines. 

It remains to describe how to close up the mapping cylinder; how the framing lines behave under splitting and joining of the arcs. 
This is depicted in Figure \ref{fig:splitting-joining-with-framing-lines}. 
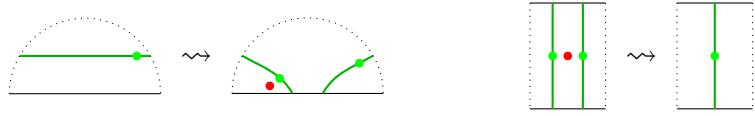
\begin{figure}
\centering
\[
\vcenter{\hbox{
\begin{tikzpicture}
\draw (-1, 0) -- (1, 0);
\draw[darkgreen, thick] ({-sqrt(3)/2}, 0.5) -- ({sqrt(3)/2}, 0.5);
\draw[dotted] (-1, 0) arc (180:0:1); 
\filldraw[green] (0.685, 0.5) circle (0.05);
\end{tikzpicture}
}}
\rightsquigarrow
\vcenter{\hbox{
\begin{tikzpicture}
\coordinate (a) at ({-sqrt(3)/2}, 0.5);
\coordinate (b) at ({sqrt(3)/2}, 0.5);
\draw (-1, 0) -- (1, 0);
\draw[darkgreen, thick] (a) to[out=-30, in=120] (-0.2, 0);
\draw[darkgreen, thick] (0.2, 0) to[out=60, in=-150] (b);
\draw[dotted] (-1, 0) arc (180:0:1); 
\filldraw[green] (0.685, 0.4) circle (0.05);
\filldraw[green] (-0.37, 0.2) circle (0.05);
\filldraw[red] (-0.5, 0.1) circle (0.05);
\end{tikzpicture}
}}
\qquad\qquad
\vcenter{\hbox{
\begin{tikzpicture}
\draw (-0.5, 0.7) -- (0.5, 0.7);
\draw (-0.5, -0.7) -- (0.5, -0.7);
\draw[dotted] (-0.5, 0.7) -- (-0.5, -0.7);
\draw[dotted] (0.5, 0.7) -- (0.5, -0.7);
\draw[darkgreen, thick] (-0.2, 0.7) -- (-0.2, -0.7);
\draw[darkgreen, thick] (0.2, 0.7) -- (0.2, -0.7);
\filldraw[green] (-0.2, 0) circle (0.05);
\filldraw[green] (0.2, 0) circle (0.05);
\filldraw[red] (0, 0) circle (0.05);
\end{tikzpicture}
}}
\rightsquigarrow
\vcenter{\hbox{
\begin{tikzpicture}
\draw (-0.5, 0.7) -- (0.5, 0.7);
\draw (-0.5, -0.7) -- (0.5, -0.7);
\draw[dotted] (-0.5, 0.7) -- (-0.5, -0.7);
\draw[dotted] (0.5, 0.7) -- (0.5, -0.7);
\draw[darkgreen, thick] (0, 0.7) -- (0, -0.7);
\filldraw[green] (0, 0) circle (0.05);
\end{tikzpicture}
}}
\]
\caption{Splitting (left) and joining (right), with framing lines (red: elliptic, green: hyperbolic)}
\label{fig:splitting-joining-with-framing-lines}
\end{figure}
\begin{itemize}
\item As an arc splits into two, a pair of a saddle point and a local maximum gets created, near the splitting point; see the left of Figure \ref{fig:splitting-joining-with-framing-lines}. 
Up to isotopy, the movie of critical points under splitting is uniquely determined as follows. 
After splitting, $F$ minus the arcs acquires an extra disk component which does not already contain a local maximum. 
A new local maximum is created in the interior of that disk component, and a new saddle point is created in the arc component neighboring that disk. 
The original saddle point on the arc follows along the other component of the arc after splitting. 
\item As two arcs join, two saddle points sandwich a local maximum and combine into a single saddle; see the right of Figure \ref{fig:splitting-joining-with-framing-lines}. 
\end{itemize}
The movie of these saddle points and local maxima determines the framing line
\[
\ell_{\mathrm{fr}} = \gamma_e \cup \gamma_h \cup \ell_{\mathrm{pair}}, 
\]
where $\gamma_e$ is the movie of the local maxima $\{p_1, \cdots, p_m\}$ (which coincide with the elliptic flow loops), $\gamma_h$ is the movie of saddle points $\{q_1, \cdots, q_n\}$, and $\ell_{\mathrm{pair}}$ is the movie of pairs of saddle points and local maxima that get created and annihilated under splitting and joining. 
Note that all the hyperbolic framing lines are on the knot holder, while all the elliptic framing lines are off the knot holder. 

Now that we have a perturbation vector field $\xi_{\mathrm{pert}}$ which is, away from the framing lines, everywhere tangent to the knot holder $\KH$, the framings induced on flow loops are the same as the ones induced by the knot holder. 
We can thus identify the $d$-fold multiple cover of any primitive hyperbolic flow loop $\gamma \subset \KH$ with $d$ parallel copies of $\gamma$ (parallel in the knot holder), in the $\mathrm{GL}_1$-skein module. 
This allows us to reformulate the $\mathrm{GL}_1$-skein-valued flow loop count in terms of a simple state sum model on the knot holder, which we describe below.

\subsection{State sum on knot holders}\label{subsec:knot-holder-state-sum}
Let $\KH$ be a knot holder, $\gamma_e \subset (S^3 \setminus K) \setminus \KH$ the elliptic flow loops, and $\ell_{\mathrm{fr}}$ the framing lines as above. 
Then, by a \emph{state}, we mean an assignment of $0$ or $1$ to each component of $\gamma_e$, together with an assignment of a nonnegative integer to each strip of the knot holder $\KH$, satisfying the admissibility condition that the sum of the numbers assigned to the incoming strips is equal to the sum of the numbers assigned to the outgoing strips, for any splitting or joining chart. 
Let $\mathcal{S}(\KH, \gamma_e)$ denote the set of all such states. 

Any multi-flow loop $\gamma \in \mathcal{O}$ (in the sense of Section \ref{subsec:setup}) determines a state $\sigma(\gamma) \in \mathcal{S}(\KH, \gamma_e)$, by counting the number of times $\gamma$ passes through each strip of the knot holder, and by recording which components of $\gamma_e$ are contained in $\gamma$. 
On the other hand, for any given state $\sigma \in \mathcal{S}(\KH, \gamma_e)$, there are only finitely many multi-flow loops $\gamma \in \mathcal{O}$ for which $\sigma(\gamma) = \sigma$. 
Such multi-flow loops can be easily determined. 
For each strip of the knot holder, if the assigned integer is $a \in \mathbb{Z}_{\geq 0}$, then we simply draw $a$ parallel strands on that strip, along the direction of the semiflow of the knot holder. 
These parallel strands can be uniquely connected across the splitting charts. 
The only non-trivial things can happen near the joining charts. 
If there are $b$ and $c$ incoming strands and $a = b+c$ outgoing strands in a joining chart, then there are $\binom{b+c}{b}$ number of ways to connect these strands across the joining chart. 
Once we have connected all the strands across all the joining charts, we get a multi-flow loop $\gamma \in \mathcal{O}$, and these are all the multi-flow loops with a given state. 

We can now repackage the flow loop count, which was defined in Definition \ref{defn:skein-valued-flow-loop-count} as a sum over all multi-flow loops $\gamma \in \mathcal{O}$, into a sum over all states:
\[
\Phi_{S^3 \setminus K}(x,q) = \sum_{\sigma \in \mathcal{S}(\KH, \gamma_e)} w(\sigma).
\]
Here, the weight $w(\sigma)$ is simply the sum of weights of all the multi-flow loops $\gamma$ for which $\sigma(\gamma) = \sigma$:
\[
w(\sigma) = 
\sum_{\substack{\gamma \in \mathcal{O} \\ \sigma(\gamma) = \sigma}} 
(-1)^{\mathrm{e}(\gamma) + \mathrm{nh}(\gamma)} 
q^{\mathrm{lk}(\gamma, \gamma) - \mathrm{lk}(\gamma, \ell_{\mathrm{fr}})}
x^{\deg \gamma}. 
\]
What makes this state sum particularly nice is that all the terms in the weight $w(\sigma)$ except for the self-linking number $\mathrm{lk}(\gamma,\gamma)$ depend only on the state $\sigma$, not on a particular multi-flow loop $\gamma$ representing it. 
Moreover, the sum $\sum_{\sigma(\gamma) = \sigma} q^{\mathrm{lk}(\gamma,\gamma)}$ is given by a product of $q$-binomials. 

More explicitly, the weight $w(\sigma)$ is determined by the product of the following local factors:
\begin{itemize}
\item Splitting:
\[
\vcenter{\hbox{
\begin{tikzpicture}[scale=0.8]
\begin{scope}[yscale=-1, xscale=-1]

\draw (0, 0.7) to[out=-90, in=90] (-0.5, -2);
\draw (0.5, -0.5) to[out=-100, in=90] (0, -2);
\draw[snake it] (-0.5, -2) -- (0, -2);

\draw (1, 0.7) to[out=-90, in=90] (1.5, -2);
\draw (0.5, -0.5) to[out=-80, in=90] (1, -2);
\draw[snake it] (1.5, -2) -- (1, -2);

\draw[snake it] (0, 0.7) -- (1, 0.7);

\draw[->, thick] (-0.15, -1.5) -- (-0.25, -1.8);
\draw[->, thick] (1.15, -1.5) -- (1.25, -1.8);

\draw[->, thick] (0.5, 0.6) -- (0.5, 0.3);

\node[below] at (0.5, 0.7){$a$};
\node[above] at (1.25, -2){$b$};
\node[above] at (-0.25, -2){$c$};
\end{scope}
\end{tikzpicture}
}}
\rightsquigarrow
\delta_{a, b+c}.
\]
That is, no extra factor from splitting charts, other than imposing the admissibility condition. 
\item Joining:
\[
\vcenter{\hbox{
\begin{tikzpicture}[scale=0.8]
\begin{scope}[yscale=-1, xscale=-1]
\draw (0, 0) to[out=90, in=-90] (-0.7, 2);
\draw (1, 0) to[out=90, in=-90] (0.3, 2);
\draw[snake it] (-0.7, 2) -- (0.3, 2);

\filldraw[white, opacity=0.9] (0, 0) to[out=90, in=-90] (0.7, 2) -- (1.7, 2) to[out=-90, in=90] (1, 0) -- cycle;
\draw (0, 0) to[out=90, in=-90] (0.7, 2);
\draw (1, 0) to[out=90, in=-90] (1.7, 2);
\draw[snake it] (0.7, 2) -- (1.7, 2);

\draw[very thick] (0, 0) -- (1, 0);

\draw (0, 0) -- (0, -1.3);
\draw (1, 0) -- (1, -1.3);
\draw[snake it] (0, -1.3) -- (1, -1.3);

\draw[->, thick] (-0.2, 1.9) -- (-0.1, 1.5);
\draw[->, thick] (1.2, 1.9) -- (1.1, 1.5);
\draw[->, thick] (0.5, -0.7) -- (0.5, -1.1);

\node[above] at (0.5, -1.3){$a$};
\node[below] at (1.2, 2){$b$};
\node[below] at (-0.2, 2){$c$};
\end{scope}
\end{tikzpicture}
}}
\rightsquigarrow
\delta_{a,b+c} \qbin{b+c}{b}_q
,\qquad
\vcenter{\hbox{
\begin{tikzpicture}[scale=0.8]
\begin{scope}[yscale=-1]
\draw (0, 0) to[out=90, in=-90] (-0.7, 2);
\draw (1, 0) to[out=90, in=-90] (0.3, 2);
\draw[snake it] (-0.7, 2) -- (0.3, 2);

\filldraw[white, opacity=0.9] (0, 0) to[out=90, in=-90] (0.7, 2) -- (1.7, 2) to[out=-90, in=90] (1, 0) -- cycle;
\draw (0, 0) to[out=90, in=-90] (0.7, 2);
\draw (1, 0) to[out=90, in=-90] (1.7, 2);
\draw[snake it] (0.7, 2) -- (1.7, 2);

\draw[very thick] (0, 0) -- (1, 0);

\draw (0, 0) -- (0, -1.3);
\draw (1, 0) -- (1, -1.3);
\draw[snake it] (0, -1.3) -- (1, -1.3);

\draw[->, thick] (-0.2, 1.9) -- (-0.1, 1.5);
\draw[->, thick] (1.2, 1.9) -- (1.1, 1.5);
\draw[->, thick] (0.5, -0.7) -- (0.5, -1.1);

\node[above] at (0.5, -1.3){$a$};
\node[below] at (1.2, 2){$c$};
\node[below] at (-0.2, 2){$b$};
\end{scope}
\end{tikzpicture}
}}
\rightsquigarrow
\delta_{a,b+c} \qbin{b+c}{b}_{q^{-1}},
\]
where the $q$-binomial coefficients are defined by
\[
\qbin{n}{k}_q := \frac{[n]_q!}{[k]_q![n-k]_q!},\quad [n]_q! := [n]_q [n-1]_q \cdots [1]_q,\quad [n]_q := \frac{1-q^n}{1-q}.
\]
Such $q$-binomial factors arise from summing $q^{\mathrm{lk}(\gamma,\gamma)}$ over all possible ways to connect the strands across the joining charts. 
Note that we are using a particular planar diagram of the knot holder to determine whether the crossings are positive or negative under joining. 
\item Crossings and half-twists: 
\[
\vcenter{\hbox{
\begin{tikzpicture}
\draw (1, 0) to[out=90, in=-90] (-0.3, 2);
\draw (0.3, 0) to[out=90, in=-90] (-1, 2);
\filldraw[white, opacity=0.9] (-1, 0) to[out=90, in=-90] (0.3, 2) -- (1, 2) to[out=-90, in=90] (-0.3, 0) -- cycle;
\draw (-1, 0) to[out=90, in=-90] (0.3, 2);
\draw (-0.3, 0) to[out=90, in=-90] (1, 2);
\draw[snake it] (-1, 0) -- (-0.3, 0);
\draw[snake it] (1, 0) -- (0.3, 0);
\draw[snake it] (-1, 2) -- (-0.3, 2);
\draw[snake it] (1, 2) -- (0.3, 2);

\draw[->, thick] (0.6, 0.2) -- (0.5, 0.5);
\draw[->, thick] (-0.6, 0.2) -- (-0.5, 0.5);
\draw[<-, thick] (0.6, 1.8) -- (0.5, 1.5);
\draw[<-, thick] (-0.6, 1.8) -- (-0.5, 1.5);

\node[above] at (0.65, 2){$a$};
\node[above] at (-0.65, 2){$b$};
\node[below] at (0.65, 0){$b$};
\node[below] at (-0.65, 0){$a$};
\end{tikzpicture}
}}
\rightsquigarrow
q^{ab}
,\qquad
\vcenter{\hbox{
\begin{tikzpicture}[xscale=-1]
\draw (1, 0) to[out=90, in=-90] (-0.3, 2);
\draw (0.3, 0) to[out=90, in=-90] (-1, 2);
\filldraw[white, opacity=0.9] (-1, 0) to[out=90, in=-90] (0.3, 2) -- (1, 2) to[out=-90, in=90] (-0.3, 0) -- cycle;
\draw (-1, 0) to[out=90, in=-90] (0.3, 2);
\draw (-0.3, 0) to[out=90, in=-90] (1, 2);
\draw[snake it] (-1, 0) -- (-0.3, 0);
\draw[snake it] (1, 0) -- (0.3, 0);
\draw[snake it] (-1, 2) -- (-0.3, 2);
\draw[snake it] (1, 2) -- (0.3, 2);

\draw[->, thick] (0.6, 0.2) -- (0.5, 0.5);
\draw[->, thick] (-0.6, 0.2) -- (-0.5, 0.5);
\draw[<-, thick] (0.6, 1.8) -- (0.5, 1.5);
\draw[<-, thick] (-0.6, 1.8) -- (-0.5, 1.5);

\node[above] at (0.65, 2){$b$};
\node[above] at (-0.65, 2){$a$};
\node[below] at (0.65, 0){$a$};
\node[below] at (-0.65, 0){$b$};
\end{tikzpicture}
}}
\rightsquigarrow
q^{-ab},
\]
\[
\vcenter{\hbox{
\begin{tikzpicture}
\draw (0.7, 0) to[out=90, in=-90] (0, 2);
\draw[white, line width=5] (0, 0) to[out=90, in=-90] (0.7, 2);
\draw (0, 0) to[out=90, in=-90] (0.7, 2);
\draw[snake it] (0, 0) -- (0.7, 0);
\draw[snake it] (0, 2) -- (0.7, 2);
\draw[->, thick] (0.35, 0.2) -- (0.35, 0.5);
\draw[->, thick] (0.35, 1.5) -- (0.35, 1.8);

\node[above] at (0.35, 2){$a$};
\node[below] at (0.35, 0){$a$};
\end{tikzpicture}
}}
\rightsquigarrow (-1)^a q^{\frac{a^2}{2}}
,\qquad
\vcenter{\hbox{
\begin{tikzpicture}[xscale=-1]
\draw (0.7, 0) to[out=90, in=-90] (0, 2);
\draw[white, line width=5] (0, 0) to[out=90, in=-90] (0.7, 2);
\draw (0, 0) to[out=90, in=-90] (0.7, 2);
\draw[snake it] (0, 0) -- (0.7, 0);
\draw[snake it] (0, 2) -- (0.7, 2);
\draw[->, thick] (0.35, 0.2) -- (0.35, 0.5);
\draw[->, thick] (0.35, 1.5) -- (0.35, 1.8);

\node[above] at (0.35, 2){$a$};
\node[below] at (0.35, 0){$a$};
\end{tikzpicture}
}}
\rightsquigarrow (-1)^a q^{-\frac{a^2}{2}}.
\]
The powers of $q$ come from the self-linking numbers of $\gamma$, and the signs $(-1)^a$ in the half-twists come from the fact that negative hyperbolic flow loops are counted with signs (i.e., $(-1)^{\mathrm{nh}(\sigma)}$). 
\item Elliptic flow loops: Of course, the factors described above only account for the linking numbers among hyperbolic flow loops, so there is also a monomial factor coming from the linking numbers with the elliptic flow loops specified by $\sigma$, together with the sign $(-1)^{\mathrm{e}(\sigma)}$. 
\item Degree: The factor $x^{\deg \sigma}$ can be easily determined, as $\deg \sigma$ is simply the sum of the numbers on the strips corresponding to the arcs $\KH \cap F$, plus the sum of the degrees of the elliptic flow loops used in $\sigma$. 
\item Linking with the framing lines:
Now it remains to describe how to determine the framing factor $q^{-\mathrm{lk}(\sigma, \ell_{\mathrm{fr}})}$. 
Recall that $\ell_{\mathrm{fr}} = \gamma_e \cup \gamma_h \cup \ell_{\mathrm{pair}}$, where $\ell_{\mathrm{pair}}$ is the locus of pairs of saddles and local maxima that get created under splitting and annihilated under joining. 
Whenever such a pair is created, it is created near one of the inner boundaries of the splitting chart, so we use the convention to keep them near one of the two boundaries of the strips, though we allow them to slide from one boundary to the other occasionally. 
Whenever it slides, we pick up a linking factor, from the linking number with the elliptic component of the pair:
\[
\vcenter{\hbox{
\begin{tikzpicture}
\draw[snake it] (0, 0) -- (0.7, 0);
\draw[snake it] (0, 2) -- (0.7, 2);
\draw (0, 0) -- (0, 2);
\draw (0.7, 0) -- (0.7, 2);
\draw[->, thick] (0.35, 0.2) -- (0.35, 0.5);
\draw[->, thick] (0.35, 1.5) -- (0.35, 1.8);

\draw[green] (0.2, 0) to[out=90, in=-90] (0.5, 2);
\draw[red] (0.1, 0) to[out=90, in=-90] (0.6, 2);

\node[above] at (0.35, 2){$a$};
\node[below] at (0.35, 0){$a$};
\end{tikzpicture}
}}
\rightsquigarrow
q^{-\frac{a}{2}}
,\qquad
\vcenter{\hbox{
\begin{tikzpicture}
\draw[red] (0.1, 0) to[out=90, in=-90] (0.6, 2);

\filldraw[white, opacity=0.8] (0, 0) -- (0, 2) -- (0.7, 2) -- (0.7, 0) -- cycle;

\draw[snake it] (0, 0) -- (0.7, 0);
\draw[snake it] (0, 2) -- (0.7, 2);
\draw (0, 0) -- (0, 2);
\draw (0.7, 0) -- (0.7, 2);
\draw[->, thick] (0.35, 0.2) -- (0.35, 0.5);
\draw[->, thick] (0.35, 1.5) -- (0.35, 1.8);

\draw[green] (0.2, 0) to[out=90, in=-90] (0.5, 2);

\node[above] at (0.35, 2){$a$};
\node[below] at (0.35, 0){$a$};
\end{tikzpicture}
}}
\rightsquigarrow
q^{\frac{a}{2}},
\]
where on the left (resp., right) the elliptic component of the pair, shown in red, is above (resp., below) the strip. 
Whenever the pair gets annihilated at a joining, by sliding it to the other side if necessary, we make sure that the pair starts from the boundary of the strip that is closer to the other strip (in the planar diagram) that gets joined together. 
This is to ensure that there is no extra framing factor coming from such joining. 
For instance, 
\[
\vcenter{\hbox{
\begin{tikzpicture}[scale=0.8]
\begin{scope}[yscale=-1, xscale=-1]
\draw (0, 0) to[out=90, in=-90] (-0.7, 2);
\draw (1, 0) to[out=90, in=-90] (0.3, 2);
\draw[snake it] (-0.7, 2) -- (0.3, 2);

\draw[green] (0.5, 0) -- (0.5, -1.2);
\draw[green] (0.5, 0) to[out=90, in=-90] (-0.2, 2);

\draw[red] (0.5, 0) to[out=90, in=-90] (0.8, 2);

\filldraw[white, opacity=0.8] (0, 0) to[out=90, in=-90] (0.7, 2) -- (1.7, 2) to[out=-90, in=90] (1, 0) -- cycle;
\draw (0, 0) to[out=90, in=-90] (0.7, 2);
\draw (1, 0) to[out=90, in=-90] (1.7, 2);

\draw[green] (0.5, 0) to[out=90, in=-90] (0.9, 2);

\draw[snake it] (0.7, 2) -- (1.7, 2);

\draw[very thick] (0, 0) -- (1, 0);

\draw (0, 0) -- (0, -1.3);
\draw (1, 0) -- (1, -1.3);
\draw[snake it] (0, -1.3) -- (1, -1.3);

\draw[->, thick] (-0.2, 1.9) -- (-0.1, 1.5);
\draw[->, thick] (1.2, 1.9) -- (1.1, 1.5);
\draw[->, thick] (0.5, -0.7) -- (0.5, -1.1);

\node[above] at (0.5, -1.3){$a$};
\node[below] at (1.2, 2){$b$};
\node[below] at (-0.2, 2){$c$};
\end{scope}
\end{tikzpicture}
}}
\rightsquigarrow
\delta_{a,b+c} \qbin{b+c}{b}_q,
\]
since there is no extra framing factor in this picture. 
If the pair was starting from the other boundary of the strip labeled $b$, then there would be an extra factor of $q^{\frac{b}{2}}$ coming from sliding the pair from left to right. 
Finally, the linking numbers with $\gamma_e$ and $\gamma_h$ can be easily determined. 
\end{itemize}

Below, we illustrate our state sum model on knot holders through examples. 
\begin{eg}[State sum on trefoil knot holders]\label{eg:trefoil-state-sum}
\begin{figure}
\centering
\[
\vcenter{\hbox{
\begin{tikzpicture}[scale=1.2]
\draw[red] (-2.5, 1) circle (0.2);
\node[red, left] at (-2.7, 1){$\gamma_e$};

\draw[green] (0.25, 0) -- (0.25, 1) to[out=90, in=-90] (0.375, 2) to[out=90, in=0] (-0.625, 3.125) to[out=180, in=90] (-1.675, 2) to[out=-90, in=180] (-0.7, -0.75) to[out=0, in=-90] cycle;
\node[green, left] at (-1.8, 1.7){$\gamma_h$};

\draw[red] (0.25, 2) to[out=90, in=0] (-0.5, 2.7) to[out=180, in=90] (-1.2, 2) to[out=-90, in=180] (-0.5, 1.3) to[out=0, in=-90] (0.375, 2);

\draw (0.5, 0) to[out=90, in=-90] (0, 2);
\draw[white, line width=5] (0, 0) to[out=90, in=-90] (0.5, 2);
\draw (0, 0) to[out=90, in=-90] (0.5, 2);

\filldraw[white, opacity=0.8] (0.25, 2) to[out=90, in=0] (-0.5, 2.75) to[out=180, in=90] (-1.25, 2) to[out=-90, in=180] (-0.5, 1.25) to[out=0, in=-90] (0.5, 2) -- cycle;
\draw (0, 2) to[out=90, in=0] (-0.5, 2.5) to[out=180, in=90] (-1, 2) to[out=-90, in=180] (-0.5, 1.5) to[out=0, in=-90] (0, 2);
\draw (0.25, 2) to[out=90, in=0] (-0.5, 2.75) to[out=180, in=90] (-1.25, 2) to[out=-90, in=180] (-0.5, 1.25) to[out=0, in=-90] (0.5, 2);

\draw (0.25, 2) to[out=90, in=0] (-0.625, 3) to[out=180, in=90] (-1.5, 2) to[out=-90, in=180] (-0.7, -0.5) to[out=0, in=-90] (0, 0);
\draw (0.5, 2) to[out=90, in=0] (-0.625, 3.25) to[out=180, in=90] (-1.8, 2) to[out=-90, in=180] (-0.7, -1.0) to[out=0, in=-90] (0.5, 0);

\draw[green] (0.25, 2) to[out=90, in=0] (-0.5, 2.65) to[out=180, in=90] (-1.15, 2) to[out=-90, in=180] (-0.5, 1.35) to[out=0, in=-90] (0.375, 2);

\draw[darkgreen, ultra thick] (0, 0) -- (0.5, 0);
\draw[darkgreen, ultra thick] (0, 2) -- (0.5, 2);

\node at (0.25, 0.5){$a$};
\node at (-0.1, 1.5){$b$};
\node at (0.31, 2.4){$a$};
\node at (-0.01, 2.4){$b$};

\draw[->] (0.25, 0.1) -- (0.25, 0.3);
\draw[->] (0.11, 2.1) -- (0.07, 2.3);
\draw[->] (0.37, 2.1) -- (0.36, 2.3);
\end{tikzpicture}
}}
\qquad\qquad\qquad
\vcenter{\hbox{
\begin{tikzpicture}[scale=0.8, 
    mid/.style={
        postaction={
            decorate,
            decoration={
                markings,
                % Mark at position 0.5 (the middle) with an arrow tip
                mark=at position 0.5 with {\arrow{>}} 
            }
        }
    },
    midback/.style={
        postaction={
            decorate,
            decoration={
                markings,
                % Mark at position 0.5 (the middle) with an arrow tip
                mark=at position 0.5 with {\arrow{<}} 
            }
        }
    },
]

\draw[green] (0.25, 0) -- (0.25, 3) to[out=90, in=180] (1, 3.75) to[out=0, in=90] (1.75, 3) -- (1.75, 0) to[out=-90, in=0] (1, -0.75) to[out=180, in=-90] cycle;

\draw (0.5, 0) to[out=90, in=-90] (0, 1);
\draw[white, line width=4, opacity=0.9] (0, 0) to[out=90, in=-90] (0.5, 1);
\draw (0, 0) to[out=90, in=-90] (0.5, 1);

\begin{scope}[shift={(0, 1)}]
    \draw (0.5, 0) to[out=90, in=-90] (0, 1);
    \draw[white, line width=4, opacity=0.9] (0, 0) to[out=90, in=-90] (0.5, 1);
    \draw (0, 0) to[out=90, in=-90] (0.5, 1);
\end{scope}

\begin{scope}[shift={(0, 2)}]
    \draw (0.5, 0) to[out=90, in=-90] (0, 1);
    \draw[white, line width=4, opacity=0.9] (0, 0) to[out=90, in=-90] (0.5, 1);
    \draw (0, 0) to[out=90, in=-90] (0.5, 1);
\end{scope}
\draw[ultra thick, darkgreen] (0, 1) -- (0.5, 1);
\draw[->] (0.25, 1.1) -- (0.25, 1.3);
\draw[ultra thick, darkgreen] (0, 2) -- (0.5, 2);
\draw[->] (0.25, 2.1) -- (0.25, 2.3);
\draw[ultra thick, darkgreen] (0, 3) -- (0.5, 3);
\draw[->] (0.25, 3.1) -- (0.25, 3.3);

\draw (0.5, 3) to[out=90, in=180] (1, 3.5) to[out=0, in=90] (1.5, 3) -- (1.5, 0) to[out=-90, in=0] (1, -0.5) to[out=180, in=-90] (0.5, 0);
\draw (0, 3) to[out=90, in=180] (1, 4) to[out=0, in=90] (2, 3) -- (2, 0) to[out=-90, in=0] (1, -1) to[out=180, in=-90] (0, 0);

\draw[white, line width=4] (1.75, 1.5) [partial ellipse = 40 : -220 : 0.6 and 0.3];
\draw[red, mid] (1.75, 1.5) [partial ellipse = 50 : -230 : 0.6 and 0.3];
\node[red, right] at (2.35, 1.5){$\gamma_e$};
\node[green, right] at (2, 2.5){$\gamma_h$};

\node at (0.25, 0){$a$};

\end{tikzpicture}
}}
\]
\caption{State sums on two knot holders for the trefoil knot}
\label{fig:trefoil-state-sums}
\end{figure}
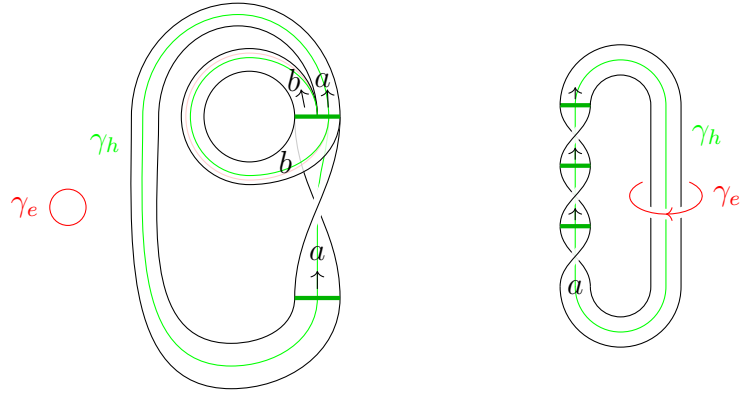
Let us consider state sums on two trefoil knot holders, one from Example \ref{eg:trefoil-direct-model}, and the other from Example \ref{eg:braid-homogeneous-link-knot-holder}, as shown in Figure \ref{fig:trefoil-state-sums}. 
The state sum on the knot holder on the left of Figure \ref{fig:trefoil-state-sums} results in the following formula for the flow loop count:
\begin{align*}
\Phi_{S^3 \setminus \mathbf{3}_1^r} (x,q)
&= 
(1-x) \sum_{a,b \geq 0} (-1)^{a} q^{\frac{a^2}{2}} \qbin{a+b}{a}_q x^{a + (a+b)} \cdot q^{\frac{a}{2}} \\
&= 1 - q x^2 - q^2 x^3 + q^5 x^5 + q^7 x^6 - q^{12} x^8 - q^{15} x^9 + O(x^{11}),
\end{align*}
where the overall factor of $(1-x)$ comes from the elliptic flow loop along the meridian, and the factor $q^{\frac{a}{2}}$ comes from the linking with the framing line $\gamma_h$. 

The state sum on the knot holder on the right of Figure \ref{fig:trefoil-state-sums} gives:
\begin{align*}
\Phi_{S^3 \setminus \mathbf{3}_1^r}(x,q) &= 
\sum_{\substack{a\geq 0 \\ 0\leq \epsilon \leq 1}}
(-1)^{a} q^{\frac{3 a^2}{2} + 2\epsilon a} x^{3a} (-x^2)^{\epsilon} \cdot q^{\frac{3a}{2} + \epsilon -a } \\
&= 1 - q x^2 - q^2 x^3 + q^5 x^5 + q^7 x^6 - q^{12} x^8 - q^{15} x^9 + O(x^{11}),
\end{align*}
where $\epsilon$ counts the elliptic flow loop $\gamma_e$, and the factor 
$q^{\frac{3a}{2} + \epsilon}$ (resp., $q^{-a}$) comes from the linking with the framing line $\gamma_h$ (resp., $\gamma_e$). 

Observe that they give the same count, which follows from Theorem \ref{mainthm:flow-loop-count}. 
Moreover, they agree with the BPS $q$-series for the trefoil knot complement \cite{GukovManolescu, Park_large-color}
\[
-q^{-1}x^{-\frac12} \widehat{Z}_{S^3 \setminus \mathbf{3}_1^r}(x,q) = 1 - q x^2 - q^2 x^3 + q^5 x^5 + q^7 x^6 - q^{12} x^8 - q^{15} x^9 + O(x^{11}),
\]
which is a special case of Theorem \ref{mainthm:flow-equals-quantum}. 
\end{eg}

\begin{eg}[State sum on figure-eight knot holders]\label{eg:figure-eight-state-sum}
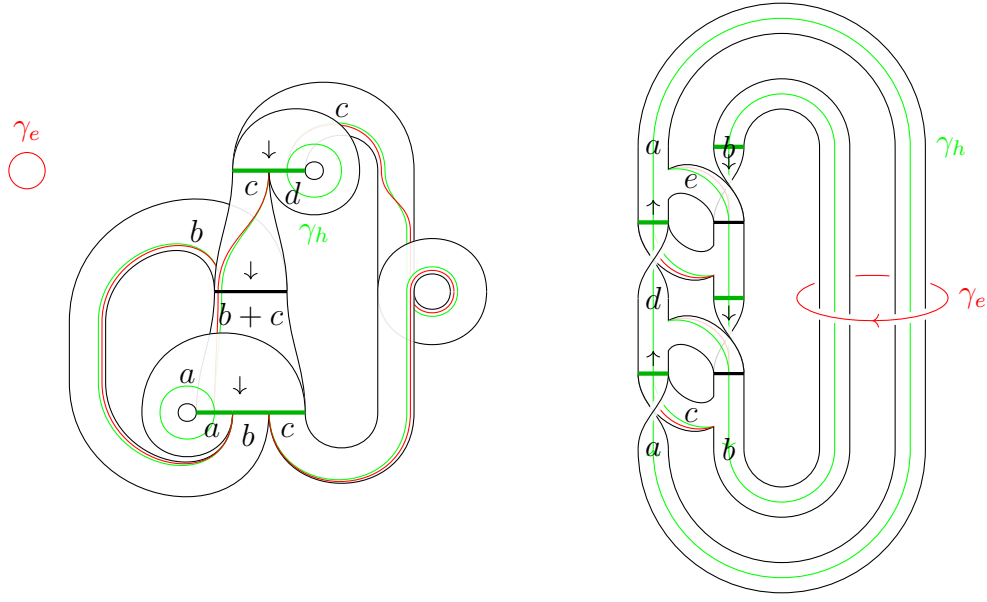
\begin{figure}
\centering
\[
\vcenter{\hbox{
\begin{tikzpicture}[scale=0.8]
\begin{scope}[xscale=0.6]

\draw[green] (0.7, 2) to[out=-90, in=90] (0.5, 0);
\draw[red] (0.6, 2) to[out=-90, in=90] (0.5, 0);

\draw (0.5, 2) to[out=-90, in=90] (0, 0);
\draw (2.5, 2) to[out=-90, in=90] (3, 0);
\filldraw[white, opacity=0.9] (-1.5, 0) to[out=90, in=90] (3, 0) -- (0, 0) to[out=90, in=90] (-0.5, 0) -- cycle;
\draw (0, 0) to[out=-90, in=-90] (-0.5, 0) to[out=90, in=90] (0, 0);
\draw (1, 0) to[out=-90, in=-90] (-1.5, 0) to[out=90, in=90] (3, 0);
\draw (1, 0) to[out=-90, in=-90] (-2.5, 0.5) -- (-2.5, 1.5) to[out=90, in=90] (0.5, 2);
\draw (2, 0) to[out=-90, in=-90] (-3.5, 0.5) -- (-3.5, 1.5) to[out=90, in=90] (2.5, 2);

\draw[green] (1, 0) to[out=-92, in=-90] (-2.7, 0.5) -- (-2.7, 1.5) to[out=90, in=90] (0.7, 2);
\draw[red] (1, 0) to[out=-91, in=-90] (-2.6, 0.5) -- (-2.6, 1.5) to[out=90, in=90] (0.7, 2);

\filldraw[white, opacity=0.9] (1, 4) to[out=-90, in=90] (0.5, 2) -- (2.5, 2) to[out=90, in=-90] (2, 4) -- cycle;
\draw (1, 4) to[out=-90, in=90] (0.5, 2);
\draw (2, 4) to[out=-90, in=90] (2.5, 2);
\draw (1, 4) to[out=90, in=90] (6, 4) -- (6, 2) to[out=-90, in=-90] (7, 2);
\draw (3, 4) to[out=90, in=90] (5, 4) -- (5, 2) to[out=-90, in=-90] (8, 2);

\draw[green] (2, 4) to[out=-90, in=90] (0.7, 2);
\draw[red] (2, 4) to[out=-90, in=90] (0.6, 2);

\draw[green] (7.2, 2) to[out=-90, in=-90] (5.8, 2) -- (5.8, 3) to[out=90, in=-90] (5.2, 4) to[out=90, in=90] (2.5, 4);
\draw[red] (7.1, 2) to[out=-90, in=-90] (5.9, 2) -- (5.9, 3) to[out=90, in=-90] (5.1, 4) to[out=90, in=90] (2.5, 4);

\filldraw[white, opacity=0.9] (5, 0) -- (5, 2) to[out=90, in=90] (8, 2) -- (7, 2) to[out=90, in=90] (6, 2) -- (6, 0);
\draw (3, 0) to[out=-90, in=-90] (5, 0) -- (5, 2) to[out=90, in=90] (8, 2);
\draw (2, 0) to[out=-90, in=-90] (6, 0) -- (6, 2) to[out=90, in=90] (7, 2);
\filldraw[white, opacity=0.9] (1, 4) to[out=90, in=90] (4.5, 4) -- (3.5, 4) to[out=90, in=90] (3, 4) -- cycle;
\draw (3, 4) to[out=90, in=90] (3.5, 4) to[out=-90, in=-90] (3, 4);
\draw (1, 4) to[out=90, in=90] (4.5, 4) to[out=-90, in=-90] (2, 4);

\draw[green] (2, 0) to[out=-88, in=-90] (5.8, 0) -- (5.8, 2) to[out=90, in=90] (7.2, 2);
\draw[red] (2, 0) to[out=-89, in=-90] (5.9, 0) -- (5.9, 2) to[out=90, in=90] (7.1, 2);

\draw[green] (0.5, 0) to[out=90, in=90] (-1, 0) to[out=-90, in=-90] cycle;
\draw[green] (2.5, 4) to[out=90, in=90] (4, 4) to[out=-90, in=-90] cycle;
\node[green, below] at (3.25, 3.3){$\gamma_h$};

\draw[darkgreen, ultra thick] (0, 0) -- (1, 0);
\draw[darkgreen, ultra thick] (1, 0) -- (2, 0);
\draw[darkgreen, ultra thick] (2, 0) -- (3, 0);
\draw[very thick] (0.5, 2) -- (2.5, 2);
\draw[darkgreen, ultra thick] (1, 4) -- (3, 4);

\draw[->] (1.2, 0.6) -- (1.2, 0.3);
\draw[->] (1.5, 2.5) -- (1.5, 2.2);
\draw[->] (2, 4.5) -- (2, 4.2);

\node[below] at (0.43, 0){$a$};
\node[below] at (1.43, 0){$b$};
\node[below] at (2.5, 0){$c$};
\node at (-0.25, 0.6){$a$};
\node[below] at (1.5, 2){$b+c$};
\node at (0, 3){$b$};
\node[below] at (1.5, 4){$c$};
\node[below] at (2.65, 4){$d$};
\node at (4, 5){$c$};

\end{scope}
\draw[red] (-2.8, 4) circle (0.3);
\node[red, above] at (-2.8, 4.3){$\gamma_e$};
\end{tikzpicture}
}}
\qquad\qquad
\vcenter{\hbox{
\begin{tikzpicture}[
    mid/.style={
        postaction={
            decorate,
            decoration={
                markings,
                % Mark at position 0.5 (the middle) with an arrow tip
                mark=at position 0.5 with {\arrow{>}} 
            }
        }
    },
    midback/.style={
        postaction={
            decorate,
            decoration={
                markings,
                % Mark at position 0.5 (the middle) with an arrow tip
                mark=at position 0.5 with {\arrow{<}} 
            }
        }
    },
]
\draw (1.3, 0) -- (1.3, 0.3) to[out=-150, in=-90] (0.3, 1);
% \filldraw[white, opacity=0.8] (0.3, 0) to[out=90, in=-90] (0.7, 1) -- (0.3, 1) -- cycle;

\draw[green] (1.3, 0.3) to[out=-170, in=-90] (0.5, 1);
\draw[red] (1.3, 0.3) to[out=-160, in=-90] (0.5, 0.8) -- (0.5, 1);
\filldraw[white, opacity=0.8] (0.3, 0) to[out=90, in=-90] (0.7, 1) -- (0.3, 1) to[out=-90, in=90] (0.7, 0) -- cycle;

\draw[green, ->] (0.5, 0) -- (0.5, 4) to[out=90, in=180] (2.2, 5.7) to[out=0, in=90] (3.9, 4) -- (3.9, 0) to[out=-90, in=0] (2.2, -1.7) to[out=180, in=-90] (0.5, 0);
\draw[green, <-] (1.5, 0) -- (1.5, 4) to[out=90, in=180] (2.2, 4.7) to[out=0, in=90] (2.9, 4) -- (2.9, 0) to[out=-90, in=0] (2.2, -0.7) to[out=180, in=-90] (1.5, 0);

\draw (1.7, 0) -- (1.7, 1);
\draw (0.7, 0) to[out=90, in=-90] (0.3, 1);
\draw[white, line width=4, opacity=0.9] (0.3, 0) to[out=90, in=-90] (0.7, 1);
\draw (0.3, 0) to[out=90, in=-90] (0.7, 1);
\draw (1.3, 1) -- (1.3, 0.7) to[out=-120, in=-90] (0.7, 1);

\draw (1.3, 1) to[out=90, in=-90] (1.7, 2);
\draw[white, line width=4, opacity=0.9] (1.7, 1) to[out=90, in=-90] (1.3, 2);
\draw (1.7, 1) to[out=90, in=-90] (1.3, 2);
\draw (0.3, 1) -- (0.3, 2);

\draw[red] (0.7, 1.7) to[out=20, in=90] (1.5, 1.2) -- (1.5, 1);

\filldraw[white, opacity=0.8] (0.7, 1.7) to[out=30, in=90] (1.7, 1) -- (1.3, 1) to[out=90, in=60] (0.7, 1.3) -- cycle;
\draw (0.7, 1) -- (0.7, 1.3) to[out=60, in=90] (1.3, 1);
\draw (0.7, 2) -- (0.7, 1.7) to[out=30, in=90] (1.7, 1);

\draw[green] (0.7, 1.7) to[out=10, in=90] (1.5, 1);

\begin{scope}[shift={(0, 2)}]
    \draw (1.3, 0) -- (1.3, 0.3) to[out=-150, in=-90] (0.3, 1);
    % \filldraw[white, opacity=0.8] (0.3, 0) to[out=90, in=-90] (0.7, 1) -- (0.3, 1) -- cycle;

    \draw[green] (1.3, 0.3) to[out=-170, in=-90] (0.5, 1);
    \draw[red] (1.3, 0.3) to[out=-160, in=-90] (0.5, 0.8) -- (0.5, 1);
    \filldraw[white, opacity=0.8] (0.3, 0) to[out=90, in=-90] (0.7, 1) -- (0.3, 1) to[out=-90, in=90] (0.7, 0) -- cycle;
    \draw[green] (0.5, 0) -- (0.5, 1);
    
    \draw (1.7, 0) -- (1.7, 1);
    \draw (0.7, 0) to[out=90, in=-90] (0.3, 1);
    \draw[white, line width=4, opacity=0.9] (0.3, 0) to[out=90, in=-90] (0.7, 1);
    \draw (0.3, 0) to[out=90, in=-90] (0.7, 1);
    \draw (1.3, 1) -- (1.3, 0.7) to[out=-120, in=-90] (0.7, 1);
    
    \draw (1.3, 1) to[out=90, in=-90] (1.7, 2);
    \draw[white, line width=4, opacity=0.9] (1.7, 1) to[out=90, in=-90] (1.3, 2);
    \draw (1.7, 1) to[out=90, in=-90] (1.3, 2);
    \draw (0.3, 1) -- (0.3, 2);

    \draw[red] (0.7, 1.7) to[out=20, in=90] (1.5, 1.2) -- (1.5, 1);
    
    \filldraw[white, opacity=0.8] (0.7, 1.7) to[out=30, in=90] (1.7, 1) -- (1.3, 1) to[out=90, in=60] (0.7, 1.3) -- cycle;
    \draw (0.7, 1) -- (0.7, 1.3) to[out=60, in=90] (1.3, 1);
    \draw (0.7, 2) -- (0.7, 1.7) to[out=30, in=90] (1.7, 1);

    \draw[green] (0.7, 1.7) to[out=10, in=90] (1.5, 1);
\end{scope}

\draw (1.7, 4) to[out=90, in=180] (2.2, 4.5) to[out=0, in=90] (2.7, 4) -- (2.7, 0) to[out=-90, in=0] (2.2, -0.5) to[out=180, in=-90] (1.7, 0);
\draw (1.3, 4) to[out=90, in=180] (2.2, 4.9) to[out=0, in=90] (3.1, 4) -- (3.1, 0) to[out=-90, in=0] (2.2, -0.9) to[out=180, in=-90] (1.3, 0);
\draw (0.7, 4) to[out=90, in=180] (2.2, 5.5) to[out=0, in=90] (3.7, 4) -- (3.7, 0) to[out=-90, in=0] (2.2, -1.5) to[out=180, in=-90] (0.7, 0);
\draw (0.3, 4) to[out=90, in=180] (2.2, 5.9) to[out=0, in=90] (4.1, 4) -- (4.1, 0) to[out=-90, in=0] (2.2, -1.9) to[out=180, in=-90] (0.3, 0);

\draw[ultra thick, darkgreen] (0.3, 1) -- (0.7, 1);
\draw[ultra thick, darkgreen] (0.3, 3) -- (0.7, 3);
\draw[very thick] (1.3, 1) -- (1.7, 1);
\draw[ultra thick, darkgreen] (1.3, 2) -- (1.7, 2);
\draw[very thick] (1.3, 3) -- (1.7, 3);
\draw[ultra thick, darkgreen] (1.3, 4) -- (1.7, 4);

\draw[->] (0.5, 1.1) -- (0.5, 1.3);
\draw[->] (0.5, 3.1) -- (0.5, 3.3);
\draw[->] (1.5, 1.9) -- (1.5, 1.7);
\draw[->] (1.5, 3.9) -- (1.5, 3.7);

\draw[white, line width=4] (3.4, 2) [partial ellipse = 40 : -220 : 1.0 and 0.3];
\draw[red, mid] (3.4, 2) [partial ellipse = 40 : -220 : 1.0 and 0.3];
\draw[red] (3.4, 2) [partial ellipse = 77 : 103 : 1.0 and 0.3];
\node[red, right] at (4.4, 2){$\gamma_e$};
\node[green, right] at (4.1, 4){$\gamma_h$};

\node at (0.5, 0){$a$};
\node at (0.5, 4){$a$};
\node at (1.5, 0){$b$};
\node at (1.5, 4){$b$};
\node at (1.0, 0.45){$c$};
\node at (0.5, 2){$d$};
\node at (1.0, 3.55){$e$};
\end{tikzpicture}
}}
\]
\caption{State sums on two knot holders for the figure-eight knot}
\label{fig:figure-eight-state-sums}
\end{figure}

Let us consider state sums on two figure-eight knot holders, one from the ``direct model'' of \cite{BirmanWilliams} (which can be derived in the same manner as in Example \ref{eg:trefoil-direct-model}), and the other\footnote{This figure-eight knot holder happens to agree with the ``branched covering model'' from \cite{BirmanWilliams}.} from Example \ref{eg:braid-homogeneous-link-knot-holder}, as shown in Figure \ref{fig:figure-eight-state-sums}. 
From the first knot holder on the left of Figure \ref{fig:figure-eight-state-sums}, we have:
\begin{align*}
\Phi_{S^3 \setminus \mathbf{4}_1}(x,q)
&= (1-x)\sum_{a,b,c,d \geq 0}
q^{c^2}
\qbin{a+b+c}{a}_{q^{-1}}
\qbin{b+c}{b}_{q}
\qbin{c+d}{c}_{q}
x^{(a+b+c)+(c+d)} \\ 
&= 1 + 2x + (q^{-1}+3+q)x^2 + (2q^{-2}+2q^{-1}+5+2q+2q^2)x^3 +O(x^4),
\end{align*}
where the overall factor of $(1-x)$ comes from the elliptic flow loop along the meridian. 
No framing factors appear, since $\gamma_e$ and $\gamma_h$ do not link the knot holder at all, and the sliding factors from $\ell_{\mathrm{pair}}$ cancel: $q^{\frac{c}{2}}\cdot q^{-\frac{c}{2}} = 1$. 

From the second knot holder on the right of Figure \ref{fig:figure-eight-state-sums}, we have:
\begin{align*}
\Phi_{S^3 \setminus \mathbf{4}_1}(x,q)
&= \sum_{\substack{a,b,c,d,e \geq 0 \\ 0\leq \epsilon \leq 1}} (-1)^{a+d+b+(b-a+d)} q^{\frac{a^2 + d^2 - b^2 - (b-a+d)^2}{2} + 2\epsilon a - 2\epsilon b} \qbin{a+c}{a}_q \qbin{a+e}{d}_q\\
&\quad \times  \qbin{b+e}{b}_{q^{-1}} \qbin{b+c}{a+c-d}_{q^{-1}} 
x^{(a+c)+(a+e)+b+(b-a+d)} 
(-x^3)^{\epsilon}
\cdot q^{\frac{a+d-b-(b-a+d)}{2} - (a-b)} \\
&= 1 + 2x + (q^{-1}+3+q)x^2 + (2q^{-2}+2q^{-1}+5+2q+2q^2)x^3 +O(x^4),
\end{align*}
where $\epsilon$ counts the elliptic flow loop $\gamma_e$, and the factor 
$q^{\frac{a+d-b-(b-a+d)}{2}}$ (resp., $q^{- (a-b)}$) comes from the linking with the framing line $\gamma_h$ (resp., $\gamma_e$). 

Observe that they give the same count, which follows from Theorem \ref{mainthm:flow-loop-count}. 
Moreover, they agree with the BPS $q$-series for the figure-eight knot complement \cite{GukovManolescu, Park_inverted}
\[
x^{-\frac12} \widehat{Z}_{S^3 \setminus \mathbf{4}_1}(x,q) = 1 + 2x + (q^{-1}+3+q)x^2 + (2q^{-2}+2q^{-1}+5+2q+2q^2)x^3 +O(x^4),
\]
which is a special case of Theorem \ref{mainthm:flow-equals-quantum}. 
\end{eg}

\begin{rmk}
The knot holders encode all the information to determine the full HOMFLYPT-skein-valued flow loop count (i.e., not just its specialization to $\mathrm{GL}_1$-skeins), where the multiple covers of a primitive flow loop $\gamma$ are counted by the skein-valued basic annulus partition function \cite[Sec. 2.2]{ChauhanEkholmLonghi}
\[
w_\gamma := 
\begin{cases}
\sum_{\lambda} (-1)^{|\lambda|} W_{\lambda,\emptyset} \otimes W_{\emptyset, \lambda^t} &\text{if } \gamma \text{ is elliptic}, \\
\sum_{\lambda} W_{\lambda, \emptyset} \otimes W_{\emptyset, \lambda} &\text{if } \gamma \text{ is positive hyperbolic}, \\
\sum_{\lambda} (-1)^{|\lambda|} W_{\lambda, \emptyset} \otimes W_{\emptyset, \lambda} &\text{if } \gamma \text{ is negative hyperbolic}.
\end{cases}
\]
However, computing such a partition function in practice is unwieldy, as it often requires prior knowledge of all colored HOMFLYPT polynomials of all link types. 
For instance, even for the figure-eight knot, it is known that all link types appear as periodic orbits \cite{Ghrist}. 
Hence, the reduction to $\mathrm{GL}_1$-skeins was crucial for our purposes to have a computable result as in Example \ref{eg:figure-eight-state-sum}. 
\end{rmk}

\subsection{Skein-categorical description}\label{subsec:skein-category}
Here we sketch how to extend down (in the sense of TQFT) to surfaces, by rephrasing the flow loop count in terms of skein categories. 
It will not be used in the rest of this paper, so uninterested readers may safely skip this subsection. 
\begin{defn}
Given a surface $F$ equipped with a line field $\xi$ with transverse zeros, its $\mathrm{GL}_1$-\emph{skein category} $\SkCat^{\mathrm{GL}_1}(F, \xi)$ is the $\mathbb{Z}[q^{\pm \frac12}]$-linear category where 
\begin{itemize}
\item Objects are finite set of points on $F$, away from the zeros of $\xi$, and 
\item For each pair of objects $X = \{x_1, \cdots, x_k\}$ and $Y = \{y_1, \cdots, y_l\}$, the $\mathbb{Z}[q^{\pm \frac12}]$-module of morphisms $\mathrm{Hom}(X,Y)$ is given by the $\mathrm{GL}_1$-skein module of $F \times [0, 1]$ with framing lines along the zeros of $\xi$, with boundary conditions given by $X \times \{0\}$ and $Y \times \{1\}$. 
That is, it is generated by framed\footnote{Allowing half-twists as before.}, oriented tangles in $F\times [0, 1]$ with incoming (resp., outgoing) boundary points at $X \times \{0\}$ (resp., $Y \times \{1\}$) where the framing agrees with $\xi$, modulo the usual $\mathrm{GL}_1$-skein relations and the framing line relation. 
\end{itemize}
\end{defn}
In the following, we implicitly pass to an appropriate grading completion of the additive closure of the skein category to allow (infinite) direct sums of objects and matrices of morphisms. 

Let $F$ be a surface equipped with a marking, as well as a perturbation vector field $\xi = \xi_{\mathrm{pert}}$ adapted to the marking (as in Section \ref{subsec:adapted-framing}). 
By a \emph{state} $\sigma$ on the marked surface $F$, we mean an assignment of $\sigma(p_i) \in \{0, 1\}$ to each base point $p_i$ and an assignment of $\sigma(\alpha_j) \in \mathbb{Z}_{\geq 0}$ to each arc $\alpha_j$. 
Let $\mathcal{S}(F)$ denote the set of all such states. 
For each state $\sigma \in \mathcal{S}(F)$ on the marked surface $F$, let $X_\sigma$ be the object of the skein category $\SkCat^{\mathrm{GL}_1}(F, \xi)$, consisting of $\sigma(\alpha_j)$ distinct interior points of $\alpha_j$, for each $j$, together with all the base points $p_i$ for which $\sigma(p_i) = 1$.\footnote{The base points $p_i$ are zeros of $\xi$, so technically, we need to slightly shift these points off the zeros of $\xi$, but it doesn't matter how we are shifting them off, as the resulting objects are all naturally isomorphic, with the isomorphism given by the tangle consisting of (almost) vertical strands, framed by $\xi$. For the same reason, it doesn't matter how we choose the $\sigma(\alpha_j)$ interior points of $\alpha_j$ relative to the zero of $\xi$ on $\alpha_j$.}
By taking the formal direct sum over all states, we obtain an object
\[
Z(F) := \bigoplus_{\sigma \in \mathcal{S}(F)} X_{\sigma} \in \SkCat^{\mathrm{GL}_1}(F, \xi). 
\]

For any self-homeomorphism $\varphi : F\rightarrow F$ preserving the set of base points, as described in Section \ref{subsec:knot-holder-construction}, there is a knot holder $S$ in $F \times [0, 1]$ given by the locus of the arcs in the mapping cylinder of $\varphi$, composed with the movie of splitting and joining of the arcs that turns the arcs of $\varphi(F)$ to those of $F$. 
Moreover, as described in Section \ref{subsec:adapted-framing}, there is a perturbation vector field $\xi = \xi_{\mathrm{pert}}$ on $F\times [0, 1]$ adapted to $S$. 
Hence, we can perform a state sum on $S$ to obtain a morphism
\[
Z(\varphi) : Z(F) \rightarrow Z(F). 
\]
More precisely, $Z(\varphi)$ is given by an $\mathcal{S}(F) \times \mathcal{S}(F)$ matrix, whose $(\sigma', \sigma)$-entry is the sum of all possible ways to connect the points $\varphi(X_\sigma)$ to the points $X_{\sigma'}$---where the base points should stay constant, and the points on arcs can follow along the knot holder $S$---times a sign given by $(-1)^{\mathrm{nh} + \mathrm{e}}$, where $\mathrm{nh}$ counts the number of negative hyperbolic flow lines\footnote{Here we are choosing an orientation of the markings, in order to tell which strips of $S$ are half-twisted. The trace of $Z(\varphi)$, however, would be independent of such an orientation choice.}, and $\mathrm{e}$ counts the inversion number of the partial permutation of the base points used in the states $\sigma$ and $\sigma'$.\footnote{Here we are using an ordering of the base points. The trace of $Z(\varphi)$, however, would be independent of such an ordering choice.}

By construction, the assignment $\varphi \mapsto Z(\varphi)$ is functorial, i.e., $Z(\varphi' \circ \varphi) = Z(\varphi') \circ Z(\varphi)$, and we obtain a representation of the mapping class group of $F$ that fixes the set of base points:
\[
\mathrm{MCG}(F) \rightarrow \mathrm{Aut}(Z(F)).
\]
Moreover, the supertrace\footnote{The $\mathbb{Z}/2$-grading is given by the parity of the number of base points used in each state. This is to get the appropriate sign for the elliptic flow loops in the mapping torus.} of $Z(\varphi)$, 
\[
\Tr Z(\varphi) \in \Sk_{q}^{\mathrm{GL}_1} (M_\varphi, \ell_{\mathrm{fr}}),
\]
which is an element of (the action completion of) the $\mathrm{GL}_1$-skein module of the mapping torus $M_\varphi$ of $\varphi$ twisted by framing line $\ell_{\mathrm{fr}}$ along the zeros of $\xi$, recovers the state sum model on the knot holder discussed in Section \ref{subsec:knot-holder-state-sum}. 
That is, this is nothing but the $\mathrm{GL}_1$-skein-valued flow loop count $\Phi_{M_\varphi}$ in $M_\varphi$.

\section{Braid group representation on Verma modules from flow loops}\label{sec:braid-group-repn}
It turns out that, in the special case when $F = D^2_n$ is a disk with $n$ punctures, we can recover the Lawrence representation \cite{Lawrence} $L_{n,m}$ of the braid group $B_n$ by counting flow lines on the knot holders associated to braids $\beta \in B_n$.\footnote{This is a special case of the mapping class group representations constructed in Section \ref{subsec:skein-category}.} 
Since the Lawrence representation is known to be equivalent to the braid group representation on an appropriate subspace of the tensor products of $n$ copies of Verma modules of the quantum group $U_q(\mathrm{sl}_2)$, this provides a key connection between the flow loop counts and quantum invariants.

\subsection{Braid group representation from knot holders}\label{subsec:braid-knot-holder}
Let $D^2_n$ be a disk with $n$ punctures $p_1, \cdots, p_n$ which are ordered. 
Treating the boundary of the disk as a source and the punctures as sinks, we draw a marking on $D^2_n$ as in Figure \ref{fig:punctured-disk-marking}, by drawing an arc $\alpha_i$ between each pair of neighboring punctures $p_i$ and $p_{i+1}$. 
\begin{figure}
\centering
\[
\vcenter{\hbox{
\begin{tikzpicture}
\draw (0, 0) [partial ellipse = 0 : 360 : 2.0 and 1.0];
\draw (-1.5, 0) circle (0.1);
\draw (-0.75, 0) circle (0.1);
\draw (0, 0) circle (0.1);
\node at (0.75, 0){$\cdots$};
\draw (1.5, 0) circle (0.1);
\draw[darkgreen, thick] (-1.4, 0) -- (-0.85, 0);
\draw[darkgreen, thick] (-0.65, 0) -- (-0.1, 0);
\draw[darkgreen, thick] (0.1, 0) -- (0.4, 0);
\draw[darkgreen, thick] (1.1, 0) -- (1.4, 0);
\node[below] at (-1.5, -0.1){$p_1$};
\node[below] at (-0.75, -0.1){$p_2$};
\node[below] at (0, -0.1){$p_3$};
\node[below] at (1.5, -0.1){$p_n$};
\node[above, darkgreen] at (-1.125, 0){$\alpha_1$};
\node[above, darkgreen] at (-0.375, 0){$\alpha_2$};
\node[above, darkgreen] at (1.125, 0){$\alpha_{n-1}$};
\end{tikzpicture}
}}
\]
\caption{A disk with $n$ punctures and markings}
\label{fig:punctured-disk-marking}
\end{figure}
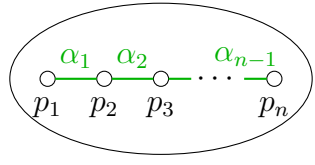
The braid group $B_n$ acts on $D^2_n$, and the movie of splitting and joining of the arcs, together with the framing lines, is shown in Figure \ref{fig:braiding-movie}. 
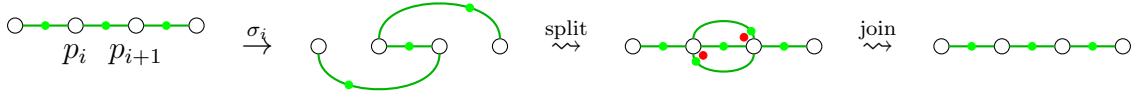
\begin{figure}
\centering
\[
\vcenter{\hbox{
\begin{tikzpicture}
\draw (-1.2, 0) circle (0.1);
\draw (-0.4, 0) circle (0.1);
\draw (0.4, 0) circle (0.1);
\draw (1.2, 0) circle (0.1);
\draw[darkgreen, thick] (-0.3, 0) -- (0.3, 0);
\draw[darkgreen, thick] (-1.1, 0) -- (-0.5, 0);
\draw[darkgreen, thick] (0.5, 0) -- (1.1, 0);

\filldraw[green] (-0.8, 0) circle (0.05);
\filldraw[green] (0.0, 0) circle (0.05);
\filldraw[green] (0.8, 0) circle (0.05);

\node[below] at (-0.4, -0.1){$p_i$};
\node[below] at (0.4, -0.1){$p_{i+1}$};
\end{tikzpicture}
}}
\;\;
\overset{\sigma_i}{\rightarrow}
\;\;
\vcenter{\hbox{
\begin{tikzpicture}
\draw (-1.2, 0) circle (0.1);
\draw (-0.4, 0) circle (0.1);
\draw (0.4, 0) circle (0.1);
\draw (1.2, 0) circle (0.1);
\draw[darkgreen, thick] (-0.3, 0) -- (0.3, 0);
\draw[darkgreen, thick] (-1.2, -0.1) to[out=-90, in=-90] (0.4, -0.1);
\draw[darkgreen, thick] (-0.4, 0.1) to[out=90, in=90] (1.2, 0.1);

\filldraw[green] (0.0, 0) circle (0.05);
\filldraw[green] (-0.8, -0.51) circle (0.05);
\filldraw[green] (0.8, 0.51) circle (0.05);
\end{tikzpicture}
}}
\;\;
\overset{\text{split}}{\rightsquigarrow}
\;\;
\vcenter{\hbox{
\begin{tikzpicture}
\draw (-1.2, 0) circle (0.1);
\draw (-0.4, 0) circle (0.1);
\draw (0.4, 0) circle (0.1);
\draw (1.2, 0) circle (0.1);
\draw[darkgreen, thick] (-0.3, 0) -- (0.3, 0);
\draw[darkgreen, thick] (-1.1, 0) -- (-0.5, 0);
\draw[darkgreen, thick] (0.5, 0) -- (1.1, 0);
\draw[darkgreen, thick] (-0.4, 0.1) to[out=90, in=90] (0.4, 0.1);
\draw[darkgreen, thick] (-0.4, -0.1) to[out=-90, in=-90] (0.4, -0.1);

\filldraw[green] (-0.8, 0) circle (0.05);
\filldraw[green] (0.0, 0) circle (0.05);
\filldraw[green] (0.8, 0) circle (0.05);

\filldraw[green] (-0.365, -0.2) circle (0.05);
\filldraw[red] (-0.27, -0.12) circle (0.05);
\filldraw[green] (0.365, 0.2) circle (0.05);
\filldraw[red] (0.27, 0.12) circle (0.05);
\end{tikzpicture}
}}
\;\;
\overset{\text{join}}{\rightsquigarrow}
\;\;
\vcenter{\hbox{
\begin{tikzpicture}
\draw (-1.2, 0) circle (0.1);
\draw (-0.4, 0) circle (0.1);
\draw (0.4, 0) circle (0.1);
\draw (1.2, 0) circle (0.1);
\draw[darkgreen, thick] (-0.3, 0) -- (0.3, 0);
\draw[darkgreen, thick] (-1.1, 0) -- (-0.5, 0);
\draw[darkgreen, thick] (0.5, 0) -- (1.1, 0);

\filldraw[green] (-0.8, 0) circle (0.05);
\filldraw[green] (0.0, 0) circle (0.05);
\filldraw[green] (0.8, 0) circle (0.05);
\end{tikzpicture}
}}
\]
\caption{Action of $\sigma_{i} \in B_n$ on the markings, together with the framing lines}
\label{fig:braiding-movie}
\end{figure}
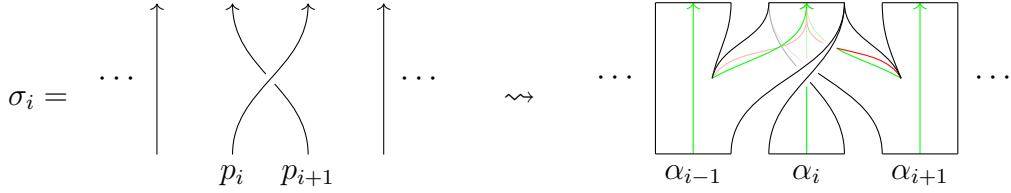
\begin{figure}
\centering
\[
\sigma_i
=
\vcenter{\hbox{
\begin{tikzpicture}
\draw[->] (-1, 0) -- (-1, 2);
\draw[->] (1, 0) to[out=90, in=-90] (0, 2);
\draw[white, line width=5] (0, 0) to[out=90, in=-90] (1, 2);
\draw[->] (0, 0) to[out=90, in=-90] (1, 2);
\draw[->] (2, 0) -- (2, 2);

\node at (-1.5, 1){$\cdots$};
\node at (2.5, 1){$\cdots$};

\node[below] at (0, 0){$p_i$};
\node[below] at (1, 0){$p_{i+1}$};
\end{tikzpicture}
}}
\quad
\rightsquigarrow
\quad
\vcenter{\hbox{
\begin{tikzpicture}

\draw (1.75, 1) to[out=112.5, in=-90] (1.5, 2);
\draw (-0.75, 1) to[out=77.5, in=-90] (-0.5, 2);
\draw (1.5, 0) to[out=90, in=-90] (0, 2);
\filldraw[white, opacity=0.7] (1, 0) to[out=90, in=-90] (0, 2) -- (1, 2) to[out=-90, in=90] (0, 0) -- cycle;
\draw (1, 0) to[out=90, in=-90] (0, 2);

\draw[green, ->] (-1, 0) -- (-1, 2);
\draw[green, ->] (0.5, 0) -- (0.5, 2);
\draw[green, ->] (2, 0) -- (2, 2);

\draw[green, ->] (1.75, 1) to[out=150, in=-90] (0.5, 2);
\draw[red, ->] (1.75, 1) to[out=120, in=-90] (0.5, 1.8) -- (0.5, 2);

\draw[red, ->] (-0.75, 1) to[out=60, in=-90] (0.5, 1.8) -- (0.5, 2);

\draw[white, line width=4] (0, 0) to[out=90, in=-90] (1, 2);
\draw[white, line width=4] (-0.5, 0) to[out=90, in=-90] (1, 2);
\filldraw[white, opacity=0.7] (-0.5, 0) to[out=90, in=-90] (1, 2) -- (0, 2) to[out=-90, in=77.5] (-0.75, 1) -- cycle;

\draw[green, ->] (-0.75, 1) to[out=30, in=-90] (0.5, 2);

\draw (-1.5, 0) -- (-0.5, 0);
\draw (-1.5, 2) -- (-0.5, 2);
\draw (-1.5, 0) -- (-1.5, 2);
\draw (0, 0) -- (1, 0);
\draw (0, 2) -- (1, 2);
\draw (1.5, 0) -- (2.5, 0);
\draw (1.5, 2) -- (2.5, 2);
\draw (2.5, 0) -- (2.5, 2);
\draw (0, 0) to[out=90, in=-90] (1, 2);
\draw (1.75, 1) to[out=112.5, in=-90] (1, 2);
\draw (-0.75, 1) to[out=77.5, in=-90] (0, 2);
\draw (-0.5, 0) to[out=90, in=-90] (1, 2);

\node at (-2.0, 1){$\cdots$};
\node at (3.0, 1){$\cdots$};

\node[below] at (-1, 0){$\alpha_{i-1}$};
\node[below] at (0.5, 0){$\alpha_i$};
\node[below] at (2, 0){$\alpha_{i+1}$};
\end{tikzpicture}
}}
\]
\caption{Knot holder for the Artin generators $\sigma_i \in B_n$. 
When $i=1$ (resp., $i=n-1$), there are no strands to the left (resp., right) of the crossing, so the strips starting from the arc $\alpha_{i-1}$ (resp., $\alpha_{i+1}$) are not present.}
\label{fig:braiding-knot-holder}
\end{figure}
The resulting knot holder is shown in Figure \ref{fig:braiding-knot-holder}. 

For any nonnegative integer $m$, let $\mathcal{S}_{n,m}$ denote the set of all ordered $(n-1)$-tuples of nonnegative integers, $(a_1, \cdots, a_{n-1})$, such that $\sum_{1\leq i\leq n-1} a_i = m$. 
Thinking of $a_i$ as the number of points on the arc $\alpha_i$, this is just the set of states on $D^2_n$ consisting of a total of $m$ points. 
Counting the flow lines on the knot holders above, we obtain a representation of $B_n$ on a vector space with basis labeled by $\mathcal{S}_{n,m}$, which we summarize as a proposition: 
\begin{prop}\label{prop:braid-group-rep}
Let $V_{n,m}$ be the $\binom{m+n-2}{m}$-dimensional vector space with basis $\{v_{\vec{a}}\}_{\vec{a} \in \mathcal{S}_{n,m}}$ labeled by $\mathcal{S}_{n,m}$. 
Then, the following assignment defines a representation of the braid group $B_n$ on $V_{n,m}$:
\[
\sigma_i \cdot v_{(a_1, \cdots, a_{n-1})} = 
\begin{cases}
\sum_{\substack{0\leq b\leq a_{i-1} \\ 0\leq c\leq a_{i+1}}}
(-1)^{a_i}
q^{\frac{{a_i}^2}{2} + \frac{a_i + b + c}{2}}
\qbin{a_i + b+c}{a_i, b, c}_q
x^{\frac{(a_i + b) + (a_i +c)}{2}} & \\
\qquad\qquad \times v_{(a_1, \cdots, a_{i-1} - b, a_{i} + b + c, a_{i+1} - c, \cdots, a_{n-1})} &\text{if }2\leq i\leq n-2, \\
\sum_{\substack{0\leq c\leq a_{2}}} (-1)^{a_i}
q^{\frac{{a_i}^2}{2} + \frac{a_i + c}{2}}
\qbin{a_1 + c}{c}_q
x^{\frac{a_i + (a_i +c)}{2}} & \\
\qquad\qquad \times
v_{(a_{1} + c, a_{2} - c, \cdots, a_{n-1})} &\text{if }i = 1, \\
\sum_{\substack{0\leq b\leq a_{n-2}}} (-1)^{a_i}
q^{\frac{{a_i}^2}{2} + \frac{a_i + b}{2}}
\qbin{a_{n-1} + b}{b}_q
x^{\frac{(a_i + b) + a_i}{2}} & \\
\qquad\qquad \times v_{(a_1, \cdots, a_{n-2} - b, a_{n-1} + b)} &\text{if }i = n-1,
\end{cases}
\]
where the $q$-trinomial coefficients are defined by
\[
\qbin{a+b+c}{a,b,c}_q := \frac{[a+b+c]_q!}{[a]_q! [b]_q! [c]_q!} = \qbin{a+b}{b}_q \qbin{a+b+c}{c}_q.
\]
\end{prop}
In the proposition above, the factor $q^{\frac{a_i + b + c}{2}}$ comes from the linking with the framing lines ($q^{\frac{a_i}{2}}$ from linking with $\gamma_h$ along $\alpha_i$, and $q^{\frac{b+c}{2}}$ from linking with $\ell_{\mathrm{pair}}$), and the degree of $x$ counts the linking number with the braid $\beta$ itself. 
While we have presented the linking number $\mathrm{lk}(\gamma, \ell_{\mathrm{fr}})$ (resp., $\mathrm{lk}(\gamma, \beta)$) as the average of the signed number of crossings where the flow line $\gamma$ goes under and over $\ell_{\mathrm{fr}}$ (resp., $\beta$), we could have instead only counted the signed number of crossings where $\gamma$ goes under $\ell_{\mathrm{fr}}$ (resp., $\beta$). 
This will yield an equivalent representation of the braid group, as the linking number remains the same once the braid is closed up. 
In this equivalent representation, the weight now takes a slightly different form: 
\[
(-1)^{a_i}
q^{\frac{{a_i}^2}{2} + \frac{a_i}{2} + c}
\qbin{a_i + b+c}{a_i, b, c}_q
x^{a_i +c}
=
(-1)^{a_i}
q^{\frac{{a_i(a_i - 1)}}{2}}
\qbin{a_i + b+c}{a_i, b, c}_q
(qx)^{a_i +c}. 
\]
Comparing this with the Lawrence representation $L_{n,m}$ written in the basis of ``standard code sequences'' as in \cite[Cor. 4.7]{Martel}\footnote{Martel's $t$ (resp., $s$) is our $q^{-1}$ (resp., $qx$).}, we immediately obtain: 
\begin{prop}\label{prop:equivalence-to-Lawrence}
The braid group representation $V_{n,m}$ in Proposition \ref{prop:braid-group-rep} is equivalent to the Lawrence representation $L_{n,m}$. 
\end{prop}

Note that, if we fill in the punctures of the disk with elliptic flow lines (and elliptic framing lines along them) and close up, we would be counting the flow loops in the suspension flow of a self-homeomorphism $\beta : D^2 \rightarrow D^2$ specified by the braid $\beta$. 
In other words, it is the flow loop count in the unknot complement $S^3 \setminus O \cong D^2 \times S^1$. 
But in the unknot complement, there is a much simpler flow with a single elliptic flow loop, so the flow loop count must be $\Phi_{S^3 \setminus O}(z,q) = 1-z$. 
Therefore, as a corollary of Theorem \ref{mainthm:flow-loop-count}, we have:
\begin{cor}
Suppose, for simplicity, that $\beta \in B_n$ is a braid that closes up to a knot (i.e., with a single component), and let $w$ denote the writhe of $\beta$. 
Then, 
\[
\sum_{\substack{m\geq 0 \\ 0\leq \epsilon \leq 1}} 
\Tr_{V_{n,m}}(\beta)\bigg\vert_{x = q^{-1+2\epsilon}} z^{m} \cdot (q^{w})^{\epsilon} (-z^n)^{\epsilon} = \Phi_{S^3 \setminus O}(z,q) = 1-z.
\]
\end{cor}
In this formula, $\epsilon$ counts the elliptic flow loop along the closure of $\beta$, the power of $z$ keeps track of the homology class in $D^2 \times S^1 \cong S^3 \setminus O$. 
In the specialization $x=q^{-1+2\epsilon}$, the factor $q^{-1}$ is from the linking number with the elliptic framing line along the closure of $\beta$, and the factor $q^{2\epsilon}$ is from the linking with the elliptic flow loop. 
Finally, the factor $(q^w)^{\epsilon}$ is from the linking of the elliptic flow loop with the hyperbolic framing lines.

\subsection{Connection to Verma modules}
Let $V_\infty(x)$ denote the Verma module of the quantum group $U_q(\mathfrak{sl}_2)$ with generic highest weight $\lambda_x := \log_q x - 1$. 
Then, the universal $R$-matrix of $U_q(\mathfrak{sl}_2)$ induces a representation
\[
B_n \rightarrow \mathrm{Aut}(V_\infty(x)^{\otimes n}). 
\]
Explicitly, in some basis $v_0, v_1, v_2, \cdots$ of $V_\infty(x)$, 
\[
\sigma_i \cdot (v_{a_1} \otimes \cdots \otimes v_{a_n}) =  \sum_{a_i', a_{i+1}' \geq 0} \check{R}(x)_{a_i, a_{i+1}}^{a_i', a_{i+1}'} \cdot v_{a_1} \otimes \cdots \otimes v_{a_i'} \otimes v_{a_{i+1}'} \otimes \cdots \otimes v_{a_n},
\]
where
\[
\check{R}(x)_{i,j}^{i',j'} := 
\delta_{i+j, i'+j'} \;q^{jj'} (q x^{-1})^{\frac{j+j'+1}{2}} \qbin{i}{j'}_{q} (q^{j+1}x^{-1};q)_{i-j'}.
\]
Let $(V_\infty(x)^{\otimes n})_m$ denote the $q^{\frac{n\lambda_x}{2} - m}$-weight subspace of $V_\infty(x)^{\otimes n}$, i.e., the span of basis vectors $v_{a_1} \otimes \cdots \otimes v_{a_n}$ for which $a_1 + \cdots + a_n = m$. 
By restricting the braid group representation to this weight subspace and further restricting to the space of null vectors
\[
N_{n,m} := \ker E \cap (V_\infty(x)^{\otimes n})_m
\]
in that weight subspace, we obtain a subrepresentation of dimension $\binom{m + n-2}{m}$. 

Kohno's theorem \cite{Kohno}, as reviewed in \cite[Sec. 4.2]{Ito_Garside}, proves that this subrepresentation $N_{n,m}$ is equivalent to the Lawrence representation $L_{n,m}$. 
Since
\[
(V_\infty(x)^{\otimes n})_m = 
\bigoplus_{0\leq k\leq m} F^{m-k} N_{n,k} \cong 
\bigoplus_{0\leq k\leq m} N_{n,k},
\]
we can express the graded trace of the braid group representation on $V_\infty(x)^{\otimes n}$ in terms of the graded trace of Lawrence representations, or of $V_{n,m}$ using Proposition \ref{prop:equivalence-to-Lawrence};
this proves Theorem \ref{mainthm:Lawrence}. 
In the convention we have been using, we can summarize this as: 
\begin{thm}\label{thm:Kohno}
The graded trace of the braid group representation on $V_\infty(x^{-1})^{\otimes n}$ can be expressed in terms of flow loop counts: 
\[
\sum_{m\geq 0} \Tr_{(V_\infty(x^{-1})^{\otimes n})_m} (\beta)\; z^m= (qx)^{\frac{w}{2}} \frac{1}{1-z} \sum_{m\geq 0} \Tr_{V_{n,m}}(\beta) \;z^m,
\]
where both sides are elements of $\mathbb{Z}[q^{\pm \frac12}, x^{\pm \frac12}][[z]]$. 
\end{thm}

As an application of Kohno's theorem, Ito \cite[Thm. 3.1]{Ito} showed that the MMR expansion of a knot $K$, presented as the closure of a braid $\beta$, can be expressed in terms of the traces of the Lawrence representation applied to $\beta$. 
Using Proposition \ref{prop:equivalence-to-Lawrence}, we rewrite Ito's theorem in terms of the representations $V_{n,m}$:\footnote{Ito's $z$ is our $x^{-1}$.}
\begin{thm}\label{thm:Ito}
Given a knot $K$ presented as the closure of a braid $\beta$ with $n$ strands and writhe $w$,  
\[
(x^{\frac12} - x^{-\frac12})\mathrm{MMR}_K(\hbar, x) = 
- q^{\frac{w -(n-1)}{2}} x^{\frac{w-n}{2}} \sum_{m\geq 0} (1 -q^{2m + (n-1)} x^n) \cdot q^{-m}\Tr_{V_{n,m}}(\beta),
\]
where the identity is understood in $\mathbb{Q}(x^{\frac12})[[\hbar]]$, after setting $q=e^\hbar$.\footnote{Ito's statement \cite[Thm. 3.1]{Ito} does not explicitly specify the ambient ring for the right-hand side. 
That the right-hand side lives in $\mathbb{Q}(x^{\frac12})[[\hbar]]$ is explained in the proof of Lemma~\ref{lem:inversion}.} 
\end{thm}
In fact, this follows formally from Theorem \ref{thm:Kohno}:
\begin{align*}
-\frac{x^{\frac12}-x^{-\frac12}}{q^{\frac12}-q^{-\frac12}} \mathrm{MMR}_K(\hbar, x)
&=
\sum_{m\geq 0} \Tr_{(V_\infty(x^{-1})^{\otimes n})_m}(\beta) \; \qty( (x^{-\frac12} q^{-\frac12})^n q^{-m} +  (x^{\frac12} q^{\frac12})^n q^{m}) \\ 
&\overset{\text{Thm.} \ref{thm:Kohno}}{=} (qx)^{\frac{w}{2}} \qty( 
\frac{(qx)^{-\frac{n}{2}}}{1-q^{-1}} \sum_{m \geq 0} \Tr_{V_{n,m}}(\beta) \; q^{-m}
+ \frac{(qx)^{\frac{n}{2}}}{1-q} \sum_{m \geq 0} \Tr_{V_{n,m}}(\beta) \; q^{m}
) \\
&= 
\frac{q^{\frac12}(qx)^{\frac{w-n}{2}}}{q^{\frac12}-q^{-\frac12}} \sum_{m\geq 0} (1 -q^{2m + (n-1)} x^n) \cdot q^{-m}\Tr_{V_{n,m}}(\beta).
\end{align*}

\section{Quantum group invariants from flow loops}\label{sec:BPS-q-series-from-flow-loops}
In this section, we prove Theorem \ref{mainthm:flow-equals-quantum}. 
\begin{lem}
Theorem \ref{mainthm:flow-equals-quantum} holds for all braid-positive knots. 
\end{lem}
\begin{proof}
If a knot $K$ is the closure of a positive braid $\beta$ with $n$ strands and $w$ crossings, then the corresponding knot holder $\KH_K$ constructed in Example \ref{eg:braid-homogeneous-link-knot-holder} is the same as the knot holder $\KH_\beta$ for the braid $\beta$ constructed in Section \ref{subsec:braid-knot-holder}, but now with an elliptic flow loop and an elliptic framing line along the braid axis. 
It follows that the flow loop count is
\begin{align*}
\Phi_{S^3 \setminus K} (x,q) &= 
\sum_{\substack{m\geq 0 \\ 0\leq \epsilon \leq 1}} q^{2\epsilon m} \Tr_{V_{n,m}}(\beta)\;(-x^n)^{\epsilon} \cdot q^{\epsilon(n-1) -m} \\
&= \sum_{m\geq 0} (1 -q^{2m + (n-1)} x^n) \cdot q^{-m}\Tr_{V_{n,m}}(\beta),
\end{align*}
where $\epsilon$ counts the elliptic flow loop along the braid axis, the factor $q^{2\epsilon m}$ is from the linking between elliptic and hyperbolic flow loops, $q^{\epsilon(n-1)}$ is from the linking of the elliptic flow loop with the hyperbolic framing lines, and $q^{-m}$ is from the linking of the hyperbolic flow loops with the elliptic framing line along the braid axis. 

Comparing with Theorem \ref{thm:Ito}, we see that this agrees (up to a monomial factor) with the MMR expansion. 
As shown in \cite{Park_large-color}, this MMR expansion can be resummed in a unique way into a power series in $q$ and $x$, which gives the BPS $q$-series $\widehat{Z}_{S^3 \setminus K}(x,q)$. 
Therefore, 
\[
\Phi_{S^3 \setminus K} (x,q) = - q^{\frac{(n-1)-w}{2}} x^{\frac{n-w}{2}} \widehat{Z}_{S^3 \setminus K}(x,q)
\]
as desired. 
\end{proof}

In the rest of this section, we will extend this result to all braid-homogeneous knots. 
The idea of proof is very similar to that of the \emph{inverted state sum} introduced in \cite{Park_inverted}. 

\begin{proof}[Proof of Theorem \ref{mainthm:flow-equals-quantum}]
Let $K$ be a knot which is the closure of a homogeneous braid $\beta$. 
Let $\KH_\beta$ denote the knot holder for braid $\beta$ constructed in Section \ref{subsec:braid-knot-holder} and $\KH_K$ denote the knot holder for $K$ constructed in Example \ref{eg:braid-homogeneous-link-knot-holder}. 
Note that $\KH_\beta$ is in general \emph{not} a knot holder for the flow on the fibered knot complement $S^3 \setminus K$. 

Starting from the right-hand side of Theorem \ref{thm:Ito}, which is a state sum on $\KH_\beta$ (together with an extra elliptic flow loop and an elliptic framing line along the braid axis), we will ``invert'' all the states on the columns of strips corresponding to negative crossings. 
We will then observe that this inverted state sum on $\KH_\beta$ can be naturally identified (up to some monomial) with the state sum on the actual knot holder $\KH_K$ for the fibered knot complement $S^3 \setminus K$. 

\textbf{Step 1: Inverted state sum on $\KH_\beta$.} 
Let $\gamma_e$ denote the braid axis of $\beta$ where we place both the elliptic flow loop and an elliptic framing line. 
The right-hand side of Theorem \ref{thm:Ito} is the state sum
\[
Z_{\beta}(\hbar, x)
:= - q^{\frac{w -(n-1)}{2}} x^{\frac{w-n}{2}}
\sum_{\sigma \in \mathcal{S}(\mathcal{H}_\beta, \gamma_e)} w(\sigma),
\]
where, as in Section \ref{subsec:knot-holder-state-sum}, $\mathcal{S}(\mathcal{H}_\beta, \gamma_e)$ denotes the set of all states, so that $\sigma \in \mathcal{S}(\mathcal{H}_\beta, \gamma_e)$
assigns a nonnegative integer to each strip of $\KH_\beta$ and either $0$ or $1$ to $\gamma_e$. 
As stated in Theorem~\ref{thm:Ito}, this state sum should be understood in terms of its perturbative expansion, as an element of $\mathbb{Q}(x^{\frac12})[[\hbar]]$; see also the proof of Lemma~\ref{lem:inversion} below. 

Define the set of inverted states $\mathcal{S}^{\mathrm{inv}}(\mathcal{H}_\beta, \gamma_e)$ in the same way as the usual states $\mathcal{S}(\mathcal{H}_\beta, \gamma_e)$, except that for the strips on the columns of negative crossings, we assign (strictly) negative integers instead of the usual nonnegative integers. 
Define the corresponding inverted state sum to be
\[
Z_{\beta}^{\mathrm{inv}}(\hbar, x)
:= - q^{\frac{w -(n-1)}{2}} x^{\frac{w-n}{2}}
\sum_{\sigma \in \mathcal{S}^{\mathrm{inv}}(\mathcal{H}_\beta, \gamma_e)} w(\sigma),
\]
where the weights $w(\sigma)$ of the inverted states are defined in the same way as before.\footnote{The integers that a state $\sigma$ assigns appear in powers of $(-1)$, $q$, and $x$, as well as in $q$-binomials in the weight $w(\sigma)$, and they all admit natural extensions to all (not just nonnegative) integers. } 
The following lemma allows us to pass to the inversion. 
\begin{lem}\label{lem:inversion}
The inversion preserves the state sum up to a sign:
\[
Z_{\beta}(\hbar, x) = (-1)^{\mathrm{col}_-} Z_{\beta}^{\mathrm{inv}}(\hbar, x)
\]
where both sides are elements of $\mathbb{Q}(x^{\frac12})[[\hbar]]$, and $\mathrm{col}_-$ denotes the number of columns of negative crossings in the homogeneous braid $\beta$. 
\end{lem}
\begin{proof}
The proof is completely analogous to the proof of \cite[Thm.~1]{Park_inverted}, so we will be concise. 

Firstly, similarly to Rozansky's proof \cite{Rozansky} of the MMR conjecture, we can introduce a ``parametrized'' analog of the state sum that gives $Z_{\beta}(\hbar, x)$. 
That is, recall from Proposition~\ref{prop:braid-group-rep} that, in the state sum defining $Z_{\beta}(\hbar, x)$, we have the following weight 
\[
\vcenter{\hbox{
\begin{tikzpicture}
\draw (1.75, 1) to[out=112.5, in=-90] (1.5, 2);
\draw (-0.75, 1) to[out=77.5, in=-90] (-0.5, 2);
\draw (1.5, 0) to[out=90, in=-90] (0, 2);
\filldraw[white, opacity=0.7] (1, 0) to[out=90, in=-90] (0, 2) -- (1, 2) to[out=-90, in=90] (0, 0) -- cycle;
\draw (1, 0) to[out=90, in=-90] (0, 2);
% \draw[green, ->] (-1, 0) -- (-1, 2);
% \draw[green, ->] (0.5, 0) -- (0.5, 2);
% \draw[green, ->] (2, 0) -- (2, 2);
% \draw[green, ->] (1.75, 1) to[out=150, in=-90] (0.5, 2);
% \draw[red, ->] (1.75, 1) to[out=120, in=-90] (0.5, 1.8) -- (0.5, 2);
% \draw[red, ->] (-0.75, 1) to[out=60, in=-90] (0.5, 1.8) -- (0.5, 2);
\draw[white, line width=4] (0, 0) to[out=90, in=-90] (1, 2);
\draw[white, line width=4] (-0.5, 0) to[out=90, in=-90] (1, 2);
\filldraw[white, opacity=0.7] (-0.5, 0) to[out=90, in=-90] (1, 2) -- (0, 2) to[out=-90, in=77.5] (-0.75, 1) -- cycle;
% \draw[green, ->] (-0.75, 1) to[out=30, in=-90] (0.5, 2);
\draw (-1.5, 0) -- (-0.5, 0);
\draw (-1.5, 2) -- (-0.5, 2);
\draw (-1.5, 0) -- (-1.5, 2);
\draw (0, 0) -- (1, 0);
\draw (0, 2) -- (1, 2);
\draw (1.5, 0) -- (2.5, 0);
\draw (1.5, 2) -- (2.5, 2);
\draw (2.5, 0) -- (2.5, 2);
\draw (0, 0) to[out=90, in=-90] (1, 2);
\draw (1.75, 1) to[out=112.5, in=-90] (1, 2);
\draw (-0.75, 1) to[out=77.5, in=-90] (0, 2);
\draw (-0.5, 0) to[out=90, in=-90] (1, 2);
\node[below] at (-1, 0){$a$};
\node[below] at (0.5, 0){$b$};
\node[below] at (2, 0){$c$};
\node[above] at (-1, 2){$a'$};
\node[above] at (0.5, 2){$b'$};
\node[above] at (2, 2){$c'$};
\draw[thick, ->] (0.5, 0.1) -- (0.5, 0.4);
\draw[thick, ->] (0.5, 1.6) -- (0.5, 1.9);
\draw[thick, ->] (2, 0.1) -- (2, 0.4);
\draw[thick, ->] (2, 1.6) -- (2, 1.9);
\draw[thick, ->] (-1, 0.1) -- (-1, 0.4);
\draw[thick, ->] (-1, 1.6) -- (-1, 1.9);
\end{tikzpicture}
}}
\;\;=\;\;
\delta_{a+b+c, a'+b'+c'}
(-1)^{b}
q^{\frac{b^2+b'}{2}}
\qbin{b'}{a-a', b, c-c'}_{q}
x^{\frac{b+b'}{2}}
=: W_{a,b,c}^{a',b',c'}
\]
for each positive crossing, and the weight $W_{a,b,c}^{a',b',c'} \bigg\vert_{\substack{q\leftrightarrow q^{-1}, \\ x\leftrightarrow x^{-1}}}$ for each negative crossing. 
The parametrized analog of this weight $W_{a,b,c}^{a',b',c'}$ can be defined to be
\[
\mathsf{W}_{a,b,c}^{a',b',c'}(\mathsfgreek{\alpha}, \mathsfgreek{\beta}, \mathsfgreek{\gamma}) := 
\delta_{a+b+c, a'+b'+c'}
\binom{b'}{a-a', b, c-c'}
\mathsfgreek{\alpha}^{a-a'}
\mathsfgreek{\beta}^{b}
\mathsfgreek{\gamma}^{c-c'},
\]
so that
\[
\mathsf{W}_{a,b,c}^{a',b',c'}(\mathsfgreek{\alpha}, \mathsfgreek{\beta}, \mathsfgreek{\gamma}) \bigg\vert_{\substack{\mathsfgreek{\alpha} = x^{\frac12}, \\ \mathsfgreek{\beta} = -x, \\ \mathsfgreek{\gamma} = x^{\frac12}}}
=
W_{a,b,c}^{a',b',c'} \bigg\vert_{q=1}.
\]
The parametrized analog for the weight of a negative crossing can be defined analogously to be $\mathsf{W}_{a,b,c}^{a',b',c'}(\mathsfgreek{\alpha}^{-1}, \mathsfgreek{\beta}^{-1}, \mathsfgreek{\gamma}^{-1})$. 
We also introduce a parameter $\mathsfgreek{\delta}$ for the elliptic flow loop along the braid axis; 
in the parametrized state sum, it will simply contribute an overall factor of $\sum_{0\leq \epsilon \leq 1} (-1)^{\epsilon}\mathsfgreek{\delta}^{\epsilon} = 1-\mathsfgreek{\delta}$. 

Let $\mathsf{Z}_\beta(\vec{\mathsfgreek{\alpha}}, \vec{\mathsfgreek{\beta}}, \vec{\mathsfgreek{\gamma}}, \mathsfgreek{\delta})$ denote the resulting parametrized state sum, where $\vec{\mathsfgreek{\alpha}}, \vec{\mathsfgreek{\beta}}, \vec{\mathsfgreek{\gamma}}$ denote the set of parameters; since there are $3$ parameters $\alpha_c, \beta_c, \gamma_c$ for each crossing $c$, the total number of parameters is $3$ times the number of crossings, plus $1$ for $\mathsfgreek{\delta}$. 
It follows from the Foata--Zeilberger formula \cite{FoataZeilberger} that 
\begin{equation}\label{eqn:parametrized-state-sum}
\mathsf{Z}_\beta(\vec{\mathsfgreek{\alpha}}, \vec{\mathsfgreek{\beta}}, \vec{\mathsfgreek{\gamma}}, \mathsfgreek{\delta}) = 
\frac{1 - \mathsfgreek{\delta}}{\det (I - \mathcal{B})} 
\in \mathbb{Q}(\vec{\mathsfgreek{\alpha}}, \vec{\mathsfgreek{\beta}}, \vec{\mathsfgreek{\gamma}}, \mathsfgreek{\delta}),
\end{equation}
where $\mathcal{B}$ denotes the transition matrix of the associated Markov chain; the rows and columns of this matrix are labeled by the strips of our knot holder away from the splitting and joining charts, and the entries of $\mathcal{B}$ are either $0$, $1$, or one of the parameters $\mathsfgreek{\alpha}^{\pm 1}$, $\mathsfgreek{\beta}^{\pm 1}$, or $\mathsfgreek{\gamma}^{\pm 1}$. 

From the construction of the parametrized weights, we immediately have
\[
\mathsf{Z}_\beta(\vec{\mathsfgreek{\alpha}}, \vec{\mathsfgreek{\beta}}, \vec{\mathsfgreek{\gamma}}, \mathsfgreek{\delta}) \bigg\vert_{\substack{\mathsfgreek{\alpha} = x^{\frac12}, \\ \mathsfgreek{\beta} = -x, \\ \mathsfgreek{\gamma} = x^{\frac12} , \\ \mathsfgreek{\delta} = x^n}}
=
Z_{\beta}(\hbar, x) \bigg\vert_{\hbar = 0}.
\]
What is more, it follows from Rozansky's lemmas \cite{Rozansky} (see also \cite[Appendix~A]{Park_skeins} for a review) that there is a sequence of differential operators $\mathsf{D}_d$, $d\geq 0$ (with $\mathsf{D}_0 = 1$) on the parameters, such that
\[
\sum_{d\geq 0} \hbar^d\,\mathsf{D}_d\cdot \mathsf{Z}_\beta(\vec{\mathsfgreek{\alpha}}, \vec{\mathsfgreek{\beta}}, \vec{\mathsfgreek{\gamma}}, \mathsfgreek{\delta}) \bigg\vert_{\substack{\mathsfgreek{\alpha} = x^{\frac12}, \\ \mathsfgreek{\beta} = -x, \\ \mathsfgreek{\gamma} = x^{\frac12} , \\ \mathsfgreek{\delta} = x^n}}
=
Z_{\beta}(\hbar, x).
\]

We can define the parametrized version of the inverted state sum in the same way, by summing over negative states (rather than non-negative states) on the inverted strips. 
The result is
\[
\sum_{d\geq 0} \hbar^d\,\mathsf{D}_d\cdot\mathsf{Z}^{\mathrm{inv}}_\beta(\vec{\mathsfgreek{\alpha}}, \vec{\mathsfgreek{\beta}}, \vec{\mathsfgreek{\gamma}}, \mathsfgreek{\delta}) \bigg\vert_{\substack{\mathsfgreek{\alpha} = x^{\frac12}, \\ \mathsfgreek{\beta} = -x, \\ \mathsfgreek{\gamma} = x^{\frac12} , \\ \mathsfgreek{\delta} = x^n}}
=
Z^{\mathrm{inv}}_{\beta}(\hbar, x),
\]
with precisely the same differential operators $\mathsf{D}_d$ as before the inversion; 
the fact that they remain the same is a simple consequence of the basic fact that a polynomial with infinitely many zeros must vanish identically (see e.g., \cite[Lem.~A.3]{Park_skeins}). 
It follows that it is enough to prove the desired statement in the classical limit but with parameters, i.e., that
\[
\mathsf{Z}_\beta(\vec{\mathsfgreek{\alpha}}, \vec{\mathsfgreek{\beta}}, \vec{\mathsfgreek{\gamma}}, \mathsfgreek{\delta}) 
= 
(-1)^{\mathrm{col}_-}
\mathsf{Z}^{\mathrm{inv}}_\beta(\vec{\mathsfgreek{\alpha}}, \vec{\mathsfgreek{\beta}}, \vec{\mathsfgreek{\gamma}}, \mathsfgreek{\delta})
\]
holds in the field of rational functions in the parameters. 

Analogously to \eqref{eqn:parametrized-state-sum}, the inverted parametrized state sum admits a simple description:
\[
\mathsf{Z}^{\mathrm{inv}}_\beta(\vec{\mathsfgreek{\alpha}}, \vec{\mathsfgreek{\beta}}, \vec{\mathsfgreek{\gamma}}, \mathsfgreek{\delta}) = 
\frac{(1 - \mathsfgreek{\delta}) w_0}{\det (I - \mathcal{B}_{\mathrm{inv}})} 
\in \mathbb{Q}(\vec{\mathsfgreek{\alpha}}, \vec{\mathsfgreek{\beta}}, \vec{\mathsfgreek{\gamma}}, \mathsfgreek{\delta}),
\]
where $w_0$ denotes the monomial weight of the ground state, i.e., the one which assigns $-1$ to all the inverted strips and $0$ to all non-inverted strips, and $\mathcal{B}_{\mathrm{inv}}$ denotes the transition matrix of the random walk (i.e., Markov chain) corresponding to this inverted parametrized state sum. 
It thus remains to show that
\begin{equation}\label{eqn:det-inversion}
\det(I-\mathcal{B}_{\mathrm{inv}}) = (-1)^{\mathrm{col}_-} w_0 \det(I-\mathcal{B}). 
\end{equation}

By the usual expansion of the determinant, 
\[
\det(I - \mathcal{B}) = \sum_{\mathcal{C}}(-1)^{|\mathcal{C}|}\mathrm{wt}(\mathcal{C}),
\quad
\det(I - \mathcal{B}_{\mathrm{inv}}) = \sum_{\mathcal{C}_{\mathrm{inv}}}(-1)^{|\mathcal{C}_{\mathrm{inv}}|}\mathrm{wt}(\mathcal{C}_{\mathrm{inv}}),
\]
where the sum is over simple multi-cycles\footnote{By a simple multi-cycle, we mean a collection of pairwise vertex-disjoint simple directed cycles.} $\mathcal{C}$ (resp., $\mathcal{C}_{\mathrm{inv}}$) in the directed graph associated to the random walk model with transition matrix $\mathcal{B}$ (resp., $\mathcal{B}_{\mathrm{inv}}$), and $\mathrm{wt}$ denotes the product of edge weights. 

Given any simple multi-cycle $\mathcal{C}$ for $\mathcal{B}$, switching the used and unused strips in each inverted column produces a simple multi-cycle $\mathcal{C}_{\mathrm{inv}}$ for $\mathcal{B}_{\mathrm{inv}}$. 
This gives a bijection between simple multi-cycles in the original random walk model and the inverted random walk model.\footnote{This is analogous to the bijection shown in \cite[Fig.~8]{Park_inverted}, although there the state sum is on a knot diagram rather than a knot holder.} 
Therefore, the desired equality \eqref{eqn:det-inversion} would follow if we show
\begin{equation}\label{eqn:weight-comparison}
(-1)^{|\mathcal{C}_{\mathrm{inv}}|}\mathrm{wt}(\mathcal{C}_{\mathrm{inv}}) = (-1)^{\mathrm{col}_-}w_0\,(-1)^{|\mathcal{C}|}\mathrm{wt}(\mathcal{C})
\end{equation}
under this bijection $\mathcal{C} \leftrightarrow \mathcal{C}_{\mathrm{inv}}$ of simple multi-cycles. 

For this, we may invert negative columns one at a time. 
Suppose $\mathcal{C}'$ is obtained from $\mathcal{C}$ by inverting a single column $c$ (i.e., switching the used and unused strips in column $c$). 
Let $r$ denote the number of connected components of $\mathcal{C}$ within the column $c$. 
A direct comparison of the local weights gives
\[
\mathrm{wt}(\mathcal{C}') = (-1)^{r} w_{0,c} \mathrm{wt}(\mathcal{C}),
\]
while 
\[
|\mathcal{C'}| - |\mathcal{C}| \equiv r+1 \mod 2. 
\]
Hence, 
\[
(-1)^{|\mathcal{C}'|} \mathrm{wt}(\mathcal{C}') 
= - w_{0,c} (-1)^{|\mathcal{C}|} \mathrm{wt}(\mathcal{C}).
\]
Applying this relation over inversion of all negative columns gives \eqref{eqn:weight-comparison} and hence \eqref{eqn:det-inversion}. 
The lemma follows. 
\end{proof}

This lemma shows that the inverted state sum, with the extra sign, has the same perturbative series as the MMR expansion.
Moreover, this inverted state sum converges in the two-variable power series ring $(x^{\frac12} - x^{-\frac12}) \mathbb{Z}[q^{\pm 1}][[x]]$, as all local weights have non-negative $x$-degree and only finitely many states contribute to each fixed $x$-degree. 
It immediately follows that it must be equal to the BPS $q$-series, as defined in \cite{Park_inverted}. 
\begin{cor}\label{cor:inverted-state-sum-equals-BPS}
For any braid-homogeneous knot $K$, 
\[
\widehat{Z}_{S^3 \setminus K}(x,q) = (-1)^{1+\mathrm{col}_-} q^{\frac{w -(n-1)}{2}} x^{\frac{w-n}{2}}
\sum_{\sigma \in \mathcal{S}^{\mathrm{inv}}(\mathcal{H}_\beta, \gamma_e)} w(\sigma),
\]
where both sides are elements of $(x^{\frac12} - x^{-\frac12}) \mathbb{Z}[q^{\pm 1}][[x]]$. 
\end{cor}

\textbf{Step 2: Identification with the state sum on $\KH_K$.}
It remains to show that the right-hand side of Corollary \ref{cor:inverted-state-sum-equals-BPS} actually computes the state sum on the knot holder $\KH_K$ from Example \ref{eg:braid-homogeneous-link-knot-holder}, up to some monomial factor. 

The key observation is the following identity which relates the weight in the inverted state sum on $\KH_\beta$ with the weight in the state sum on $\KH_K$: 
\begin{align*}
&
\vcenter{\hbox{
\begin{tikzpicture}[xscale=-1]
\draw (1.75, 1) to[out=112.5, in=-90] (1.5, 2);
\draw (-0.75, 1) to[out=77.5, in=-90] (-0.5, 2);
\draw (1.5, 0) to[out=90, in=-90] (0, 2);
\filldraw[white, opacity=0.7] (1, 0) to[out=90, in=-90] (0, 2) -- (1, 2) to[out=-90, in=90] (0, 0) -- cycle;
\draw (1, 0) to[out=90, in=-90] (0, 2);
\draw[green, ->] (-1, 0) -- (-1, 2);
\draw[green, ->] (0.5, 0) -- (0.5, 2);
\draw[green, ->] (2, 0) -- (2, 2);
\draw[green, ->] (1.75, 1) to[out=150, in=-90] (0.5, 2);
\draw[red, ->] (1.75, 1) to[out=120, in=-90] (0.5, 1.8) -- (0.5, 2);
\draw[red, ->] (-0.75, 1) to[out=60, in=-90] (0.5, 1.8) -- (0.5, 2);
\draw[white, line width=4] (0, 0) to[out=90, in=-90] (1, 2);
\draw[white, line width=4] (-0.5, 0) to[out=90, in=-90] (1, 2);
\filldraw[white, opacity=0.7] (-0.5, 0) to[out=90, in=-90] (1, 2) -- (0, 2) to[out=-90, in=77.5] (-0.75, 1) -- cycle;
\draw[green, ->] (-0.75, 1) to[out=30, in=-90] (0.5, 2);
\draw (-1.5, 0) -- (-0.5, 0);
\draw (-1.5, 2) -- (-0.5, 2);
\draw (-1.5, 0) -- (-1.5, 2);
\draw (0, 0) -- (1, 0);
\draw (0, 2) -- (1, 2);
\draw (1.5, 0) -- (2.5, 0);
\draw (1.5, 2) -- (2.5, 2);
\draw (2.5, 0) -- (2.5, 2);
\draw (0, 0) to[out=90, in=-90] (1, 2);
\draw (1.75, 1) to[out=112.5, in=-90] (1, 2);
\draw (-0.75, 1) to[out=77.5, in=-90] (0, 2);
\draw (-0.5, 0) to[out=90, in=-90] (1, 2);
\node[below] at (-1, 0){$c$};
\node[below] at (0.5, 0){$-1-b$};
\node[below] at (2, 0){$a$};
\node[above] at (-1, 2){$c'$};
\node[above] at (0.5, 2){$-1-b'$};
\node[above] at (2, 2){$a'$};
\draw[thick, ->] (0.5, 0.1) -- (0.5, 0.4);
\draw[thick, ->] (0.5, 1.6) -- (0.5, 1.9);
\draw[thick, ->] (2, 0.1) -- (2, 0.4);
\draw[thick, ->] (2, 1.6) -- (2, 1.9);
\draw[thick, ->] (-1, 0.1) -- (-1, 0.4);
\draw[thick, ->] (-1, 1.6) -- (-1, 1.9);
\end{tikzpicture}
}}
\\
&= 
\delta_{a+(-1-b)+c, a'+(-1-b')+c'}
(-1)^{-1-b}
q^{-\frac{(-1-b)^2+(-1-b')}{2}}
\qbin{-1-b'}{a-a', -1-b, c-c'}_{q^{-1}}
x^{-\frac{(-1-b)+(-1-b')}{2}} \\
&= 
-x
\cdot 
\delta_{a + b' + c, a' + b + c'} (-1)^{b'} q^{-\frac{b'^2 + b}{2}}
\qbin{b}{a-a', b', c-c'}_{q^{-1}}
x^{\frac{b+b'}{2}}
\\
&= - x \cdot q^{\frac{b'-b}{2}}
\vcenter{\hbox{
\begin{tikzpicture}[xscale=-1]
\draw (-1.5, 0) -- (-0.5, 0);
\draw (-1.5, 2) -- (-0.5, 2);
\draw (-1.5, 0) -- (-1.5, 2);
\draw (-0.75, 1) to[out=77.5, in=-90] (-0.5, 2);
\draw (-0.75, 1) to[out=50, in=90] (1, 0);
\draw[green, ->] (-1, 0) -- (-1, 2);
\draw[green, ->] (2, 0) -- (2, 2);
\draw[green] (-0.75, 1) to[out=30, in=90] (0.5, 0);
\draw[red] (-0.75, 1) to[out=40, in=90] (0.5, 0.2) -- (0.5, 0);
\filldraw[white, opacity=0.7] (1, 0) to[out=90, in=-90] (0, 2) -- (1, 2) to[out=-90, in=90] (0, 0) -- cycle;
\draw[green, <-] (0.5, 0) -- (0.5, 2);
\draw (1, 0) to[out=90, in=-90] (0, 2);
\draw[white, line width=4] (0, 0) to[out=90, in=-90] (1, 2);
\draw[white, line width=4] (1.75, 1) to[out=130, in=90] (0, 0);
\draw[red] (1.75, 1) to[out=140, in=90] (0.5, 0.2) -- (0.5, 0);
\filldraw[white, opacity=0.7] (1.75, 1) to[out=130, in=90] (0, 0) -- (1, 0) to[out=90, in=180] (1.25, 0.25) -- cycle;
\draw[green] (1.75, 1) to[out=150, in=90] (0.5, 0);
\draw (0, 0) to[out=90, in=-90] (1, 2);
\draw (1.75, 1) to[out=112.5, in=-90] (1.5, 2);
\draw (1.75, 1) to[out=130, in=90] (0, 0);
\draw (1.5, 0) to[out=90, in=0] (1.25, 0.25) to[out=180, in=90] (1, 0);
\draw (0, 0) -- (1, 0);
\draw (0, 2) -- (1, 2);
\draw (1.5, 0) -- (2.5, 0);
\draw (1.5, 2) -- (2.5, 2);
\draw (2.5, 0) -- (2.5, 2);
\draw (-0.5, 0) to[out=90, in=180] (-0.25, 0.25) to[out=0, in=90] (0, 0);
\node[below] at (-1, 0){$c$};
\node[below] at (0.5, 0){$b$};
\node[below] at (2, 0){$a$};
\node[above] at (-1, 2){$c'$};
\node[above] at (0.5, 2){$b'$};
\node[above] at (2, 2){$a'$};
\draw[thick, <-] (0.5, 0.1) -- (0.5, 0.4);
\draw[thick, <-] (0.5, 1.6) -- (0.5, 1.9);
\draw[thick, ->] (2, 0.1) -- (2, 0.4);
\draw[thick, ->] (2, 1.6) -- (2, 1.9);
\draw[thick, ->] (-1, 0.1) -- (-1, 0.4);
\draw[thick, ->] (-1, 1.6) -- (-1, 1.9);
\end{tikzpicture}
}}.
\end{align*}
Since we are inverting a whole column of strips, the product of the factors $q^{\frac{b'-b}{2}}$ will cancel, but the factor $(-x)$ will survive, one for each negative crossing. 

On top of this, for each inverted column, we need to compare the contributions from the linking between
\begin{enumerate}
\item hyperbolic flow loops (in that column) and elliptic flow loop, 
\item hyperbolic flow loops (in that column) and elliptic framing line, and
\item hyperbolic framing line (in that column) and elliptic flow loop.
\end{enumerate}
Suppose $b$ is the nonnegative integer that a state on $\KH_K$ assigns to a strip at the bottom of a column of negative crossings. 
Then the product of the contributions listed above in the state sum on $\KH_K$ is given by
\[
q^{-2\epsilon b} \cdot q^{b} \cdot q^{-\epsilon}.
\]
On the other hand, the corresponding inverted state on $\KH_\beta$ would assign the negative integer $-1-b$ to that strip, and the product of the contributions listed above in the inverted state sum on $\KH_\beta$ is given by 
\[
q^{2\epsilon (-1-b)} \cdot q^{-(-1-b)} \cdot q^{\epsilon} = 
(q^{-2\epsilon b} \cdot q^{b} \cdot q^{-\epsilon}) \cdot q.
\]
That is, it is $q$ times the one in $\KH_K$, and hence there is an extra factor of $q$ for each inverted column. 

Putting everything together, we conclude, as desired, that
\begin{align*}
\widehat{Z}_{S^3 \setminus K}(x,q) &= 
(-1)^{1+\mathrm{col}_-} q^{\frac{w -(n-1)}{2}} x^{\frac{w-n}{2}}
\sum_{\sigma \in \mathcal{S}^{\mathrm{inv}}(\mathcal{H}_\beta, \gamma_e)} w(\sigma) \\
&= 
(-1)^{1+\mathrm{col}_-} q^{\frac{w -(n-1)}{2}} x^{\frac{w-n}{2}} \cdot (-x)^{\mathrm{cr}_-}q^{\mathrm{col}_-}\;
\Phi_{S^3 \setminus K}(x,q) \\
&= (-1)^{1 + \mathrm{cr}_- + \mathrm{col}_-} q^{\frac{w-(n-1)}{2} + \mathrm{col}_-} x^{\frac{w-n}{2} + \mathrm{cr}_-}\; \Phi_{S^3 \setminus K}(x,q) \\
&= (-1)^{1 + \lambda} q^{g-\lambda} x^{-\frac12 + g} \;\Phi_{S^3 \setminus K}(x,q),
\end{align*}
where $\mathrm{cr}_-$ denotes the total number of negative crossings in $\beta$, and $g$ is the Seifert genus of $K$, and $\lambda$ is the Hopf invariant (a.k.a.~enhanced Milnor invariant \cite{Rudolph_Hopf}) of $K$;\footnote{The appearance of the Hopf invariant $\lambda$ in the exponent of $q$ in the leading term of the BPS $q$-series was observed previously by \cite{OSSS}.} 
the last equality uses the standard formulas for the Seifert genus and the Hopf invariant for braid-homogeneous knots. 
\end{proof}

Note that this extra monomial is consistent with our previous computations for the trefoil and figure-eight knots in Examples \ref{eg:trefoil-state-sum} and \ref{eg:figure-eight-state-sum}.

\appendix

\section{An invariant for fibered cones}\label{sec:fibered_cone}
Here, we record a natural generalization of Theorem~\ref{mainthm:flow-loop-count} to any fibered 3-manifold, as it may be of independent interest. 

Let $Y$ be a 3-manifold, possibly with boundary, and suppose $\pi : Y \rightarrow S^1$ is a fibration, so that $Y$ is a mapping torus of the monodromy of a fiber $F$. 
Choose a generic self-diffeomorphism $\varphi : F \rightarrow F$ representing the monodromy, and by small isotopy of $\varphi$ near the boundary $\partial F$, we assume that the boundary is a sink. 
We are interested in the count of periodic orbits of the suspension flow of $\varphi$. 
Note that, by our assumption that the boundary is a sink, there are no periodic orbits in some neighborhood of the boundary. 
Choose a generic vector field $\xi_{\mathrm{fr}}$ on $Y$, so that the locus where it is parallel to the suspension flow is given by some link $\ell_{\mathrm{fr}}$ in $Y$. 
As remarked in Section~\ref{subsec:skein-valued_flow_loop_count}, the $\mathrm{GL}_1$-skein-valued flow loop count generalizes straightforwardly to this setup, with the same proof.  
That is, the count (as in Definition~\ref{defn:skein-valued-flow-loop-count})
\[
\Phi_{Y, \pi} 
:= \sum_{\gamma \in \mathcal{O}} (-1)^{\mathrm{e}(\gamma) + \mathrm{nh}(\gamma)} [\gamma]
\;\in \widehat{\Sk}^{\mathrm{GL}_1}_q(Y, \ell_{\mathrm{fr}})
,
\]
where, as before, $\mathcal{O}$ is the set of multi-flow loops, and the $\mathrm{GL}_1$-skein module is action-completed by the closed 1-form $\alpha_\pi := \pi^*(d\theta)$, is well-defined and is invariant under deformation of $\varphi$ and $\xi_{\mathrm{fr}}$.\footnote{While this invariant lives in $\widehat{\Sk}^{\mathrm{GL}_1}_q(Y, \ell_{\mathrm{fr}})$, which depends on the framing line $\ell_{\mathrm{fr}}$ given by the locus where the suspension flow and $\xi_{\mathrm{fr}}$ are parallel, hence a priori depends on $\varphi$ and $\xi_{\mathrm{fr}}$, any two choices of $(\varphi, \xi_{\mathrm{fr}})$ can be smoothly deformed from one to another, which induces an isomorphism between the corresponding $\widehat{\Sk}^{\mathrm{GL}_1}_q(Y, \ell_{\mathrm{fr}})$.
That is, $\Phi_{Y,\pi}$ gives a global section over such choices.} 

When $Y$ happens to be the complement $S^3 \setminus K$ of some link $K \subset S^3$, then there is a well-defined notion of linking number, inducing a canonical isomorphism (analogously to Proposition~\ref{prop:GL_1-skein_series_identification}) between $\widehat{\Sk}^{\mathrm{GL}_1}_q(S^3 \setminus K, \ell_{\mathrm{fr}})$ and the completion $R_{\alpha_\pi}$ of $\mathbb{Z}[q^{\pm}][H_1(S^3 \setminus K)] \cong \mathbb{Z}[q^{\pm}][x_1^{\pm 1}, \cdots, x_s^{\pm 1}]$ where we allow infinite sums in the growing $\alpha_\pi$-direction; 
more precisely, the completion is given by
\[
R_\alpha := 
\left\{\sum_{h \in H_1(S^3 \setminus K)} a_h \mathbf{x}^h \;\bigg\vert\; a_h \in \mathbb{Z}[q^{\pm}],\; \#\{h : a_h \neq 0, \alpha(h) \leq N\} < \infty \text{ for every }N \right\}.
\]
In this way, we obtain a multi-variable series invariant for any fibration $\pi : S^3 \setminus K \rightarrow S^1$ of a link complement $S^3 \setminus K$. 
Note that we are not assuming that the fibers are Seifert surfaces of $K$. 
When the fibration comes from that of a fibered link in the usual sense, then this reduces to Remark~\ref{rmk:fibered_link}. 

It turns out that there is an interesting connection between these flow loop counts and fibered cones. 
We first briefly review Thurston's theory of fibered cones. 
Recall that the \emph{Thurston norm} \cite{Thurston} is a seminorm on $H_2(Y, \partial Y; \mathbb{R}) \cong H^1(Y; \mathbb{R})$ defined by
\[
||c||_T := \min_{S} \sum_{i=1}^{n}\chi_{-}(S_i),
\]
where the minimum is taken over all embedded surfaces $S = \sqcup_{i} S_i$ representing $c$, where $S_i$ are connected components, and 
\[
\chi_-(\Sigma) := \max\{0, -\chi(\Sigma)\}.
\]
The unit ball $B_T \subseteq H^1(Y; \mathbb{R})$ with respect to the Thurston norm is called the \emph{Thurston norm ball}, which is a finite-sided rational convex polyhedron
 \cite{Thurston}. 
This is a genuine compact polytope when $Y$ is hyperbolic. 

For a top-dimensional face $f$ of $B_T$, let $C(f)$ denote the open cone from the origin through the relative interior of $f$. 
We say that an integral class $a \in H^1(Y;\mathbb{Z})$ is \emph{fibered} if it can be represented by a fibration $\pi : Y \rightarrow S^1$, $a = \pi^*[d\theta]$. 
Thurston's fibered face theorem says that there is a distinguished set of faces of $B_T$ which determine all fibrations of $Y$ over $S^1$: 
\begin{thm}[{\cite[Thm.~5]{Thurston}}]
If one integral class in $C(f)$ is fibered, then every integral class in $C(f)$ is fibered. 
\end{thm}
Such a face $f$ is called a \emph{fibered face} and the associated cone $C(f)$ a \emph{fibered cone}. 
In \cite[Thm.~5]{Thurston}, Thurston also shows that the union of fibered cones is exactly the set of cohomology classes in $H^1(Y;\mathbb{R})$ representable by nowhere-vanishing closed 1-forms. 

Fibered cones can be covered by open sets each of which can be represented by a single flow, as we review below. 
For any nowhere-vanishing vector field $v$ on $Y$ tangent to $\partial Y$, set 
\[
C_v := \{a \in H^1(Y;\mathbb{R}) \;\vert\; \exists \alpha \text{ with } d\alpha = 0,\; \alpha(v) > 0 \text{ everywhere, and } a=[\alpha]\}.
\]
Each $C_v$ is an open convex cone, so $U_v := C_v \cap C(f)$ is also open. 
For any nowhere-vanishing closed 1-form $\alpha$ representing a class $[\alpha] \in C(f)$, there is a vector field such that $\alpha(v) > 0$ everywhere, so the sets $U_v$ form an open cover of the fibered cone:
\[
C(f) = \bigcup_{v} U_v.
\]
Moreover, for any integral class $[\alpha] \in C_v \cap H^1(Y;\mathbb{Z})$ represented by a closed 1-form $\alpha$ with $\alpha(v) > 0$ and integral periods, we get a fibration $\pi : Y \rightarrow S^1$ with $\alpha = \pi^*(d\theta)$ by integrating $\alpha$ starting from some base point in $Y$. 

As a result, the flow loop counts in the same fibered cone are all the same: 
\begin{cor}
Let $Y$ be a 3-manifold, possibly with boundary, and let $\pi : Y \rightarrow S^1$ be a fibration representing some fibered class $[\alpha] = \pi^*[d\theta] \in H^1(Y;\mathbb{Z})$. 
Then, the $\mathrm{GL}_1$-skein-valued flow loop count $\Phi_{Y,\pi}$ discussed above depends only on the fibered cone $C(f)$ containing $[\alpha]$, not on the precise fibered class $[\alpha]$ in $C(f)$. 
That is, we get an invariant $\Phi_{Y,C(f)}$ for each fibered cone $C(f)$. 
\end{cor}
\begin{proof}
If two fibered classes are contained in the same open set $U_v$, then they can be represented by the same flow $\phi_v$ generated by $v$, so their flow loop count is the same. 
Now suppose $U_v \cap U_w \neq \emptyset$. 
Since $U_v \cap U_w$ is an open cone, it contains a rational class and hence, after rescaling, an integral fibered class $[\alpha]$. 
Since we can compute the flow loop count associated to $[\alpha]$ using either $v$ or $w$, the flow loop counts on $U_v$ and $U_w$ agree. 
Since the sets $U_v$ form an open cover of the connected cone $C(f)$, any two fibered classes in $C(f)$ can be joined by a chain of pairwise overlapping sets, and it follows that the flow loop count is independent of the precise fibered class in $C(f)$. 
\end{proof}
When $Y$ is the complement $S^3 \setminus K$ of a hyperbolic link $K$, we thus obtain an invariant $\Phi_{S^3 \setminus K, C(f)}$ of a fibered cone $C(f)$, which lives in $R_{C(f)} := \cap_{\alpha \in C(f)} R_\alpha$. 

A natural question is whether the same invariants can be obtained from quantum group invariants, similarly to Conjecture~\ref{conj:flow-equals-quantum}. 
Indeed, in recent work \cite{PassaroSanMartinSuarez}, Passaro and San Martin Suarez observed that, if a link $K$ admits a certain ``$\alpha$-nice'' diagram, then one can define a quantum invariant of $K$ living in $R_\alpha$. 
They further conjecture that their invariant is actually an invariant of the fibered cone containing $\alpha$. 
Thus, the following is a natural extension of Conjecture~\ref{conj:flow-equals-quantum}:
\begin{conj}
For any link $K$ and its fibered cone $C(f)$, the dynamical invariant $\Phi_{S^3 \setminus K, C(f)} \in R_{C(f)}$ agrees with the quantum invariant studied by \cite{PassaroSanMartinSuarez}.
\end{conj}

\bibliography{ref}
\bibliographystyle{alpha}

\end{document}